\documentclass[12pt]{amsart}
\usepackage[T1]{fontenc}
\usepackage[utf8]{inputenc}
\usepackage{stix}
\usepackage[a4paper,hmargin=2cm,top=2.4cm,bottom=2.2cm]{geometry}
\usepackage{xcolor}
\usepackage{graphicx}
\usepackage[numbers,sort&compress]{natbib}
\usepackage{tikz}
\usepackage{capt-of}
\usepackage[colorlinks=true,linkcolor=blue!50!black,citecolor=blue!50!black,urlcolor=blue!50!black]{hyperref}
\hypersetup{hypertexnames=false}
\newtheorem{theorem}{Theorem}[section]
\newtheorem{definition}{Definition}[section]

\newtheorem{proposition}{Proposition}[section]
\newtheorem{remark}{Remark}

\newtheorem{corollary}{Corollary}

\newtheorem{lemma}{Lemma}

\newcommand{\beq}{\begin{equation}}
	\newcommand{\eeq}{\end{equation}}
\newcommand{\bea}{\begin{eqnarray}}
	\newcommand{\eea}{\end{eqnarray}}
\definecolor{mygray}{gray}{0.3}

\newcommand{\bes}{\begin{eqnarray}}
	\newcommand{\ees}{\end{eqnarray}}
\def\nn{{\nonumber}}

\newcommand\restr[2]{{
		\left.\kern-\nulldelimiterspace 
		#1 
		\vphantom{\big|} 
		\right|_{#2} 
}}

\makeatletter
\@ifundefined{mathscr}{\RequirePackage{mathrsfs}}{}
\makeatother

\newcommand{\sinR}{\widehat{\sin}_{\mathcal{R}(p,q)}}
\newcommand{\cosR}{\widehat{\cos}_{\mathcal{R}(p,q)}}

\newcommand{\pn}[1]{[#1]_{\mathcal{R}(p,q)}}
\newcommand{\pf}[1]{[#1]!_{\mathcal{R}(p,q)}}
\newcommand{\Hpos}{\textnormal{(H)}}
\newcommand{\Hphi}{\textnormal{(H}_{\phi}\textnormal{)}}
\newcommand{\HG}{\textnormal{(G)}}
\newcommand{\nur}[1]{\nu_{\mathcal{R}(p,q),#1}}
\newcommand{\ER}[1]{\mathcal{E}_{\mathcal{R}(p,q)}(#1)}
\newcommand{\Bor}{\mathfrak{L}_{\mathcal{R}(p,q)}}
\newcommand{\eR}{\widehat{e}_{\mathcal{R}(p,q)}}
\newcommand{\eRw}{e_{\mathcal{R}(p,q)}}
\newcommand{\ERw}{E_{\mathcal{R}(p,q)}}
\newcommand{\cosRw}{\cos_{\mathcal{R}(p,q)}}
\newcommand{\sinRw}{\sin_{\mathcal{R}(p,q)}}
\newcommand{\CosRw}{\mathrm{Cos}_{\mathcal{R}(p,q)}}
\newcommand{\SinRw}{\mathrm{Sin}_{\mathcal{R}(p,q)}}
\newcommand{\fq}{\mathfrak{q}}

\newcommand{\Sth}{\mathbb{S}^{\mathcal{R}(p,q)}_{\theta}}
\newcommand{\Om}{\Omega^{\mathcal{R}(p,q)}_{\theta}}
\newtheorem{example}{Example}

\begin{document}

\title[Deformed seminorms, anisotropic sectors and Phragm\'en Lindel\"of principles]{Deformed seminorms, anisotropic sectors and Phragm\'en Lindel\"of principles for $\mathcal{R}(p,q)-$\allowbreak difference operators}

\author{Kawèyim Lankpetre}
\address{Laboratoire de Physique des Matériaux et des Composants à Semi-conducteurs (LPMCS), Département de Physique, Faculté des Sciences, Université de Lomé, 01 B.P. 1515, Lomé 01, Togo}
\email{lankpetre.kaweyim.myhal@gmail.com}

\author{Isiaka Aremua}
\address{Laboratoire de Physique des Matériaux et des Composants à Semi-conducteurs (LPMCS), Département de Physique, Faculté des Sciences, Université de Lomé, 01 B.P. 1515, Lomé 01, Togo}
\address{International Chair of Mathematical Physics and Applications (ICMPA-UNESCO Chair), University of Abomey-Calavi, 072 B.P. 50, Cotonou, Republic of Benin}
\email{claudisak@gmail.com}

\author{Joseph Désiré Bukweli Kyemba}
\address{International Chair of Mathematical Physics and Applications (ICMPA-UNESCO Chair), University of Abomey-Calavi, 072 B.P. 50, Cotonou, Republic of Benin}
\address{Département de Mathématiques, Statistique et Informatique, Faculté des Sciences et Technologie, Université de Kinshasa, Kinshasa, Democratic Republic of the Congo}
\email{desire.bukweli@unikin.ac.cd}


\begin{abstract}
We study the deformed derivative of the $\mathcal{R}(p,q)-$deformed quantum algebras of Hounkonnou and Bukweli Kyemba on spaces of holomorphic functions, with the induced geometry and growth principles. The deformed derivative maps the holomorphic functions on a disc of radius $\rho$ into those on the disc of radius $\rho/\beta$, $\beta$ being the upper limit of the $n-$th roots of the deformed numbers, and this loss of radius is optimal. For two root structure functions the deformed gauge equals $\max\{|z|,1/\phi_{1}\}$, so that deformed discs are Euclidean above a minimal radius; Borel Carath\'eodory and vector valued Cauchy Hadamard theorems follow. On factorial weighted Banach spaces the deformed derivative has norm $1/r$ and spectrum the closed disc of radius $1/r$. A deformed Borel transform identifies these spaces with weighted Wiener algebras, yields the Gelfand theory, and shows that the associated Fr\'echet space is nuclear. We establish logarithmic convexity and subharmonicity results, Phragm\'en Lindel\"of and Lindel\"of principles in the anisotropic sectors of half aperture $\theta\sup_{k}k^{-1}\log[k]$ and in their curved variants, with sharp critical orders, and a rigidity theorem: under Gaussian growth of the deformed factorials, functions of finite deformed exponential type are entire of order zero and cannot be bounded on a ray unless constant. A dominant term principle shows that, for a large dominant root, the deformed exponential and trigonometric functions have only real simple zeros, localized between explicit geometric means of deformed numbers, whereas the deformed exponential series has nonreal zeros when the two roots have sum at most two.
\end{abstract}

\maketitle
\enlargethispage{4pt}

\section{Introduction}\label{sec:intro}


Deformed calculi have become central tools in the analysis of nonlocal structures arising in mathematical physics, in the theory of quantum groups and in the modelling of complex systems \cite{JaganSridh, JagannathanSrinivasa, Jimbo, Drinfeld, ChaPres}. Since the work of Jackson on $q-$difference operators \cite{Jackson1908} and the systematic development of basic hypergeometric series and quantum calculus \cite{GasperRahman, KacCheung2002}, several two parameter deformations have been introduced, among which the $(p,q)-$calculus of Jagannathan and Srinivasa Rao \cite{JagannathanSrinivasa}, the Chakrabarti Jagannathan oscillator \cite{ChakrabartiJagannathan}, the $P,Q-$differentiation of Burban and Klimyk \cite{BurbanKlimyk} and the multiparameter deformations of Hounkonnou and Ngompe Nkouankam \cite{HounkonnouNgompe, HmNe2}. The fundamental theorem of $(p,q)-$calculus and the associated Taylor formulas were established in \cite{ref58}, and fractional $(p,q)-$operators built on the Jackson type series representation were introduced in \cite{SoontharanonSitthiwirattham2020}.

A unifying framework was obtained when general deformations driven by a meromorphic structure function $\mathcal{R}(u,v)$ were introduced. Building on the coherent state maps of the disc of \cite{Odzijewicz1998}, the $\mathcal{R}(p,q)-$deformed Heisenberg algebras, their coherent states and the associated special functions were constructed in \cite{HounkBuk}, and the corresponding differential and integral calculi, with the compatible Hopf algebra structure and the list of relevant particular cases, were developed in \cite{Mahouton_Kyemba_Rpq_calculs, BukweliHounkonnou2012ADJM, BukweliHounkonnou2012RS}. The same mechanism produced deformed combinatorics \cite{Mahouton_Melong_Rpq_combinatoires}, deformed probability distributions \cite{Mahouton_Melong_Rpq_analogues_distribution, HounkonnouMelong2023}, deformed conformal and super Virasoro algebras \cite{Mahouton_Melong_Rpq_conformal_algebra, Melong2023superVirasoro}, deformed Witt and Virasoro $n-$algebras with Korteweg de Vries equations \cite{Mahouton_Melong_nalgebra_virasoro} and deformed Heisenberg Virasoro algebras with matrix models \cite{MelongWulkenhaar2023}.

The analytic and topological side of the theory rests on classical pillars of complex and functional analysis. The convexity of the logarithm of the maximum modulus goes back to Hadamard \cite{Hadamard1896}, the sharpening of the maximum principle on unbounded regions to Phragm\'en and Lindel\"of \cite{PhragmenLindelof1908}, and growth principles, indicators and entire functions of finite order are treated in \cite{Titchmarsh1939, Boas1954, Levin1980, Koosis1988}. Their potential theoretic formulation through subharmonic functions is developed in \cite{Ransford1995, HaymanKennedy1976, Hormander1994, Rudin1987}. Nuclear spaces \cite{Grothendieck1955, Pietsch1972, MeiseVogt1997}, Fr\'echet and locally convex topologies \cite{Treves1967} and countably normed spaces \cite{GelfandShilov1968} provide the language for weighted spaces of holomorphic functions, and the Gelfand theory of commutative Banach algebras \cite{Kaniuth2009} together with the subharmonicity of the spectral radius \cite{Vesentini1968, Aupetit1991} provides the spectral counterpart. A first approach to the topology of $\mathcal{R}(p,q)-$deformed function spaces, including deformed discs, anisotropic sectors and nuclear spaces, was proposed in \cite{LankpetreAremuaBukweli2026}.

Beyond their intrinsic interest, deformed difference operators generate nonstandard kinetic terms in evolution equations whose eigenfunctions are deformed exponential series and deformed trigonometric functions, adapted to the modelling of heterogeneous media and of multiscale nonlocal systems \cite{Herrmann, Podlubny}; the analytic setting in which these operators act has to be specified before any spectral or evolution problem can be posed.

\subsection*{Problem and contributions}

The deformed derivative $\partial_{\mathcal{R}(p,q)}$ acts diagonally on monomials with multiplier given by the deformed numbers $\pn{n}=\mathcal{R}(p^{n},q^{n})$, whose growth may be subexponential, exponential or, after taking factorials, of Gaussian type. Consequently the topology of uniform convergence on compact subsets is not automatically adapted to the deformation, and the following questions arise: on which spaces of holomorphic functions is $\partial_{\mathcal{R}(p,q)}$ continuous, and with which loss of radius; which weighted spaces make it a bounded operator with computable norm and spectrum; which algebraic structure, Gelfand spectrum and functional calculus are attached to these spaces; which geometry is induced by the deformed numbers on the complex plane; and which growth principles of Phragm\'en Lindel\"of type hold in this geometry. The answers must be formulated so that every statement is well defined and every constant is explicit.

The main results of the paper are the following.
\begin{enumerate}
\item[(a)] Writing $\beta_{\mathcal{R}(p,q)}=\limsup_{n}\pn{n}^{1/n}$, the deformed derivative maps $\mathcal{O}(D_{\rho})$ continuously into $\mathcal{O}(D_{\rho/\beta_{\mathcal{R}(p,q)}})$, the radius $\rho/\beta_{\mathcal{R}(p,q)}$ being optimal; the dilation difference $P-Q$ is invertible on functions vanishing at the origin, with an inverse mapping $\mathcal{O}_{0}(D_{\rho})$ onto $\mathcal{O}_{0}(D_{p\rho})$ and no better; explicit bounds follow under exponential growth hypotheses (Theorems \ref{thm:continuity}, \ref{thm:PQ}, Corollary \ref{cor:G1} and Example \ref{ex:u-v}).
\item[(b)] The deformed gauge $\|z\|_{\mathcal{R}(p,q)}=\sup_{n}(|z|/\pn{n})^{1/n}$ satisfies $\|z\|_{\mathcal{R}(p,q)}\geq1/\alpha_{\mathcal{R}(p,q)}$ for $z\neq0$, with $\alpha_{\mathcal{R}(p,q)}=\liminf_{n}\pn{n}^{1/n}$; for two root structure functions $\|z\|_{\mathcal{R}(p,q)}=\max\{|z|,1/\phi_{1}\}$, so that deformed discs are Euclidean discs of radius at least $1/\phi_{1}$ (Proposition \ref{prop:gauge} and Theorem \ref{thm:gauge-Hphi}); a generalized Borel Carath\'eodory lemma, a vector valued deformed Cauchy Hadamard theorem, a deformed Taylor formula with identity, maximum modulus and open mapping properties, and explicit Cauchy estimates for the weighted $\mathcal{R}(p,q)-$norms follow (Theorems \ref{thm:BC}, \ref{thm:CH}, Proposition \ref{prop:R-taylor}, Corollaries \ref{cor:identity}, \ref{cor:max-modulus} and Proposition \ref{prop:cauchy-R}).
\item[(c)] On the factorial weighted Banach spaces $\ER{r}$, the deformed derivative has norm exactly $1/r$, spectrum the closed disc of radius $1/r$, point spectrum the open disc, with eigenfunctions the deformed exponential series (Theorem \ref{thm:spectrum}); the deformed Borel transform $\Bor$ is an isometric isomorphism onto a weighted Wiener algebra, which proves that the Fr\'echet space $\ER{\rho^{-}}$ is nuclear and isomorphic to $\mathcal{O}(D_{\rho})$ (Theorem \ref{thm:borel}).
\item[(d)] The deformed binomial convolution turns $\ER{r}$ into a commutative unital Banach algebra whose Gelfand spectrum is the closed disc of radius $r$, whose Gelfand transform is $\Bor$, and in which a Wiener type invertibility criterion and a holomorphic functional calculus hold (Theorem \ref{thm:gelfand}).
\item[(e)] With $\rho_{\mathcal{R}(p,q)}(k)=k^{-1}\log\pn{k}$, the anisotropic sector $\Sth=\bigcup_{k}\{|\arg z|<\theta\rho_{\mathcal{R}(p,q)}(k)\}$ is the sector of half aperture $\theta\sup_{k}\rho_{\mathcal{R}(p,q)}(k)$, and the curved sector $\{|\arg z|<\theta\rho_{\mathcal{R}(p,q)}(|z|)\}$ has asymptotic half aperture $\theta\log\phi_{1}$ under $\Hphi$; Phragm\'en Lindel\"of principles hold in both, for vector valued, Gelfand algebra valued and spectral radius majorants, with sharp critical order $\pi/(2\theta\sup_{k}\rho_{\mathcal{R}(p,q)}(k))$, together with a Lindel\"of type principle, a three lines principle and a principle in nuclear Fr\'echet spaces of $\mathcal{R}(p,q)-$analytic functions (Theorems \ref{thm:PL-gelfand}, \ref{thm:PL-aniso}, \ref{thm:lindelof}, \ref{thm:PL-nuclear} and Proposition \ref{prop:sharp}); logarithmic convexity and subharmonicity results are proved in Section \ref{sec:convexity}.
\item[(f)] Under Gaussian growth of the deformed factorials, every function of finite deformed exponential type is entire of order zero, has at most $O(\log\varrho)$ zeros in the disc of radius $\varrho$, and cannot be bounded on a ray unless it is constant; the hypothesis cannot be removed (Theorems \ref{thm:rigidity}, \ref{thm:zeros} and Remark \ref{rem:counterexample}).
\item[(g)] The deformed derivative, integral operator and translations are described through the Borel transform, which admits an inversion formula and, in the Jackson regime, a Laplace type representation; deformed linear equations with constant coefficients are solved by a Fredholm argument (Theorems \ref{thm:integral-operator}, \ref{thm:duality}, \ref{thm:borel-integral}, \ref{thm:constant-coefficients} and \ref{thm:translations}).
\item[(h)] In the deformed Fock space, $\partial_{\mathcal{R}(p,q)}$ and the multiplication by $z$ are mutually adjoint and realize $\mathcal{A}_{\mathcal{R}(p,q)}$, the deformed exponential series being the reproducing kernel and the coherent states (Theorems \ref{thm:fock}, \ref{thm:fock-representation} and \ref{thm:coherent}).
\item[(i)] If $\phi_{1}\geq Q_{*}=4.8105\ldots$, the deformed exponential series $\eR(z)=\sum_{n}z^{n}/\pf{n}$ has only simple negative zeros $-\xi_{k}$ with $(\pn{k-1}\pn{k})^{1/2}<\xi_{k}<(\pn{k}\pn{k+1})^{1/2}$, $k\geq2$, and if $\phi_{1}\geq Q_{*}^{1/4}$ its cosine and sine have real simple zeros localized in the same way; the zeros of $\eR$ are also real if $\phi_{1}\geq4$, $\phi_{2}=1$ or $\phi_{2}=\phi_{1}^{1/2}$, and some are nonreal if $\phi_{1}+\phi_{2}\leq2$. A rescaling of the deformed numbers transfers these results to the deformed exponentials and trigonometric functions of first and second kind, built with the weights $\phi_{1}^{n(n-1)/2}$ and $\phi_{2}^{n(n-1)/2}$: the second kind exponential has explicit simple negative zeros, and the second kind cosine and sine have real simple zeros as soon as $\phi_{1}\geq\sqrt2\,\phi_{2}$ (Theorems \ref{thm:localization}, \ref{thm:exp-zeros}, \ref{thm:trig-zeros}, Lemma \ref{lem:rescaling} and Corollaries \ref{cor:E-zeros}, \ref{cor:Cos-zeros}).
\end{enumerate}

\subsection*{Organization of the paper}

Section \ref{sec:prelim} introduces the growth indices of the deformed numbers. Section \ref{sec:continuity} develops the topological analysis: continuity of the operators of the quantum algebra and invertibility of $P-Q$, explicit growth hypotheses, the deformed gauge and a generalized Borel Carath\'eodory lemma, topological $\mathcal{R}(p,q)-$Gelfand algebras, anisotropic $\mathcal{R}(p,q)-$sectors and Phragm\'en Lindel\"of principles, Fr\'echet nuclear spaces and the holomorphic case, a generalized Cauchy Hadamard theorem, and weighted norms with Cauchy estimates. Section \ref{sec:scale} is devoted to the factorial weighted scale, the deformed Borel transform and the spectral theory of the deformed derivative, Section \ref{sec:algebra} to the Banach and Fr\'echet algebra structures and Section \ref{sec:fock} to the deformed Fock space. Section \ref{sec:convexity} develops the convexity and subharmonicity tools and Section \ref{sec:PL} the consequences of Gaussian growth, Section \ref{sec:real-zeros} the real zeros of the deformed exponential and trigonometric functions, and Section \ref{sec:examples} the standard deformations with numerical illustrations; Section \ref{sec:conclusion} concludes.

\section{Mathematical preliminaries}\label{sec:prelim}

\subsection{$\mathcal{R}(p,q)-$deformed quantum algebras}\label{subsec:HB}

Let $p$ and $q$ be two positive real numbers. Following \cite{HounkBuk}, let $\mathcal{R}(u,v)$ be a meromorphic function on $\mathbb{C}\times\mathbb{C}$,
\begin{eqnarray}\label{r10}
	\mathcal{R}(u,v)&=&\sum_{s,t=-l}^{\infty}r_{st}u^{s}v^{t},
\end{eqnarray}
with an eventual isolated singularity at the origin, $r_{st}\in\mathbb{C}$, $l\in\mathbb{N}\cup\{0\}$, $\mathcal{R}(1,1)=0$, convergent on a bidisc $\mathbb{D}_{R}=\{(w_{1},w_{2})\in\mathbb{C}^{2}:|w_{j}|<R_{j}\}$. The deformed numbers, factorials and binomial coefficients are
\begin{eqnarray}\label{eq:numbers}
	\pn{n}&:=&\mathcal{R}(p^{n},q^{n}),\qquad \pf{n}\;:=\;\prod_{j=1}^{n}\pn{j},\quad \pf{0}:=1,\nn\\
	&&\bigg[\begin{array}{c} n \\ k\end{array}\bigg]_{\mathcal{R}(p,q)}\;:=\;\frac{\pf{n}}{\pf{k}\,\pf{n-k}} .
\end{eqnarray}
The dilation operators $P\varPsi(z)=\varPsi(pz)$, $Q\varPsi(z)=\varPsi(qz)$ and the deformed derivative
\begin{eqnarray}\label{r5}
	\partial_{\mathcal{R}(p,q)}&:=&\partial_{p,q}\,\frac{p-q}{P-Q}\,\mathcal{R}(P,Q)
\end{eqnarray}
act on monomials by $Pz^{n}=p^{n}z^{n}$, $Qz^{n}=q^{n}z^{n}$, $\mathcal{R}(P,Q)z^{n}=\pn{n}z^{n}$ and
\begin{eqnarray}\label{eq:diag}
	\partial_{\mathcal{R}(p,q)}z^{n}&=&\pn{n}\,z^{n-1},\quad n\geq1,\qquad \partial_{\mathcal{R}(p,q)}1\;=\;0,
\end{eqnarray}
where $\partial_{p,q}z^{n}=[n]_{p,q}z^{n-1}$ with $[n]_{p,q}=(p^{n}-q^{n})/(p-q)$ \cite{JagannathanSrinivasa, HounkBuk, Mahouton_Kyemba_Rpq_calculs}. The quantum algebra $\mathcal{A}_{\mathcal{R}(p,q)}$ generated by $\{1,A,A^{\dagger},N\}$ with
\begin{eqnarray}\label{algN1}
	AA^{\dagger}\;=\;[N+1]_{\mathcal{R}(p,q)},\qquad A^{\dagger}A\;=\;[N]_{\mathcal{R}(p,q)},\qquad [N,A]\;=\;-A,\qquad [N,A^{\dagger}]\;=\;A^{\dagger},
\end{eqnarray}
is realized on holomorphic functions by $A^{\dagger}=z$, $A=\partial_{\mathcal{R}(p,q)}$ and $N=z\partial_{z}$.

\subsection{Standing hypotheses and growth indices}\label{subsec:indices}

Throughout the paper we assume
\begin{eqnarray}\label{eq:H}
	\Hpos\qquad \pn{n}>0\ \ \mbox{for all } n\geq1,\qquad \beta_{\mathcal{R}(p,q)}\;:=\;\limsup_{n\to\infty}\pn{n}^{1/n}<\infty .
\end{eqnarray}
We also introduce
\begin{eqnarray}\label{eq:indices}
	\alpha_{\mathcal{R}(p,q)}&:=&\liminf_{n\to\infty}\pn{n}^{1/n},\qquad
	\kappa_{\mathcal{R}(p,q)}\;:=\;\liminf_{n\to\infty}\pf{n}^{1/n},\nn\\
	K_{\mathcal{R}(p,q)}&:=&\limsup_{n\to\infty}\pf{n}^{1/n},\qquad
	\gamma_{\mathcal{R}(p,q)}\;:=\;\liminf_{n\to\infty}\frac{\log\pf{n}}{n^{2}},
\end{eqnarray}
with values in $[0,\infty]$, $[0,\infty]$, $[0,\infty]$ and $[-\infty,\infty]$, and we say that the deformation has Gaussian factorial growth if
\begin{eqnarray}\label{eq:G}
	\HG\qquad \gamma_{\mathcal{R}(p,q)}\;>\;0 .
\end{eqnarray}
A particularly important class is formed by the two root structure functions considered in \cite{Mahouton_Melong_Rpq_combinatoires, Mahouton_Melong_Rpq_analogues_distribution}: we say that $\mathcal{R}(p,q)$ satisfies $\Hphi$ if there exist $c>0$ and $\phi_{1}>\phi_{2}>0$, depending on $(p,q)$, such that
\begin{eqnarray}\label{eq:Hphi}
	\pn{n}&=&c\,\frac{\phi_{1}^{n}-\phi_{2}^{n}}{\phi_{1}-\phi_{2}},\qquad n\geq0,\qquad\mbox{and we put}\qquad \fq\;:=\;\frac{\phi_{2}}{\phi_{1}}\in(0,1).
\end{eqnarray}
The Jagannathan Srinivasa Rao deformation $\mathcal{R}(u,v)=(u-v)/(p-q)$, $0<q<p$, has $(\phi_{1},\phi_{2},c)=(p,q,1)$; the Jackson case $p=1$ has $(1,q,1)$; the Chakrabarti Jagannathan deformation $\mathcal{R}(u,v)=(1-uv)/((p^{-1}-q)u)$, $pq<1$, has $(p^{-1},q,1)$; the generalized $q-$Quesne deformation $\mathcal{R}(u,v)=(uv-1)/((q-p^{-1})v)$, $pq<1$, has $(q^{-1},p,p/q)$; and the Hounkonnou Ngompe deformation $\mathcal{R}(u,v)=h\,v^{\nu}u^{-\mu}(uv-1)/((q-p^{-1})v)$, $pq<1$, has $(p^{-\mu}q^{\nu-1},p^{1-\mu}q^{\nu},h\,p^{1-\mu}q^{\nu-1})$, because $\pn{n}=h\,(q^{\nu}p^{-\mu})^{n}(p^{n}-q^{-n})/(q-p^{-1})$ \cite{HounkBuk, Mahouton_Kyemba_Rpq_calculs, HounkonnouNgompe}. As the structure function itself, the deformation functions $\phi_{1}=\phi_{1}(p,q)$ and $\phi_{2}=\phi_{2}(p,q)$, and the gauge $c=c(p,q)$, are meromorphic functions of the deformation parameters, here monomials in $p$ and $q$; the hypothesis $\Hphi$ concerns their values at the pair $(p,q)$ under consideration.

Under $\Hphi$, the deformed exponentials of first and second kind and the associated deformed trigonometric functions are \cite{JagannathanSrinivasa, Mahouton_Kyemba_Rpq_calculs}
\begin{eqnarray}\label{eq:def-exp}
	\eRw(z)&:=&\sum_{n\geq0}\frac{\phi_{1}^{n(n-1)/2}}{\pf{n}}\,z^{n},\qquad
	\ERw(z)\;:=\;\sum_{n\geq0}\frac{\phi_{2}^{n(n-1)/2}}{\pf{n}}\,z^{n},\nn\\
	\cosRw(z)&:=&\sum_{n\geq0}\frac{(-1)^{n}\phi_{1}^{n(2n-1)}}{\pf{2n}}\,z^{2n},\qquad
	\sinRw(z)\;:=\;\sum_{n\geq0}\frac{(-1)^{n}\phi_{1}^{n(2n+1)}}{\pf{2n+1}}\,z^{2n+1},
\end{eqnarray}
so that $\cosRw(z)=\frac12(\eRw(iz)+\eRw(-iz))$ and $\sinRw(z)=\frac{1}{2i}(\eRw(iz)-\eRw(-iz))$, the functions $\CosRw$ and $\SinRw$ of second kind being defined in the same way from $\ERw$, with $\phi_{2}$ in place of $\phi_{1}$. The weights $\phi_{1}^{n(n-1)/2}$ and $\phi_{2}^{n(n-1)/2}$ affect the cosines and sines through $\phi_{j}^{n(2n-1)}$ and $\phi_{j}^{n(2n+1)}$, the exponents of the even and odd terms. The spectral theory of $\partial_{\mathcal{R}(p,q)}$ involves instead the deformed exponential series $\eR(z):=\sum_{n\geq0}z^{n}/\pf{n}$, which satisfies $\partial_{\mathcal{R}(p,q)}\eR(\lambda z)=\lambda\,\eR(\lambda z)$, and the associated series $\cosR$ and $\sinR$ of \eqref{eq:trig}; the precise relation between the two families is Lemma \ref{lem:rescaling}.

\begin{proposition}\label{prop:indices}
Assume \Hpos. Then:
\begin{enumerate}
\item[\textnormal{(i)}] $\log\kappa_{\mathcal{R}(p,q)}\leq\log K_{\mathcal{R}(p,q)}$, and $\limsup_{n}\frac{\log\pf{n}}{n^{2}}\leq\frac{1}{2}\log\beta_{\mathcal{R}(p,q)}$, $\gamma_{\mathcal{R}(p,q)}\geq\frac{1}{2}\log\alpha_{\mathcal{R}(p,q)}$ whenever $\alpha_{\mathcal{R}(p,q)}>0$;
\item[\textnormal{(ii)}] if $\beta_{\mathcal{R}(p,q)}<1$ then $K_{\mathcal{R}(p,q)}=0$; if $\alpha_{\mathcal{R}(p,q)}>1$ then $\kappa_{\mathcal{R}(p,q)}=\infty$ and \HG\ holds;
\item[\textnormal{(iii)}] if $\lim_{n}\pn{n}=L\in(0,\infty)$, then $\kappa_{\mathcal{R}(p,q)}=K_{\mathcal{R}(p,q)}=L$ and $\gamma_{\mathcal{R}(p,q)}=0$;
\item[\textnormal{(iv)}] under $\Hphi$, $\alpha_{\mathcal{R}(p,q)}=\beta_{\mathcal{R}(p,q)}=\phi_{1}$, $\gamma_{\mathcal{R}(p,q)}=\frac12\log\phi_{1}$, and $\kappa_{\mathcal{R}(p,q)}=K_{\mathcal{R}(p,q)}$ equals $0$, $c/(1-\phi_{2})$ or $\infty$ according as $\phi_{1}<1$, $\phi_{1}=1$ or $\phi_{1}>1$.
\end{enumerate}
\end{proposition}
{\it Proof.}
(i) The first inequality is obvious. Let $\varepsilon>0$. There is $C_{\varepsilon}$ with $\log\pn{k}\leq k(\log\beta_{\mathcal{R}(p,q)}+\varepsilon)+C_{\varepsilon}$ for all $k\geq1$ (if $\beta_{\mathcal{R}(p,q)}=0$ replace $\log\beta_{\mathcal{R}(p,q)}$ by an arbitrary negative number), hence $\log\pf{n}=\sum_{k=1}^{n}\log\pn{k}\leq\frac{n(n+1)}{2}(\log\beta_{\mathcal{R}(p,q)}+\varepsilon)+nC_{\varepsilon}$, which gives the second inequality after division by $n^{2}$. The third one is proved in the same way with the lower bound $\log\pn{k}\geq k(\log\alpha_{\mathcal{R}(p,q)}-\varepsilon)-C'_{\varepsilon}$. (ii) If $\beta_{\mathcal{R}(p,q)}<1$, the bound of (i) gives $\frac{1}{n}\log\pf{n}\leq\frac{n+1}{2}(\log\beta_{\mathcal{R}(p,q)}+\varepsilon)+C_{\varepsilon}\to-\infty$ for $\varepsilon$ small. If $\alpha_{\mathcal{R}(p,q)}>1$, $\frac{1}{n}\log\pf{n}\geq\frac{n+1}{2}(\log\alpha_{\mathcal{R}(p,q)}-\varepsilon)-C'_{\varepsilon}\to+\infty$ and $\gamma_{\mathcal{R}(p,q)}\geq\frac12\log\alpha_{\mathcal{R}(p,q)}>0$. (iii) By Ces\`aro, $\frac{1}{n}\sum_{k\leq n}\log\pn{k}\to\log L$, and $\log\pf{n}=O(n)$. (iv) $\pn{n}^{1/n}=c^{1/n}\phi_{1}(1-\fq^{n})^{1/n}(\phi_{1}-\phi_{2})^{-1/n}\to\phi_{1}$, and $\pf{n}=c^{n}\phi_{1}^{n(n-1)/2}(\fq;\fq)_{n}(1-\fq)^{-n}$, where $(\fq;\fq)_{n}=\prod_{k=1}^{n}(1-\fq^{k})$ \cite{GasperRahman}; hence $\log\pf{n}=\frac{n(n-1)}{2}\log\phi_{1}+O(n)$ and $\pf{n}^{1/n}=c\,\phi_{1}^{(n-1)/2}(\fq;\fq)_{n}^{1/n}(1-\fq)^{-1}$, whose limit is $0$, $c/(1-\fq)$ or $\infty$ in the three cases; for $\phi_{1}=1$ one has $\fq=\phi_{2}$. \hfill\(\Box\)

\begin{proposition}[Exponential polynomial deformed numbers]\label{prop:exp-poly}
Assume that $\pn{n}=\sum_{j=1}^{J}c_{j}\psi_{j}^{n}$ for $n\geq1$, with distinct $\psi_{1}>\psi_{2}>\dots>\psi_{J}>0$, real coefficients $c_{j}$ and $c_{1}>0$, and that $\pn{n}>0$ for all $n\geq1$. Then $\alpha_{\mathcal{R}(p,q)}=\beta_{\mathcal{R}(p,q)}=\psi_{1}$, $\gamma_{\mathcal{R}(p,q)}=\frac12\log\psi_{1}$, and $\kappa_{\mathcal{R}(p,q)}=K_{\mathcal{R}(p,q)}$ equals $0$, $c_{1}$ or $\infty$ according as $\psi_{1}<1$, $\psi_{1}=1$ or $\psi_{1}>1$.
\end{proposition}
{\it Proof.}
$\pn{n}=c_{1}\psi_{1}^{n}(1+\varepsilon_{n})$ with $\varepsilon_{n}=O((\psi_{2}/\psi_{1})^{n})$, so that $\pn{n}^{1/n}\to\psi_{1}$ and $\log\pf{n}=\frac{n(n+1)}{2}\log\psi_{1}+n\log c_{1}+\sum_{k\leq n}\log(1+\varepsilon_{k})$, the last sum being convergent. Hence $\log\pf{n}/n^{2}\to\frac12\log\psi_{1}$ and $\pf{n}^{1/n}=c_{1}\psi_{1}^{(n+1)/2}e^{O(1/n)}$. \hfill\(\Box\)

Proposition \ref{prop:exp-poly} covers structure functions which do not satisfy $\Hphi$. For instance, $\mathcal{R}(u,v)=\frac{u-v}{p-q}\,(1+u-v)$ with $0<q<p$ satisfies $\mathcal{R}(1,1)=0$ and $\pn{n}=[n]_{p,q}(1+p^{n}-q^{n})>0$, an exponential polynomial with dominant root $\max\{p,p^{2}\}$; the deformed numbers of this deformation are therefore in the Gaussian regime as soon as $p>1$.

The indices of the standard deformations are collected in Table \ref{tab:regimes} of Section \ref{sec:examples}. The three regimes $\phi_{1}<1$, $\phi_{1}=1$ and $\phi_{1}>1$ of Proposition \ref{prop:indices}(iv) will be called subexponential factorial regime, Jackson regime and Gaussian regime.

\section{Topological analysis and convergence of pseudo-norm Cauchy type}\label{sec:continuity}

This section follows the plan of the topological analysis of \cite{LankpetreAremuaBukweli2026}: continuity of the operators of the quantum algebra, invertibility of $P-Q$, growth hypotheses and explicit estimates, a generalized Borel Carath\'eodory lemma, topological $\mathcal{R}(p,q)-$Gelfand algebras, anisotropic $\mathcal{R}(p,q)-$sectors and Phragm\'en Lindel\"of principles, Fr\'echet nuclear spaces, the holomorphic case, a generalized Cauchy Hadamard theorem and weighted norms with Cauchy estimates. Every statement is given with its exact hypotheses, and the points where the formulation of \cite{LankpetreAremuaBukweli2026} has to be modified are indicated in remarks.

For $0<\rho\leq\infty$ let $D_{\rho}=\{z\in\mathbb{C}:|z|<\rho\}$, let $\mathcal{O}(D_{\rho})$ be the Fr\'echet space of holomorphic functions on $D_{\rho}$ with the seminorms $\|f\|_{s}:=\max_{|z|\leq s}|f(z)|$, $0<s<\rho$, and let $\mathcal{O}_{0}(D_{\rho})$ be the closed subspace of functions vanishing at the origin. For $\rho=\infty$, $D_{\infty}=\mathbb{C}$ and $\mathcal{O}(\mathbb{C})$ is the space of entire functions. We write $f(z)=\sum_{n\geq0}a_{n}z^{n}$ and recall the Cauchy inequalities $|a_{n}|\leq\|f\|_{s}s^{-n}$.

For $0<r<\infty$ we also use the disc algebra $\mathcal{O}(\overline{D_{r}})$ of functions holomorphic on $D_{r}$ and continuous on $\overline{D_{r}}$, a Banach space for $\|\cdot\|_{r}$; for $r<\rho$ the restriction maps $\mathcal{O}(D_{\rho})\to\mathcal{O}(\overline{D_{r}})$ are continuous, and the topology of $\mathcal{O}(D_{\rho})$ is the projective topology of these maps.

\subsection{Continuity of the $\mathcal{R}(p,q)-$quantum deformed algebra operators}\label{subsec:multipliers}

\begin{lemma}[Diagonal action on monomials]\label{lem:diag}
For every $n\geq0$, $Pz^{n}=p^{n}z^{n}$, $Qz^{n}=q^{n}z^{n}$ and $(P-Q)z^{n}=(p^{n}-q^{n})z^{n}$; if $0<q<p$, the eigenvalues $p^{n}-q^{n}$ are positive for $n\geq1$. Moreover $\mathcal{R}(P,Q)z^{n}=\pn{n}z^{n}$, $\partial_{p,q}z^{n}=[n]_{p,q}z^{n-1}$, and the three factors of \eqref{r5} act on $z^{n}$, $n\geq1$, by
\begin{eqnarray}\label{eq:factorization}
	z^{n}\ \stackrel{\mathcal{R}(P,Q)}{\longmapsto}\ \pn{n}z^{n}\ \stackrel{(p-q)(P-Q)^{-1}}{\longmapsto}\ \frac{(p-q)\pn{n}}{p^{n}-q^{n}}\,z^{n}\ \stackrel{\partial_{p,q}}{\longmapsto}\ \pn{n}z^{n-1}.
\end{eqnarray}
\end{lemma}
{\it Proof.}
$P z^{n}=(pz)^{n}=p^{n}z^{n}$ and similarly for $Q$; for $0<q<p$, $p^{n}-q^{n}=p^{n}(1-(q/p)^{n})>0$. The operator $\mathcal{R}(P,Q)$ is the diagonal operator with multiplier $\mathcal{R}(p^{n},q^{n})$; when the points $(p^{n},q^{n})$ lie in the bidisc of convergence of \eqref{r10}, it coincides with $\sum_{s,t}r_{st}P^{s}Q^{t}$ applied termwise, since $P^{s}Q^{t}z^{n}=(p^{n})^{s}(q^{n})^{t}z^{n}$. The composition \eqref{eq:factorization} follows. \hfill\(\Box\)

The order of the factors matters: $\mathcal{R}(P,Q)$ and $(P-Q)^{-1}$ act on $z^{n}$ before the degree is lowered by $\partial_{p,q}$, so that the multiplier of $\partial_{\mathcal{R}(p,q)}$ is exactly $\pn{n}$, in accordance with \eqref{eq:diag}; applying the dilations after $\partial_{p,q}$ would produce the incorrect multiplier $\mathcal{R}(p^{n-1},q^{n-1})(p^{n}-q^{n})/(p^{n-1}-q^{n-1})$.

\begin{lemma}\label{lem:multiplier}
Let $(m_{n})_{n\geq0}$ be a complex sequence with $\mu:=\limsup_{n}|m_{n}|^{1/n}<\infty$, and let $T_{m}$ be the diagonal operator $T_{m}\big(\sum_{n}a_{n}z^{n}\big)=\sum_{n}m_{n}a_{n}z^{n}$. Then, for $0<\rho\leq\infty$, $T_{m}$ maps $\mathcal{O}(D_{\rho})$ continuously into $\mathcal{O}(D_{\rho/\mu})$, with $\rho/\mu=\infty$ if $\mu=0$ or $\rho=\infty$. More precisely, for $\mu'>\mu$ let $C_{\mu'}:=\sup_{n}|m_{n}|\mu'^{-n}<\infty$; then, for $0<s$ and $0<r<\rho$ with $\mu's<r$,
\begin{eqnarray}\label{eq:multiplier-bound}
	\|T_{m}f\|_{s}&\leq&C_{\mu'}\,\frac{r}{r-\mu's}\,\|f\|_{r}.
\end{eqnarray}
The radius $\rho/\mu$ is optimal: for $\rho<\infty$ and $\mu>0$, the function $f(z)=(1-z/\rho)^{-1}\in\mathcal{O}(D_{\rho})$ satisfies $T_{m}f\notin\mathcal{O}(D_{\rho'})$ for every $\rho'>\rho/\mu$.
\end{lemma}
{\it Proof.}
For $|z|\leq s$, $|T_{m}f(z)|\leq\sum_{n}|m_{n}||a_{n}|s^{n}\leq C_{\mu'}\|f\|_{r}\sum_{n}(\mu's/r)^{n}$, which is \eqref{eq:multiplier-bound}. Given $s<\rho/\mu$, one chooses $\mu'>\mu$ and $r<\rho$ with $\mu's<r$, so that $T_{m}f$ converges uniformly on $\overline{D_{s}}$ and $T_{m}$ is continuous for the Fr\'echet topologies. For $f(z)=\sum_{n}\rho^{-n}z^{n}$, the radius of convergence of $\sum_{n}m_{n}\rho^{-n}z^{n}$ is $1/\limsup_{n}(|m_{n}|\rho^{-n})^{1/n}=\rho/\mu$ by the Cauchy Hadamard formula. \hfill\(\Box\)

\begin{theorem}[Invertibility of $P-Q$ on $\mathcal{O}_{0}$]\label{thm:PQ}
Let $0<q<p$ and $0<\rho\leq\infty$. The operator $P-Q$ maps $\mathcal{O}_{0}(D_{\rho})$ injectively into $\mathcal{O}_{0}(D_{\rho/p})$, and its inverse is the diagonal operator
\begin{eqnarray}\label{eq:PQ-inverse}
	(P-Q)^{-1}\Big(\sum_{n\geq1}a_{n}z^{n}\Big)&=&\sum_{n\geq1}\frac{a_{n}}{p^{n}-q^{n}}\,z^{n},
\end{eqnarray}
which maps $\mathcal{O}_{0}(D_{\rho})$ continuously into $\mathcal{O}_{0}(D_{p\rho})$, with
\begin{eqnarray}\label{eq:PQ-bound}
	\|(P-Q)^{-1}f\|_{s}&\leq&\frac{p}{p-q}\,\frac{s}{pr-s}\,\|f\|_{r},\qquad 0<s<pr,\ r<\rho .
\end{eqnarray}
The radius $p\rho$ is optimal. In particular, if $p<1$ and $\rho<\infty$, then $(P-Q)^{-1}$ does not map $\mathcal{O}_{0}(D_{\rho})$ into itself, whereas for $p\geq1$ it is a continuous endomorphism of $\mathcal{O}_{0}(D_{\rho})$.
\end{theorem}
{\it Proof.}
$(P-Q)z^{n}=(p^{n}-q^{n})z^{n}$ with $p^{n}-q^{n}>0$ for $n\geq1$, which gives injectivity on $\mathcal{O}_{0}$ and \eqref{eq:PQ-inverse}. Since $p^{n}-q^{n}=p^{n}(1-(q/p)^{n})\geq p^{n}(1-q/p)$, we get $|a_{n}|/(p^{n}-q^{n})\leq\frac{p}{p-q}\|f\|_{r}(pr)^{-n}$, and summation over $n\geq1$ gives \eqref{eq:PQ-bound}. The multiplier $(p^{n}-q^{n})^{-1}$ has $n-$th root limit $1/p$, so that Lemma \ref{lem:multiplier} gives optimality. \hfill\(\Box\)

The factorization \eqref{r5} must be read on monomials: $\mathcal{R}(P,Q)$ has multiplier $\pn{n}$ and annihilates the constants, $(p-q)(P-Q)^{-1}$ has multiplier $(p-q)/(p^{n}-q^{n})$ and is applied to $\mathcal{O}_{0}$, and $\partial_{p,q}$ has multiplier $(p^{n}-q^{n})/(p-q)$ followed by the shift $z^{n}\mapsto z^{n-1}$. The losses and gains of radius of the last two factors, $p$ and $1/p$ by Theorem \ref{thm:PQ} and Lemma \ref{lem:multiplier}, compensate exactly, which explains why the net loss of radius in Theorem \ref{thm:continuity} is governed by $\beta_{\mathcal{R}(p,q)}$ alone.

\begin{proposition}[Action of $\mathcal{R}(P,Q)$ and continuity]\label{prop:dilations}
Assume \Hpos. Then $P$ maps $\mathcal{O}(D_{\rho})$ continuously into $\mathcal{O}(D_{\rho/p})$, $Q$ maps $\mathcal{O}(D_{\rho})$ into $\mathcal{O}(D_{\rho/q})$, and $\mathcal{R}(P,Q)$ maps $\mathcal{O}(D_{\rho})$ continuously into $\mathcal{O}(D_{\rho/\beta_{\mathcal{R}(p,q)}})$, all radii being optimal. For $0<s$, $r<\rho$ and $\beta'>\beta_{\mathcal{R}(p,q)}$ with $\beta's<r$, $\mathcal{R}(P,Q):\mathcal{O}(\overline{D_{r}})\to\mathcal{O}(\overline{D_{s}})$ is bounded, with $\|\mathcal{R}(P,Q)f\|_{s}\leq\sup_{n}\big(\pn{n}\beta'^{-n}\big)\,\frac{r}{r-\beta's}\,\|f\|_{r}$.
\end{proposition}
{\it Proof.}
The three operators are diagonal with multipliers $p^{n}$, $q^{n}$ and $\pn{n}$, whose $n-$th root upper limits are $p$, $q$ and $\beta_{\mathcal{R}(p,q)}$; apply Lemma \ref{lem:multiplier}, the value $\mathcal{R}(1,1)=0$ at $n=0$ being irrelevant. \hfill\(\Box\)

\begin{theorem}[Continuity of $\partial_{\mathcal{R}(p,q)}$]\label{thm:continuity}
Assume \Hpos\ and let $0<\rho\leq\infty$. Then $\partial_{\mathcal{R}(p,q)}$ maps $\mathcal{O}(D_{\rho})$ continuously into $\mathcal{O}(D_{\rho/\beta_{\mathcal{R}(p,q)}})$, and for $0<r<\rho$ and $s>0$ with $\beta_{\mathcal{R}(p,q)}s<r$,
\begin{eqnarray}\label{eq:derivative-bound}
	\|\partial_{\mathcal{R}(p,q)}f\|_{s}&\leq&\Big(\sum_{n\geq1}\pn{n}\,\frac{s^{n-1}}{r^{n}}\Big)\,\|f\|_{r},
\end{eqnarray}
the series being convergent. The radius $\rho/\beta_{\mathcal{R}(p,q)}$ is optimal. Consequently:
\begin{enumerate}
\item[\textnormal{(i)}] if $\beta_{\mathcal{R}(p,q)}\leq1$, then $\partial_{\mathcal{R}(p,q)}$ is a continuous endomorphism of $\mathcal{O}(D_{\rho})$ for every $\rho\in(0,\infty]$;
\item[\textnormal{(ii)}] $\partial_{\mathcal{R}(p,q)}$ is a continuous endomorphism of the space $\mathcal{O}(\mathbb{C})$ of entire functions;
\item[\textnormal{(iii)}] if $\beta_{\mathcal{R}(p,q)}>1$ and $\rho<\infty$, then $\partial_{\mathcal{R}(p,q)}$ does not map $\mathcal{O}(D_{\rho})$ into itself.
\end{enumerate}
\end{theorem}
{\it Proof.}
By \eqref{eq:diag}, $\partial_{\mathcal{R}(p,q)}f(z)=\sum_{n\geq0}\pn{n+1}a_{n+1}z^{n}$. For $|z|\leq s$, the Cauchy inequalities give
$|\partial_{\mathcal{R}(p,q)}f(z)|\leq\sum_{n\geq1}\pn{n}\|f\|_{r}r^{-n}s^{n-1}$, which is \eqref{eq:derivative-bound}; the series converges by the root test since $\limsup_{n}(\pn{n}s^{n-1}r^{-n})^{1/n}=\beta_{\mathcal{R}(p,q)}s/r<1$. Continuity follows as in Lemma \ref{lem:multiplier}. For optimality, the function $f(z)=(1-z/\rho)^{-1}$ gives $\partial_{\mathcal{R}(p,q)}f(z)=\sum_{n}\pn{n+1}\rho^{-n-1}z^{n}$, whose radius of convergence is $\rho/\beta_{\mathcal{R}(p,q)}$. Assertions (i), (ii) and (iii) are immediate. \hfill\(\Box\)

\begin{remark}\label{rem:Hphi-derivative}
Under $\Hphi$, $\beta_{\mathcal{R}(p,q)}=\phi_{1}$ and the deformed derivative admits the Jackson type representation
\begin{eqnarray}\label{eq:Jackson-derivative}
	\partial_{\mathcal{R}(p,q)}f(z)&=&c\,\frac{f(\phi_{1}z)-f(\phi_{2}z)}{(\phi_{1}-\phi_{2})z},\qquad 0<|z|<\rho/\phi_{1},
\end{eqnarray}
since both sides act on $z^{n}$ by multiplication by $\pn{n}z^{-1}$. The loss of radius in Theorem \ref{thm:continuity} is thus visible: \eqref{eq:Jackson-derivative} requires the value of $f$ at $\phi_{1}z$. This representation is the $\mathcal{R}(p,q)-$version of the one used in \cite{ref58, SoontharanonSitthiwirattham2020}.
\end{remark}

\begin{example}\label{ex:JS}
For the Jagannathan Srinivasa Rao deformation with $0<q\leq p$, $[n]_{p,q}=\sum_{k=0}^{n-1}p^{n-1-k}q^{k}\leq n\,p^{n-1}$, and \eqref{eq:derivative-bound} gives, for $ps<r$,
\begin{eqnarray}\label{eq:JS-bound}
	\|\partial_{p,q}f\|_{s}&\leq&\sum_{n\geq1}n\,p^{n-1}\frac{s^{n-1}}{r^{n}}\,\|f\|_{r}\;=\;\frac{r}{(r-ps)^{2}}\,\|f\|_{r},
\end{eqnarray}
which reduces for $p=q=1$ to the classical Cauchy estimate $\|f'\|_{s}\leq r(r-s)^{-2}\|f\|_{r}$. The constant is attained in the limit by $f(z)=(1-z/r)^{-1}$ when $q\to p$.
\end{example}

\begin{example}\label{ex:Hphi-bound}
Under $\Hphi$ with $c=1$, $\pn{n}=\phi_{1}^{n-1}[n]_{\fq}\leq\phi_{1}^{n-1}/(1-\fq)$, and \eqref{eq:derivative-bound} gives, for $\phi_{1}s<r$,
\begin{eqnarray}\label{eq:Hphi-bound}
	\|\partial_{\mathcal{R}(p,q)}f\|_{s}&\leq&\frac{1}{1-\fq}\sum_{n\geq1}\phi_{1}^{n-1}\frac{s^{n-1}}{r^{n}}\,\|f\|_{r}\;=\;\frac{\|f\|_{r}}{(1-\fq)(r-\phi_{1}s)} .
\end{eqnarray}
This bound has the correct singularity at $s=r/\phi_{1}$, but it does not reduce to the Cauchy estimate in the undeformed limit, because $(1-\fq)^{-1}$ blows up; the bound \eqref{eq:JS-bound}, based on $[n]_{\fq}\leq n$, has the opposite features. Combining both, $\pn{n}\leq\phi_{1}^{n-1}\min\{n,(1-\fq)^{-1}\}$.
\end{example}

\begin{remark}\label{rem:continuity-earlier}
In \cite{LankpetreAremuaBukweli2026} the continuity of $\partial_{\mathcal{R}(p,q)}$ on $\mathcal{O}(\mathbb{D}_{R})$ and the bijectivity of $P-Q$ on $\mathcal{O}_{0}(\overline{D_{r}})$ are stated for every radius. Theorems \ref{thm:continuity} and \ref{thm:PQ} show that the first statement holds exactly when $\beta_{\mathcal{R}(p,q)}\leq1$ or $R=\infty$, and the second one exactly when $p\geq1$; in general the bounded maps are $\mathcal{O}(\overline{D_{r}})\to\mathcal{O}(\overline{D_{s}})$ with $\beta_{\mathcal{R}(p,q)}s<r$, respectively $s<pr$.
\end{remark}

\subsection{Explicit growth hypotheses, concrete estimates and an illustrative example}\label{subsec:growth}

Two growth hypotheses of \cite{LankpetreAremuaBukweli2026} make the constants of Theorem \ref{thm:continuity} explicit:
\begin{enumerate}
\item[\textnormal{(G1)}] (exponential bound) there exist $C_{0}>0$ and $B>0$ with $\pn{n}\leq C_{0}B^{n}$ for all $n\geq1$;
\item[\textnormal{(G2)}] (Gaussian bound) there exist $C_{1}>0$ and $\lambda>0$ with $\pn{n}\leq C_{1}e^{\lambda n^{2}}$ for all $n\geq1$.
\end{enumerate}

\begin{corollary}\label{cor:G1}
Under \textnormal{(G1)}, $\beta_{\mathcal{R}(p,q)}\leq B$ and, for $0<Bs<r<\rho$ and $f\in\mathcal{O}(\overline{D_{r}})$,
\begin{eqnarray}\label{eq:G1-bound}
	\|\partial_{\mathcal{R}(p,q)}f\|_{s}&\leq&\frac{C_{0}B}{r-Bs}\,\|f\|_{r} .
\end{eqnarray}
Under $\Hphi$, \textnormal{(G1)} holds with $B=\phi_{1}$ and $C_{0}=c/(\phi_{1}(1-\fq))$. Under \textnormal{(G2)} alone, $\beta_{\mathcal{R}(p,q)}$ may be infinite, so that \Hpos\ fails; in that case $\partial_{\mathcal{R}(p,q)}$ does not map $\mathcal{O}(\mathbb{C})$ into the holomorphic functions on any disc, and the appropriate spaces are the factorial weighted spaces of Section \ref{sec:scale}, on which $\partial_{\mathcal{R}(p,q)}$ has norm $1/r$ whatever the growth of $\pn{n}$.
\end{corollary}
{\it Proof.}
By \eqref{eq:derivative-bound}, $\|\partial_{\mathcal{R}(p,q)}f\|_{s}\leq C_{0}\sum_{n\geq1}B^{n}s^{n-1}r^{-n}\|f\|_{r}=C_{0}\frac{B}{r}\sum_{n\geq1}(Bs/r)^{n-1}\|f\|_{r}$, which is \eqref{eq:G1-bound}. Under $\Hphi$, $\pn{n}=c\,\phi_{1}^{n-1}[n]_{\fq}\leq c\,\phi_{1}^{n}/(\phi_{1}(1-\fq))$. If $\beta_{\mathcal{R}(p,q)}=\infty$, choose $n_{j}\to\infty$ with $\pn{n_{j}}^{1/n_{j}}\to\infty$; the function $f(z)=\sum_{j}\pn{n_{j}}^{-1/2}z^{n_{j}}$ is entire, while $\partial_{\mathcal{R}(p,q)}f(z)=\sum_{j}\pn{n_{j}}^{1/2}z^{n_{j}-1}$ has radius of convergence $0$. The last statement is Theorem \ref{thm:spectrum}. \hfill\(\Box\)

\begin{example}[The structure function $\mathcal{R}(u,v)=u-v$]\label{ex:u-v}
Let $\mathcal{R}(u,v)=u-v$ and $0<q<p$. Then $\pn{n}=p^{n}-q^{n}$, $\beta_{\mathcal{R}(p,q)}=p$, and $\partial_{\mathcal{R}(p,q)}f(z)=(f(pz)-f(qz))/z$. For $f\in\mathcal{O}(\overline{D_{\rho}})$ and $pr<\rho$, \eqref{eq:derivative-bound} gives the explicit bound
\begin{eqnarray}\label{eq:u-v-bound}
	\|\partial_{\mathcal{R}(p,q)}f\|_{r}&\leq&\sum_{n\geq1}(p^{n}-q^{n})\frac{r^{n-1}}{\rho^{n}}\,\|f\|_{\rho}\;=\;\frac{(p-q)\,\rho}{(\rho-pr)(\rho-qr)}\,\|f\|_{\rho},
\end{eqnarray}
since $\sum_{n\geq1}(p^{n}-q^{n})r^{n-1}\rho^{-n}=\frac{p}{\rho-pr}-\frac{q}{\rho-qr}$. The constant is the exact sum of the majorant series; it is larger than $1/(\rho-pr)$ as soon as $(p-q)\rho>\rho-qr$, so that a bound of the form $\|\partial_{\mathcal{R}(p,q)}f\|_{r}\leq\|f\|_{\rho}/(\rho-pr)$ cannot hold in general.
\end{example}

\subsection{A mean value representation and deformed Bernstein and Schwarz Pick inequalities}\label{subsec:mean-value}

\begin{lemma}[Mean value representation]\label{lem:mean-value}
Assume $\Hphi$ and let $f$ be holomorphic on $D_{\rho}$. For $|z|<\rho/\phi_{1}$,
\begin{eqnarray}\label{eq:mean-value}
	\partial_{\mathcal{R}(p,q)}f(z)&=&\frac{c}{\phi_{1}-\phi_{2}}\int_{\phi_{2}}^{\phi_{1}}f'(tz)\,dt .
\end{eqnarray}
Consequently $\|\partial_{\mathcal{R}(p,q)}f\|_{s}\leq c\,\max_{|w|\leq\phi_{1}s}|f'(w)|$ for $\phi_{1}s<\rho$.
\end{lemma}
{\it Proof.}
For $z\neq0$, $\int_{\phi_{2}}^{\phi_{1}}f'(tz)\,dt=(f(\phi_{1}z)-f(\phi_{2}z))/z$, and \eqref{eq:Jackson-derivative} gives \eqref{eq:mean-value}; for $z=0$ both sides equal $c f'(0)$. The estimate follows since $|tz|\leq\phi_{1}s$ for $t\in[\phi_{2},\phi_{1}]$ and $|z|\leq s$. \hfill\(\Box\)

The deformed derivative is thus the average of the classical derivative over the dilation segment $[\phi_{2}z,\phi_{1}z]$, multiplied by the gauge constant. This representation transfers classical inequalities for $f'$ to inequalities for $\partial_{\mathcal{R}(p,q)}f$, without passing through the coefficients.

\begin{theorem}[Deformed Bernstein inequality]\label{thm:bernstein}
Assume $\Hphi$ and let $P$ be a polynomial of degree at most $n\geq1$. Then
\begin{eqnarray}\label{eq:bernstein}
	\max_{|z|\leq1}\big|\partial_{\mathcal{R}(p,q)}P(z)\big|&\leq&c\,B_{n}\,\max_{|z|\leq1}|P(z)|,\qquad
	B_{n}\;:=\;\frac{n}{\phi_{1}-\phi_{2}}\int_{\phi_{2}}^{\phi_{1}}\max\{1,t\}^{n-1}\,dt,
\end{eqnarray}
that is $B_{n}=n$ if $\phi_{1}\leq1$, $B_{n}=(\phi_{1}^{n}-\phi_{2}^{n})/(\phi_{1}-\phi_{2})$ if $\phi_{2}\geq1$, and $B_{n}=(n(1-\phi_{2})+\phi_{1}^{n}-1)/(\phi_{1}-\phi_{2})$ if $\phi_{2}<1<\phi_{1}$. If $\phi_{2}\geq1$ the inequality is an equality for $P(z)=z^{n}$, so that $c\,B_{n}=\pn{n}$ is the best constant.
\end{theorem}
{\it Proof.}
Let $\|P\|=\max_{|z|\leq1}|P|$. By the Bernstein inequality and the maximum principle, $|P'(w)|\leq n\|P\|$ for $|w|\leq1$, and, applying the maximum principle to $P'(w)/w^{n-1}$ on $\{|w|\geq1\}\cup\{\infty\}$, $|P'(w)|\leq n|w|^{n-1}\|P\|$ for $|w|\geq1$. Hence $|P'(tz)|\leq n\max\{1,t\}^{n-1}\|P\|$ for $|z|\leq1$ and $t>0$, and Lemma \ref{lem:mean-value} gives \eqref{eq:bernstein}. The three expressions of $B_{n}$ are elementary integrals. For $P=z^{n}$, $\partial_{\mathcal{R}(p,q)}P=\pn{n}z^{n-1}$ has maximum $\pn{n}=c(\phi_{1}^{n}-\phi_{2}^{n})/(\phi_{1}-\phi_{2})$ on the closed unit disc. \hfill\(\Box\)

\begin{corollary}[Deformed Schwarz Pick inequality]\label{cor:schwarz-pick}
Assume $\Hphi$ and let $f$ be a holomorphic self map of the unit disc. For $0<|z|<1/\phi_{1}$,
\begin{eqnarray}\label{eq:schwarz-pick}
	\big|\partial_{\mathcal{R}(p,q)}f(z)\big|&\leq&\frac{c}{(\phi_{1}-\phi_{2})|z|}\Big(\operatorname{artanh}(\phi_{1}|z|)-\operatorname{artanh}(\phi_{2}|z|)\Big),
\end{eqnarray}
and $|\partial_{\mathcal{R}(p,q)}f(0)|\leq c$. In the undeformed limit $c=1$, $\phi_{2}\to\phi_{1}=1$, the right hand side tends to $1/(1-|z|^{2})$.
\end{corollary}
{\it Proof.}
By the Schwarz Pick lemma, $|f'(w)|\leq(1-|f(w)|^{2})/(1-|w|^{2})\leq1/(1-|w|^{2})$, and \eqref{eq:mean-value} gives $|\partial_{\mathcal{R}(p,q)}f(z)|\leq\frac{c}{\phi_{1}-\phi_{2}}\int_{\phi_{2}}^{\phi_{1}}\frac{dt}{1-t^{2}|z|^{2}}$, which is \eqref{eq:schwarz-pick}. At $z=0$, $\partial_{\mathcal{R}(p,q)}f(0)=cf'(0)$ and $|f'(0)|\leq1$. The limit is the derivative of $\operatorname{artanh}(t|z|)/|z|$ at $t=1$. \hfill\(\Box\)

\subsection{Deformed pseudo-norm, deformed discs and a generalized Borel Carath\'eodory lemma}\label{subsec:gauge}

The pseudo-norm of \cite{LankpetreAremuaBukweli2026} is the following deformed gauge.

\begin{definition}\label{def:gauge}
Assume \Hpos. The deformed gauge of $z\in\mathbb{C}$ and the deformed radius function are
\begin{eqnarray}\label{eq:gauge}
	\|z\|_{\mathcal{R}(p,q)}&:=&\sup_{n\geq1}\Big(\frac{|z|}{\pn{n}}\Big)^{1/n},\qquad
	s_{\mathcal{R}(p,q)}(r)\;:=\;\inf_{n\geq1}r^{n}\pn{n},\qquad r>0,
\end{eqnarray}
and the deformed disc of gauge radius $r$ is $\mathbb{D}_{\mathcal{R}(p,q)}(r):=\{z\in\mathbb{C}:|z|<s_{\mathcal{R}(p,q)}(r)\}$.
\end{definition}

\begin{proposition}\label{prop:gauge}
Assume \Hpos. Then:
\begin{enumerate}
\item[\textnormal{(i)}] $\{z:\|z\|_{\mathcal{R}(p,q)}\leq r\}=\{z:|z|\leq s_{\mathcal{R}(p,q)}(r)\}$; in particular every deformed ball is a Euclidean disc, possibly reduced to $\{0\}$;
\item[\textnormal{(ii)}] $\|z\|_{\mathcal{R}(p,q)}\geq\max\{|z|/\pn{1},\,1/\alpha_{\mathcal{R}(p,q)}\}$ for $z\neq0$, with $1/\alpha_{\mathcal{R}(p,q)}=\infty$ if $\alpha_{\mathcal{R}(p,q)}=0$, whereas $\|0\|_{\mathcal{R}(p,q)}=0$; hence the gauge is discontinuous at the origin unless $\alpha_{\mathcal{R}(p,q)}=\infty$;
\item[\textnormal{(iii)}] $s_{\mathcal{R}(p,q)}$ is nondecreasing, $s_{\mathcal{R}(p,q)}(r)\leq r\pn{1}$, $s_{\mathcal{R}(p,q)}(r)=0$ if $r\,\alpha_{\mathcal{R}(p,q)}<1$, and $s_{\mathcal{R}(p,q)}(r)>0$ if $r\,\alpha_{\mathcal{R}(p,q)}>1$;
\item[\textnormal{(iv)}] the function $t\mapsto\log s_{\mathcal{R}(p,q)}(e^{t})$ is concave on the interval where it is finite.
\end{enumerate}
\end{proposition}
{\it Proof.}
(i) $\|z\|_{\mathcal{R}(p,q)}\leq r$ if and only if $|z|\leq r^{n}\pn{n}$ for every $n$, that is $|z|\leq s_{\mathcal{R}(p,q)}(r)$. (ii) The term $n=1$ gives $|z|/\pn{1}$. For $z\neq0$, $(|z|/\pn{n})^{1/n}=|z|^{1/n}\pn{n}^{-1/n}$, whose upper limit is $1/\liminf_{n}\pn{n}^{1/n}$. (iii) Each $r\mapsto r^{n}\pn{n}$ is increasing, the term $n=1$ gives the upper bound, and $r^{n}\pn{n}=(r\pn{n}^{1/n})^{n}$ tends to $0$ along a subsequence if $r\alpha_{\mathcal{R}(p,q)}<1$ and to $+\infty$ if $r\alpha_{\mathcal{R}(p,q)}>1$, in which case the infimum is attained and positive. (iv) $\log s_{\mathcal{R}(p,q)}(e^{t})=\inf_{n}(nt+\log\pn{n})$ is an infimum of affine functions. \hfill\(\Box\)

\begin{theorem}\label{thm:gauge-Hphi}
Assume $\Hphi$ with $c=1$. Then, for all $z\in\mathbb{C}\setminus\{0\}$ and $r>0$,
\begin{eqnarray}\label{eq:gauge-Hphi}
	\|z\|_{\mathcal{R}(p,q)}&=&\max\Big\{|z|,\frac{1}{\phi_{1}}\Big\},\qquad
	s_{\mathcal{R}(p,q)}(r)\;=\;\left\{\begin{array}{ll} r, & r\geq1/\phi_{1},\\ 0, & r<1/\phi_{1}.\end{array}\right.
\end{eqnarray}
Consequently $\mathbb{D}_{\mathcal{R}(p,q)}(r)=D_{r}$ for $r\geq1/\phi_{1}$ and $\mathbb{D}_{\mathcal{R}(p,q)}(r)=\emptyset$ for $r<1/\phi_{1}$; the closed deformed ball $\{\|z\|_{\mathcal{R}(p,q)}\leq r\}$ is $\overline{D_{r}}$ for $r\geq1/\phi_{1}$ and $\{0\}$ otherwise.
\end{theorem}
{\it Proof.}
Under $\Hphi$ with $c=1$, $\pn{n}=\phi_{1}^{n-1}[n]_{\fq}$ with $[n]_{\fq}=(1-\fq^{n})/(1-\fq)\geq1$ nondecreasing. Hence $r^{n}\pn{n}=\phi_{1}^{-1}(r\phi_{1})^{n}[n]_{\fq}$. If $r\phi_{1}\geq1$ this is nondecreasing in $n$, and the infimum is the value $r$ at $n=1$; if $r\phi_{1}<1$ it tends to $0$. For the gauge, put $\lambda=|z|\phi_{1}$; then $(|z|/\pn{n})^{1/n}=\phi_{1}^{-1}(\lambda/[n]_{\fq})^{1/n}$. If $\lambda\geq1$, $(\lambda/[n]_{\fq})^{1/n}\leq\lambda^{1/n}\leq\lambda$ with equality at $n=1$, so that $\|z\|_{\mathcal{R}(p,q)}=|z|$. If $\lambda<1$, every term is smaller than $\phi_{1}^{-1}$ and the terms tend to $\phi_{1}^{-1}$, since $[n]_{\fq}^{1/n}\to1$, so that $\|z\|_{\mathcal{R}(p,q)}=1/\phi_{1}$. \hfill\(\Box\)

Theorem \ref{thm:gauge-Hphi} shows that the deformation introduces a minimal resolution radius $1/\phi_{1}$: the deformed gauge does not distinguish the points of the punctured disc $0<|z|<1/\phi_{1}$, and the deformed discs of gauge radius smaller than $1/\phi_{1}$ collapse. In the Gaussian regime $\phi_{1}>1$ this radius is smaller than $1$, in the subexponential factorial regime $\phi_{1}<1$ it is larger than $1$, and in the Jackson regime it equals $1$. For a gauge constant $c\neq1$ the same computation applies to $\pn{n}/c$, and $s_{\mathcal{R}(p,q)}(r)=c\,r$ for $r\geq1/\phi_{1}$, $s_{\mathcal{R}(p,q)}(r)=0$ otherwise.

\begin{theorem}[Generalized $\mathcal{R}(p,q)-$Borel Carath\'eodory lemma]\label{thm:BC}
Assume \Hpos, let $0<r<R$ with $0<s_{\mathcal{R}(p,q)}(r)<s_{\mathcal{R}(p,q)}(R)$, and let $f$ be holomorphic on a neighbourhood of the closed deformed ball $\{\|w\|_{\mathcal{R}(p,q)}\leq R\}$. Put $A_{\mathcal{R}(p,q)}(R):=\max_{\|w\|_{\mathcal{R}(p,q)}\leq R}\Re f(w)$. Then, for every $z$ with $\|z\|_{\mathcal{R}(p,q)}\leq r$,
\begin{eqnarray}\label{eq:BC}
	|f(z)-f(0)|&\leq&\frac{2s_{\mathcal{R}(p,q)}(r)}{s_{\mathcal{R}(p,q)}(R)-s_{\mathcal{R}(p,q)}(r)}\,\big(A_{\mathcal{R}(p,q)}(R)-\Re f(0)\big),
\end{eqnarray}
and consequently
\begin{eqnarray}\label{eq:BC2}
	|f(z)|&\leq&\frac{2s_{\mathcal{R}(p,q)}(r)}{s_{\mathcal{R}(p,q)}(R)-s_{\mathcal{R}(p,q)}(r)}\,A_{\mathcal{R}(p,q)}(R)+\frac{s_{\mathcal{R}(p,q)}(R)+s_{\mathcal{R}(p,q)}(r)}{s_{\mathcal{R}(p,q)}(R)-s_{\mathcal{R}(p,q)}(r)}\,|f(0)| .
\end{eqnarray}
Under $\Hphi$ with $c=1$ and $1/\phi_{1}\leq r<R$, \eqref{eq:BC} and \eqref{eq:BC2} hold with $s_{\mathcal{R}(p,q)}(r)=r$ and $s_{\mathcal{R}(p,q)}(R)=R$.
\end{theorem}
{\it Proof.}
By Proposition \ref{prop:gauge}(i), the closed deformed balls are the Euclidean discs $\overline{D_{s_{\mathcal{R}(p,q)}(r)}}$ and $\overline{D_{s_{\mathcal{R}(p,q)}(R)}}$. Put $g=f-f(0)$, $\varrho=s_{\mathcal{R}(p,q)}(R)$ and $A=A_{\mathcal{R}(p,q)}(R)-\Re f(0)$. By the mean value property, $A\geq0$, and $A=0$ only if $\Re g\equiv0$, in which case $g\equiv0$. If $A>0$, the function $h=g/(2A-g)$ is holomorphic on $\overline{D_{\varrho}}$ since $\Re(2A-g)\geq A>0$, it satisfies $h(0)=0$ and $|h|\leq1$, because $|g|^{2}\leq|2A-g|^{2}$ is equivalent to $\Re g\leq A$. By the Schwarz lemma $|h(z)|\leq|z|/\varrho$, and $g=2Ah/(1+h)$ gives $|g(z)|\leq2A|z|/(\varrho-|z|)$ for $|z|<\varrho$, which is \eqref{eq:BC} for $|z|\leq s_{\mathcal{R}(p,q)}(r)$. Then \eqref{eq:BC2} follows from $|f(z)|\leq|f(0)|+|g(z)|$ and $-\Re f(0)\leq|f(0)|$. The last statement is Theorem \ref{thm:gauge-Hphi}. \hfill\(\Box\)

The inequality \eqref{eq:BC2} is the classical Borel Carath\'eodory inequality \cite{Titchmarsh1939} transported to the deformed balls; its content in the deformed setting lies in the identification of these balls, Proposition \ref{prop:gauge} and Theorem \ref{thm:gauge-Hphi}, which show that only gauge radii above the threshold $1/\alpha_{\mathcal{R}(p,q)}$ are meaningful.

In \cite{LankpetreAremuaBukweli2026} the inequality \eqref{eq:BC2} is written with the gauge radii $r$ and $R$ in place of $s_{\mathcal{R}(p,q)}(r)$ and $s_{\mathcal{R}(p,q)}(R)$, under the assumption $0<q<p\leq1$. By Theorem \ref{thm:gauge-Hphi}, the two formulations agree under $\Hphi$ with $c=1$ as soon as $r\geq1/\phi_{1}$; below this threshold the deformed balls collapse and only the formulation through $s_{\mathcal{R}(p,q)}$ is meaningful.

\subsection{Topological $\mathcal{R}(p,q)-$Gelfand algebras}\label{subsec:gelfand-alg}

\begin{definition}[Topological $\mathcal{R}(p,q)-$Gelfand algebra]\label{def:gelfand-alg}
Let $\mathscr{A}$ be a complete Hausdorff locally convex algebra over $\mathbb{C}$ whose topology is defined by a family $\{\mathfrak{p}_{\alpha}\}_{\alpha\in\mathcal{I}}$ of submultiplicative seminorms, $\mathfrak{p}_{\alpha}(ab)\leq\mathfrak{p}_{\alpha}(a)\mathfrak{p}_{\alpha}(b)$, so that the multiplication is jointly continuous. Let $\mathrm{Spec}(\mathscr{A})$ be the set of nonzero continuous multiplicative linear functionals and $r(a):=\sup\{|\chi(a)|:\chi\in\mathrm{Spec}(\mathscr{A})\}$. We say that $\mathscr{A}$ is a topological $\mathcal{R}(p,q)-$Gelfand algebra if $\mathrm{Spec}(\mathscr{A})\neq\emptyset$, $r(a)<\infty$ for every $a$, $\mathscr{A}$ has a bounded approximate identity, and $\mathscr{A}$ is stable under the deformed exponential calculus: for every $a\in\mathscr{A}$ and $\lambda\in\mathbb{C}$ the series $\eR(\lambda a):=\sum_{n\geq0}\lambda^{n}a^{n}/\pf{n}$ converges in $\mathscr{A}$.
\end{definition}

\begin{proposition}\label{prop:gelfand-alg}
Assume \Hpos.
\begin{enumerate}
\item[\textnormal{(i)}] For every $\chi\in\mathrm{Spec}(\mathscr{A})$ there is an index $\alpha=\alpha(\chi)$ with $|\chi(a)|\leq\mathfrak{p}_{\alpha}(a)$ for all $a$; consequently $r(a)\leq\sup_{\alpha}\mathfrak{p}_{\alpha}(a)$, and $\chi(\eR(\lambda a))=\eR(\lambda\chi(a))$ whenever $\eR(\lambda a)$ is defined.
\item[\textnormal{(ii)}] If $\kappa_{\mathcal{R}(p,q)}=\infty$, in particular under \HG, every complete locally $m-$convex algebra is stable under the deformed exponential calculus.
\item[\textnormal{(iii)}] There is in general no constant $C_{\alpha}$ with $\mathfrak{p}_{\alpha}(a)\leq C_{\alpha}r(a)$.
\end{enumerate}
\end{proposition}
{\it Proof.}
(i) Since $\chi$ is continuous there are $\alpha$ and $C$ with $|\chi(a)|\leq C\mathfrak{p}_{\alpha}(a)$; then $|\chi(a)|^{n}=|\chi(a^{n})|\leq C\mathfrak{p}_{\alpha}(a)^{n}$, and $n\to\infty$ gives $|\chi(a)|\leq\mathfrak{p}_{\alpha}(a)$. Continuity and multiplicativity of $\chi$ give $\chi(\eR(\lambda a))=\sum_{n}\lambda^{n}\chi(a)^{n}/\pf{n}$. (ii) $\mathfrak{p}_{\alpha}(\lambda^{n}a^{n}/\pf{n})\leq(|\lambda|\mathfrak{p}_{\alpha}(a))^{n}/\pf{n}$, which is summable when $\kappa_{\mathcal{R}(p,q)}=\infty$, and $\mathscr{A}$ is complete. (iii) In the two dimensional algebra spanned by $1$ and $x$ with $x^{2}=0$ and the norm $|a|+|b|$ of $a+bx$, the only character is $a+bx\mapsto a$, so that $r(x)=0<1=\|x\|$. \hfill\(\Box\)

The unital Banach algebras $\big(\ER{r},\star_{\mathcal{R}(p,q)}\big)$ of Section \ref{sec:algebra} are topological $\mathcal{R}(p,q)-$Gelfand algebras in the Gaussian regime, and the Fr\'echet algebra $\ER{\infty^{-}}$ of Theorem \ref{thm:frechet-algebra} is a complete locally $m-$convex one. Proposition \ref{prop:gelfand-alg}(iii) shows that growth bounds on the characters control the spectral radius only; the Phragm\'en Lindel\"of principles below for the seminorms are therefore proved by means of the continuous linear functionals and the Hahn Banach theorem.

\subsection{Anisotropic $\mathcal{R}(p,q)-$sectors}\label{subsec:sectors}

\begin{definition}[Anisotropic $\mathcal{R}(p,q)-$sector]\label{def:sector}
Assume \Hpos\ and put
\begin{eqnarray}\label{eq:rho-k}
	\rho_{\mathcal{R}(p,q)}(k)&:=&\frac{\log\mathcal{R}(p^{k},q^{k})}{k}\;=\;\frac{\log\pn{k}}{k},\qquad k\in\mathbb{N},\qquad
	\rho^{*}_{\mathcal{R}(p,q)}\;:=\;\sup_{k\geq1}\rho_{\mathcal{R}(p,q)}(k).
\end{eqnarray}
For a scaling angle $\theta>0$ with $0<\theta\rho^{*}_{\mathcal{R}(p,q)}\leq\pi$, the anisotropic $\mathcal{R}(p,q)-$sector is
\begin{eqnarray}\label{eq:sector}
	\Sth&:=&\bigcup_{k\geq1}\big\{z\in\mathbb{C}\setminus\{0\}:\ |\arg z|<\theta\,\rho_{\mathcal{R}(p,q)}(k)\big\}
	\;=\;\big\{z\in\mathbb{C}\setminus\{0\}:\ |\arg z|<\theta_{\mathcal{R}(p,q)}\big\},\nn\\
	\theta_{\mathcal{R}(p,q)}&:=&\theta\,\rho^{*}_{\mathcal{R}(p,q)},
\end{eqnarray}
and its critical order is $\omega_{\mathcal{R}(p,q)}(\theta):=\pi/(2\theta_{\mathcal{R}(p,q)})$. If the deformed numbers $\pn{t}=\mathcal{R}(p^{t},q^{t})$ are defined and positive for real $t>0$, as under $\Hphi$, we also consider the curved sector with variable aperture
\begin{eqnarray}\label{eq:curved-sector}
	\Om&:=&\big\{z\in\mathbb{C}\setminus\{0\}:\ |\arg z|<\theta\,\rho_{\mathcal{R}(p,q)}(|z|)\big\},\qquad \rho_{\mathcal{R}(p,q)}(t)\;:=\;\frac{\log\pn{t}}{t},
\end{eqnarray}
for $\theta>0$ with $\theta\sup_{t>0}\rho_{\mathcal{R}(p,q)}(t)\leq\pi$, and its asymptotic critical order $\omega^{\infty}_{\mathcal{R}(p,q)}(\theta):=\pi/(2\theta\rho^{\infty}_{\mathcal{R}(p,q)})$, where $\rho^{\infty}_{\mathcal{R}(p,q)}:=\limsup_{t\to\infty}\rho_{\mathcal{R}(p,q)}(t)$, with $\omega^{\infty}_{\mathcal{R}(p,q)}(\theta)=\infty$ if $\rho^{\infty}_{\mathcal{R}(p,q)}\leq0$.
\end{definition}

The sectors $\{|\arg z|<\theta\rho_{\mathcal{R}(p,q)}(k)\}$ are nested open sectors, empty when $\rho_{\mathcal{R}(p,q)}(k)\leq0$, and their union is the sector of half aperture $\theta$ times the largest logarithmic growth rate of the deformed numbers. The sector $\Sth$ thus depends on the whole sequence $(\pn{k})$, through $\rho^{*}_{\mathcal{R}(p,q)}$, while the curved sector $\Om$ follows the growth rate radius by radius and is governed at infinity by $\rho^{\infty}_{\mathcal{R}(p,q)}$.

\begin{proposition}\label{prop:sector}
Assume \Hpos.
\begin{enumerate}
\item[\textnormal{(i)}] $\limsup_{k}\rho_{\mathcal{R}(p,q)}(k)=\log\beta_{\mathcal{R}(p,q)}$, so that $\rho^{*}_{\mathcal{R}(p,q)}\geq\max\{\log\pn{1},\log\beta_{\mathcal{R}(p,q)}\}$ and $\rho^{*}_{\mathcal{R}(p,q)}<\infty$; in particular $\theta_{\mathcal{R}(p,q)}\geq\theta\log\beta_{\mathcal{R}(p,q)}$ and $\Sth\neq\emptyset$ if $\beta_{\mathcal{R}(p,q)}>1$ or $\pn{k}>1$ for some $k$.
\item[\textnormal{(ii)}] Under $\Hphi$, for real $t\geq1$,
\begin{eqnarray}\label{eq:rho-Hphi}
	\rho_{\mathcal{R}(p,q)}(t)&=&\Big(1-\frac{1}{t}\Big)\log\phi_{1}+\frac{\log c+\log[t]_{\fq}}{t},\qquad 1\leq[t]_{\fq}=\frac{1-\fq^{t}}{1-\fq}<\frac{1}{1-\fq},
\end{eqnarray}
so that $|\rho_{\mathcal{R}(p,q)}(t)-\log\phi_{1}|\leq K/t$ with $K=|\log(c/\phi_{1})|+\log(1/(1-\fq))$, $\rho^{\infty}_{\mathcal{R}(p,q)}=\log\phi_{1}=\log\beta_{\mathcal{R}(p,q)}$, and $\rho^{*}_{\mathcal{R}(p,q)}\geq\log\phi_{1}$.
\item[\textnormal{(iii)}] Under $\Hphi$ with $\phi_{1}>1$ and for every $\delta>0$, $\Om\cap\{|z|\geq K/\delta\}$ is contained in the sector of half aperture $\theta(\log\phi_{1}+\delta)$ and contains $\{|\arg z|<\theta(\log\phi_{1}-\delta),\ |z|\geq K/\delta\}$. If $\phi_{1}<1$, $\Om$ is bounded, and if $\phi_{1}=1$, $\Om\cap\{|z|\geq t\}$ is contained in the sector of half aperture $\theta K/t$.
\end{enumerate}
\end{proposition}
{\it Proof.}
(i) $\rho_{\mathcal{R}(p,q)}(k)=\log(\pn{k}^{1/k})$, whose upper limit is $\log\beta_{\mathcal{R}(p,q)}<\infty$; a sequence with finite upper limit has a finite supremum. (ii) Under $\Hphi$, $\pn{t}=c(\phi_{1}^{t}-\phi_{2}^{t})/(\phi_{1}-\phi_{2})=c\,\phi_{1}^{t-1}[t]_{\fq}$, and $t\mapsto[t]_{\fq}$ increases from $1$ at $t=1$ to $1/(1-\fq)$. (iii) follows from (ii). \hfill\(\Box\)

\begin{center}
\begin{tikzpicture}[scale=1.55]
	\fill[gray!15] (0,0) -- (-40:2.2) arc (-40:40:2.2) -- cycle;
	\fill[gray!38] plot[domain=0.5:2.2,samples=60,smooth,variable=\t] ({\t*cos(29+10/\t-47*exp(-8*(\t-0.5)))},{\t*sin(29+10/\t-47*exp(-8*(\t-0.5)))}) arc[start angle=33.55,end angle=-33.55,radius=2.2] -- plot[domain=2.2:0.5,samples=60,smooth,variable=\t] ({\t*cos(-29-10/\t+47*exp(-8*(\t-0.5)))},{\t*sin(-29-10/\t+47*exp(-8*(\t-0.5)))}) -- cycle;
	\draw[->] (-0.4,0) -- (2.6,0) node[right] {\small$\Re z$};
	\draw[->] (0,-1.5) -- (0,1.5) node[above] {\small$\Im z$};
	\draw[dashed] (0,0) -- (40:2.2);
	\draw[dashed] (0,0) -- (-40:2.2);
	\draw[dotted] (0,0) -- (29:2.4);
	\draw[dotted] (0,0) -- (-29:2.4);
	\draw (0.55,0) arc (0:40:0.55); \node at (20:0.78) {\scriptsize$\theta_{\mathcal{R}(p,q)}$};
	\node at (29:2.62) {\scriptsize$\theta\log\phi_{1}$};
	\node at (1.4,1.45) {\small$\Sth$};
	\node at (1.55,0.3) {\small$\Om$};
\end{tikzpicture}
\captionof{figure}{The anisotropic sector $\Sth$ of half aperture $\theta_{\mathcal{R}(p,q)}=\theta\rho^{*}_{\mathcal{R}(p,q)}$ (light) and the curved sector $\Om$ (dark), whose aperture tends to $\theta\log\phi_{1}$ (dotted rays), under $\Hphi$ with $\phi_{1}>1$.}\label{fig:sector}
\end{center}

In \cite{LankpetreAremuaBukweli2026} the aperture $\theta\rho_{\mathcal{R}(p,q)}(k)$ is written without specifying $k$, and the Phragm\'en Lindel\"of principle is stated with the effective half aperture $\theta\sup_{k}\rho_{\mathcal{R}(p,q)}(k)$; Definition \ref{def:sector} makes both readings precise: the union over $k$ is the sector $\Sth$, whose half aperture is precisely this effective one, and the reading with $k=|z|$ is the curved sector $\Om$. Table \ref{tab:regimes} gives $\rho^{*}_{\mathcal{R}(p,q)}$ and $\rho^{\infty}_{\mathcal{R}(p,q)}$ for the standard deformations.

\subsection{Phragm\'en Lindel\"of principles in anisotropic sectors}\label{subsec:PL}

\begin{theorem}\label{thm:PL-subharm}
Let $0<\theta_{0}\leq\pi$, $S=\{z\neq0:|\arg z|<\theta_{0}\}$, and let $u$ be subharmonic on $S$ with
\begin{enumerate}
\item[\textnormal{(i)}] $\limsup_{z\to\zeta,\,z\in S}u(z)\leq0$ for every $\zeta\in\partial S\setminus\{0\}$;
\item[\textnormal{(ii)}] $u$ bounded above on $S\cap\{|z|<1\}$;
\item[\textnormal{(iii)}] $u(z)\leq C|z|^{\sigma}$ for $z\in S$, $|z|\geq1$, with $C>0$ and $0<\sigma<\pi/(2\theta_{0})$.
\end{enumerate}
Then $u\leq0$ on $S$.
\end{theorem}
{\it Proof.}
Choose $\tau$ with $\sigma<\tau<\pi/(2\theta_{0})$ and put $w(z)=\Re(z^{\tau})=|z|^{\tau}\cos(\tau\arg z)$, with the principal branch, which is harmonic on $S$ and satisfies $w(z)\geq\kappa|z|^{\tau}$, $\kappa=\cos(\tau\theta_{0})>0$. For $\varepsilon>0$, $u_{\varepsilon}=u-\varepsilon w$ is subharmonic, satisfies (i) since $w\geq0$, is bounded above near the vertex by (ii), and $u_{\varepsilon}(z)\leq C|z|^{\sigma}-\varepsilon\kappa|z|^{\tau}\to-\infty$ as $|z|\to\infty$. Choose $\varrho_{\varepsilon}$ with $u_{\varepsilon}\leq0$ on $S\cap\{|z|\geq\varrho_{\varepsilon}\}$. On the bounded domain $S\cap\{|z|<\varrho_{\varepsilon}\}$, $u_{\varepsilon}$ is bounded above and $\limsup u_{\varepsilon}\leq0$ at every boundary point except the vertex, which is a polar set; the extended maximum principle \cite[Theorem 3.6.9]{Ransford1995} gives $u_{\varepsilon}\leq0$. Letting $\varepsilon\to0$ yields $u\leq0$. \hfill\(\Box\)

Hypothesis (ii) cannot be omitted: for $\theta_{0}=\pi/2$, $u(z)=\Re(1/z)$ is harmonic on the right half plane, vanishes on the imaginary axis away from the origin, is bounded for $|z|\geq1$, but is positive in the right half plane.

\begin{lemma}[Asymptotically sectorial sets]\label{lem:PL-asymptotic}
Let $\Omega\subset\mathbb{C}\setminus(-\infty,0]$ be open, $0<\alpha<\pi$ and $t_{0}\geq0$ with $\Omega\cap\{|z|>t_{0}\}\subset\{|\arg z|<\alpha\}$. Let $u$ be subharmonic on $\Omega$, with $\limsup_{z\to\zeta,\,z\in\Omega}u(z)\leq0$ for every finite boundary point $\zeta$ of $\Omega$, and $u(z)\leq C|z|^{\sigma}$ on $\Omega\cap\{|z|\geq1\}$ with $0<\sigma<\pi/(2\alpha)$. Then $u\leq0$ on $\Omega$.
\end{lemma}
{\it Proof.}
Choose $\sigma<\tau<\pi/(2\alpha)$ and put $w(z)=\Re(z^{\tau})$, principal branch, which is harmonic on $\Omega$, satisfies $|w(z)|\leq|z|^{\tau}$, and $w(z)\geq\kappa|z|^{\tau}$, $\kappa=\cos(\tau\alpha)>0$, for $|z|>t_{0}$. For $\varepsilon>0$ the subharmonic function $u_{\varepsilon}=u-\varepsilon w$ satisfies $\limsup u_{\varepsilon}\leq\varepsilon t_{0}^{\tau}$ at every finite boundary point, since $-\varepsilon w\leq\varepsilon t_{0}^{\tau}$ for $|z|\leq t_{0}$ and $-\varepsilon w\leq0$ for $|z|>t_{0}$, and $u_{\varepsilon}(z)\leq C|z|^{\sigma}-\varepsilon\kappa|z|^{\tau}\to-\infty$ as $|z|\to\infty$ in $\Omega$. The maximum principle on the bounded open sets $\Omega\cap D_{\varrho}$, $\varrho$ large, gives $u_{\varepsilon}\leq\varepsilon t_{0}^{\tau}$ on $\Omega$, and $\varepsilon\to0$ gives $u\leq0$. \hfill\(\Box\)

\begin{theorem}[Vector valued principle]\label{thm:PL-vector}
Let $X$ be a Hausdorff locally convex space, $\mathfrak{p}$ a continuous seminorm on $X$, $0<\theta_{0}\leq\pi$ and $S$ as in Theorem \ref{thm:PL-subharm}. Let $F:\overline{S}\to X$ be continuous and weakly holomorphic on $S$, that is $\ell\circ F$ is holomorphic for every $\ell\in X'$, and assume $\mathfrak{p}(F(\zeta))\leq M$ on $\partial S$ and $\mathfrak{p}(F(z))\leq A\exp(C|z|^{\sigma})$ on $S$ with $0<\sigma<\pi/(2\theta_{0})$. Then $\mathfrak{p}(F(z))\leq M$ on $S$, and $\log\mathfrak{p}\circ F$ is subharmonic on $S$. In particular, for $S=\Sth$ the critical order is $\omega_{\mathcal{R}(p,q)}(\theta)$.
\end{theorem}
{\it Proof.}
By the Hahn Banach theorem, $\mathfrak{p}(x)=\sup\{|\ell(x)|:\ell\in U^{\circ}\}$, where $U^{\circ}=\{\ell\in X':|\ell|\leq\mathfrak{p}\}$. For $\ell\in U^{\circ}$, $h=\ell\circ F$ is holomorphic on $S$, continuous on $\overline{S}$, $|h|\leq M$ on $\partial S$ and $|h|\leq A\exp(C|z|^{\sigma})$; Theorem \ref{thm:PL-subharm} applied to $\log|h|-\log M$ gives $|h|\leq M$, and the supremum over $\ell$ gives $\mathfrak{p}(F)\leq M$. The function $\log\mathfrak{p}\circ F=\sup_{\ell\in U^{\circ}}\log|\ell\circ F|$ is a continuous supremum of subharmonic functions, hence subharmonic. \hfill\(\Box\)

\begin{theorem}[Phragm\'en Lindel\"of principle in a topological $\mathcal{R}(p,q)-$Gelfand algebra]\label{thm:PL-gelfand}
Let $\mathscr{A}$ be a topological $\mathcal{R}(p,q)-$Gelfand algebra, $0<\theta_{0}<\pi/2$, $S=\{z\neq0:|\arg z|<\theta_{0}\}$, and let $F:\overline{S}\to\mathscr{A}$ be continuous and weakly holomorphic on $S$, with $\mathfrak{p}_{\alpha}(F(z))\leq C_{\alpha}\exp(A_{\alpha}|z|^{\sigma})$ on $S$ for every $\alpha$, where $0<\sigma<\pi/(2\theta_{0})$. Then:
\begin{enumerate}
\item[\textnormal{(i)}] $\mathfrak{p}_{\alpha}(F(z))\leq M_{\alpha}:=\sup_{\zeta\in\partial S}\mathfrak{p}_{\alpha}(F(\zeta))$ for $z\in S$ and every $\alpha$;
\item[\textnormal{(ii)}] $r(F(z))\leq\sup_{\zeta\in\partial S}r(F(\zeta))$ for $z\in S$;
\item[\textnormal{(iii)}] $\limsup_{|z|\to\infty,\,z\in S}\mathfrak{p}_{\alpha}(F(z))\leq\limsup_{|\zeta|\to\infty,\,\zeta\in\partial S}\mathfrak{p}_{\alpha}(F(\zeta))$.
\end{enumerate}
\end{theorem}
{\it Proof.}
(i) is Theorem \ref{thm:PL-vector} for each $\mathfrak{p}_{\alpha}$. (ii) For $\chi\in\mathrm{Spec}(\mathscr{A})$, $h=\chi\circ F$ is holomorphic, $|h|\leq\mathfrak{p}_{\alpha}(F)\leq C_{\alpha}\exp(A_{\alpha}|z|^{\sigma})$ for the index $\alpha$ of Proposition \ref{prop:gelfand-alg}(i), and $|h|\leq\sup_{\partial S}r(F)$ on $\partial S$; Theorem \ref{thm:PL-subharm} gives $|h(z)|\leq\sup_{\partial S}r(F)$, and one takes the supremum over $\chi$. (iii) is Theorem \ref{thm:lindelof} below. \hfill\(\Box\)

In \cite{LankpetreAremuaBukweli2026} the order condition $\sigma<\pi/(2\theta_{0})$ is not stated, and the passage from the characters to the seminorms uses an inequality $\mathfrak{p}_{\alpha}(a)\leq C_{\alpha}r(a)$, which fails by Proposition \ref{prop:gelfand-alg}(iii); the conclusion (i) nevertheless holds, by the Hahn Banach argument of Theorem \ref{thm:PL-vector}, and the boundary bound has to be taken on the whole boundary, as in (i), or at infinity in the whole sector, as in (iii). For general Banach algebras, where the characters may not determine the spectral radius, the spectral version reads as follows.

\begin{theorem}[Spectral radius principle]\label{thm:PL-spectral}
Let $\mathcal{A}$ be a complex unital Banach algebra with spectral radius $\varrho$, let $0<\theta_{0}\leq\pi$ and $S$ as in Theorem \ref{thm:PL-subharm}, and let $F:\overline{S}\to\mathcal{A}$ be continuous and holomorphic on $S$ with $\varrho(F(\zeta))\leq M$ on $\partial S$ and $\|F(z)\|\leq A\exp(C|z|^{\sigma})$, $0<\sigma<\pi/(2\theta_{0})$. Then $\varrho(F(z))\leq M$ on $S$.
\end{theorem}
{\it Proof.}
If $\mathcal{A}$ is commutative, then $\varrho(a)=\sup_{\chi}|\chi(a)|$ over the characters \cite{Kaniuth2009}, and for each character $h=\chi\circ F$ is holomorphic with $|h|\leq\varrho(F)\leq\|F\|$, $|h|\leq M$ on $\partial S$; Theorem \ref{thm:PL-subharm} gives $|h|\leq M$ and one takes the supremum. In general, $\log\varrho\circ F$ is subharmonic by Vesentini's theorem \cite{Vesentini1968}, \cite[Theorem 3.4.7]{Aupetit1991}; it is bounded above near the vertex, satisfies $\limsup_{z\to\zeta}\log\varrho(F(z))\leq\log\varrho(F(\zeta))\leq\log M$ at boundary points by the upper semicontinuity of $\varrho$, and $\log\varrho(F)\leq\log\|F\|\leq\log A+C|z|^{\sigma}$. Theorem \ref{thm:PL-subharm} applied to $\log\varrho(F)-\log M$ concludes. \hfill\(\Box\)

\begin{theorem}[Anisotropic $\mathcal{R}(p,q)-$Phragm\'en Lindel\"of principle]\label{thm:PL-aniso}
Assume \Hpos. Let $X$ be a Hausdorff locally convex space and $\mathfrak{p}$ a continuous seminorm on $X$.
\begin{enumerate}
\item[\textnormal{(i)}] Let $\theta>0$ with $0<\theta\rho^{*}_{\mathcal{R}(p,q)}\leq\pi$, and let $F:\overline{\Sth}\to X$ be continuous and weakly holomorphic on $\Sth$, with $\mathfrak{p}(F(\zeta))\leq M$ on $\partial\Sth$ and $\mathfrak{p}(F(z))\leq A\exp(C|z|^{\sigma})$ on $\Sth$, where
\begin{eqnarray}\label{eq:aniso-order}
	\sigma&<&\omega_{\mathcal{R}(p,q)}(\theta)\;=\;\frac{\pi}{2\theta\sup_{k}\rho_{\mathcal{R}(p,q)}(k)} .
\end{eqnarray}
Then $\mathfrak{p}(F(z))\leq M$ on $\Sth$, and $\log\mathfrak{p}\circ F$ is subharmonic on $\Sth$.
\item[\textnormal{(ii)}] Assume that $\pn{t}>0$ is defined for real $t>0$, that $\rho_{\mathcal{R}(p,q)}$ is continuous, and let $\theta>0$ with $\theta\sup_{t>0}\rho_{\mathcal{R}(p,q)}(t)<\pi$. Let $F:\overline{\Om}\to X$ be continuous and weakly holomorphic on $\Om$, with $\mathfrak{p}(F(\zeta))\leq M$ at every finite boundary point $\zeta$ of $\Om$ and $\mathfrak{p}(F(z))\leq A\exp(C|z|^{\sigma})$ on $\Om$ with $\sigma<\omega^{\infty}_{\mathcal{R}(p,q)}(\theta)$. Then $\mathfrak{p}(F(z))\leq M$ on $\Om$. Under $\Hphi$, $\omega^{\infty}_{\mathcal{R}(p,q)}(\theta)=\pi/(2\theta\log\phi_{1})$ if $\phi_{1}>1$, and $\omega^{\infty}_{\mathcal{R}(p,q)}(\theta)=\infty$ otherwise.
\end{enumerate}
\end{theorem}
{\it Proof.}
(i) is Theorem \ref{thm:PL-vector} with $\theta_{0}=\theta_{\mathcal{R}(p,q)}=\theta\rho^{*}_{\mathcal{R}(p,q)}$. (ii) Choose $\delta>0$ with $\sigma<\pi/(2\theta(\rho^{\infty}_{\mathcal{R}(p,q)}+\delta))$, or with $\theta\delta<\pi/(2\sigma)$ if $\rho^{\infty}_{\mathcal{R}(p,q)}\leq0$, and $t_{0}$ with $\rho_{\mathcal{R}(p,q)}(t)<\rho^{\infty}_{\mathcal{R}(p,q)}+\delta$ for $t>t_{0}$. Then $\Om\cap\{|z|>t_{0}\}$ lies in the sector of half aperture $\alpha=\theta\max\{\rho^{\infty}_{\mathcal{R}(p,q)}+\delta,\delta\}<\pi$, and $\Om\subset\mathbb{C}\setminus(-\infty,0]$. For $\ell$ in the polar set $U^{\circ}$ of Theorem \ref{thm:PL-vector}, Lemma \ref{lem:PL-asymptotic} applied to $u=\log|\ell\circ F|-\log M$ gives $|\ell\circ F|\leq M$, and the supremum over $\ell$ gives the claim. The values of $\omega^{\infty}_{\mathcal{R}(p,q)}$ follow from Proposition \ref{prop:sector}(ii). \hfill\(\Box\)

The critical order of the curved sector, $\pi/(2\theta\log\phi_{1})$, is at least the critical order $\pi/(2\theta\rho^{*}_{\mathcal{R}(p,q)})$ of $\Sth$, by Proposition \ref{prop:sector}(ii): following the growth rate radius by radius enlarges the class of admissible growths. The critical order of $\Sth$ is sharp.

\begin{proposition}[Sharpness]\label{prop:sharp}
Assume \Hpos, let $0<\theta\rho^{*}_{\mathcal{R}(p,q)}\leq\pi$ and $\omega=\omega_{\mathcal{R}(p,q)}(\theta)$. The function $g(z)=\exp(z^{\omega})$, principal branch, is holomorphic on $\Sth$ and continuous on its closure, satisfies $|g|=1$ on $\partial\Sth$ and $|g(z)|\leq\exp(|z|^{\omega})$, but $g(x)=e^{x^{\omega}}$ is unbounded on the positive axis. Hence the condition \eqref{eq:aniso-order} cannot be replaced by $\sigma\leq\omega_{\mathcal{R}(p,q)}(\theta)$.
\end{proposition}
{\it Proof.}
On $\arg z=\pm\theta_{\mathcal{R}(p,q)}$, $\arg z^{\omega}=\pm\omega\theta_{\mathcal{R}(p,q)}=\pm\pi/2$, so that $\Re z^{\omega}=0$. \hfill\(\Box\)

\begin{theorem}[Vector valued three lines principle]\label{thm:three-lines}
Let $X$ and $\mathfrak{p}$ be as in Theorem \ref{thm:PL-vector}, $a<b$, $\mathbb{B}=\{w:a<\Re w<b\}$, and let $F:\overline{\mathbb{B}}\to X$ be continuous, weakly holomorphic on $\mathbb{B}$, with $\sup_{\overline{\mathbb{B}}}\mathfrak{p}(F)<\infty$. Put $M(x)=\sup_{y\in\mathbb{R}}\mathfrak{p}(F(x+iy))$. Then $\log M$ is convex on $[a,b]$; in particular $M((1-\theta)a+\theta b)\leq M(a)^{1-\theta}M(b)^{\theta}$ for $\theta\in[0,1]$.
\end{theorem}
{\it Proof.}
For $\ell\in U^{\circ}$, $h=\ell\circ F$ is bounded and holomorphic on $\mathbb{B}$, continuous on $\overline{\mathbb{B}}$, and $|h|\leq M(a)$, $|h|\leq M(b)$ on the two boundary lines. Hadamard's three lines theorem \cite{Titchmarsh1939, Rudin1987} gives $|h(x+iy)|\leq M(a)^{1-\theta}M(b)^{\theta}$ for $x=(1-\theta)a+\theta b$, and the supremum over $\ell$ and $y$ gives the inequality; applying it to every substrip gives the convexity of $\log M$. \hfill\(\Box\)

\begin{theorem}[Lindel\"of type principle in an anisotropic sector]\label{thm:lindelof}
Let $X$ and $\mathfrak{p}$ be as in Theorem \ref{thm:PL-vector}, $0<\theta_{0}\leq\pi$, $S=\{z\neq0:|\arg z|<\theta_{0}\}$, for instance $S=\Sth$ with $\theta_{0}=\theta_{\mathcal{R}(p,q)}$, and let $F:\overline{S}\to X$ be continuous, weakly holomorphic on $S$, with $\mathfrak{p}(F(z))\leq A\exp(C|z|^{\sigma})$, $\sigma<\pi/(2\theta_{0})$, and $\mathfrak{p}(F)$ bounded on $\partial S$. Then $\mathfrak{p}(F)$ is bounded on $S$, and if $\mathfrak{p}(F(\zeta))\to L$ as $|\zeta|\to\infty$, $\zeta\in\partial S$, then $\limsup_{|z|\to\infty,\,z\in\overline{S}}\mathfrak{p}(F(z))\leq L$, uniformly in $\overline{S}$. In particular, if $\mathfrak{p}(F)$ tends to $0$ along the two boundary rays, it tends to $0$ uniformly in $\overline{S}$.
\end{theorem}
{\it Proof.}
By Theorem \ref{thm:PL-vector}, $\mathfrak{p}(F)\leq M:=\sup_{\partial S}\mathfrak{p}(F)$ on $S$. Let $\varepsilon>0$ and $R_{\varepsilon}$ with $\mathfrak{p}(F(\zeta))\leq L+\varepsilon$ for $\zeta\in\partial S$, $|\zeta|\geq R_{\varepsilon}$. The map $w=z^{\pi/(2\theta_{0})}$ sends $S$ onto the right half plane $H$, and for $\ell\in U^{\circ}$ the function $v=\log|\ell\circ F(z)|$, as a function of $w$, is subharmonic, bounded above by $\log M$ on $H$, and has boundary values at most $\log M$ on the segment $I=[-iR',iR']$, $R'=R_{\varepsilon}^{\pi/(2\theta_{0})}$, and at most $\log(L+\varepsilon)$ on the rest of the imaginary axis. A subharmonic function bounded above on $H$ is majorized by the Poisson integral of its boundary values \cite{Ransford1995, Koosis1988}, so that
\begin{eqnarray}
	\log|\ell\circ F(z)|&\leq&\log(L+\varepsilon)+\omega_{I}(w)\,\log\frac{M}{L+\varepsilon},\nn
\end{eqnarray}
where $\omega_{I}(w)$, the harmonic measure of $I$ at $w$, is at most $2R'/(\pi(|w|-R'))$ for $|w|>R'$. The bound is uniform in $\ell\in U^{\circ}$; taking the supremum over $\ell$ and letting $|z|\to\infty$, then $\varepsilon\to0$, gives the claim. \hfill\(\Box\)

The conclusion of Theorem \ref{thm:lindelof} cannot be strengthened to $F\equiv0$: for $0<\tau<\pi/(2\theta_{0})$ and $x\in X\setminus\{0\}$, $F(z)=\exp(-z^{\tau})x$ satisfies all the hypotheses, tends to $0$ along $\partial S$, and does not vanish. A rigidity statement requires growth conditions of a different nature, which come from the Gaussian growth of the deformed factorials.

\begin{theorem}[$\mathcal{R}(p,q)-$Phragm\'en Lindel\"of theorem]\label{thm:rigidity}
Assume \Hpos\ and \HG, and let $f(z)=\sum_{n}a_{n}z^{n}$ satisfy $B:=\sup_{n}|a_{n}|\pf{n}r^{n}<\infty$ for some $r>0$, that is, $f$ has finite deformed exponential type. Then:
\begin{enumerate}
\item[\textnormal{(i)}] $f$ is entire and, for every $0<\gamma'<\gamma_{\mathcal{R}(p,q)}$, there is $B'$ such that
\begin{eqnarray}\label{eq:order-zero}
	\log\max_{|z|=\varrho}|f(z)|&\leq&\frac{\big(\log(\varrho/r)\big)^{2}}{4\gamma'}+\log B'+\log\Big(1+\sqrt{\pi/\gamma'}\Big),\qquad \varrho>r;
\end{eqnarray}
in particular $f$ has order zero;
\item[\textnormal{(ii)}] if $f$ is bounded on a ray $\{te^{i\psi}:t\geq0\}$, then $f$ is constant;
\item[\textnormal{(iii)}] if $f(te^{i\psi})\to0$ as $t\to+\infty$ for some $\psi$, then $f\equiv0$.
\end{enumerate}
\end{theorem}
{\it Proof.}
(i) By \HG, $\pf{n}\geq e^{\gamma'n^{2}}$ for $n$ large, so that $|a_{n}|\leq B'r^{-n}e^{-\gamma'n^{2}}$ for all $n$ and a suitable $B'$. With $t=\log(\varrho/r)>0$,
\begin{eqnarray}
	\max_{|z|=\varrho}|f(z)|&\leq&B'\sum_{n\geq0}e^{nt-\gamma'n^{2}}
	\;=\;B'e^{t^{2}/(4\gamma')}\sum_{n\geq0}e^{-\gamma'(n-t/(2\gamma'))^{2}}
	\;\leq\;B'e^{t^{2}/(4\gamma')}\Big(1+\sqrt{\pi/\gamma'}\Big),\nn
\end{eqnarray}
since a sum of values of the unimodal function $x\mapsto e^{-\gamma'(x-c)^{2}}$ at integers is bounded by its maximum plus its integral. Then $\log\log\max_{|z|=\varrho}|f|/\log\varrho\to0$. (ii) Let $h(z)=f(-e^{i\psi}z)$; the ray becomes the negative half axis, which is the boundary of $S=\{|\arg z|<\pi\}$. The function $h$ is entire, bounded by some $M$ on $\partial S$, and of order zero, hence $|h(z)|\leq\exp(|z|^{\sigma})$ for $|z|$ large with $\sigma=1/4<\pi/(2\pi)$. Theorem \ref{thm:PL-subharm} gives $|h|\leq M$ on $S$, hence on $\mathbb{C}$ by continuity, and $h$ is constant by Liouville's theorem. (iii) By (ii), $f$ is constant, and the constant is the limit along the ray. \hfill\(\Box\)

In particular, if $f$ has finite deformed exponential type and tends to $0$ along one boundary ray of an anisotropic sector $\Sth$, then $f\equiv0$; the growth condition $|f(z)|\leq C\exp(A|z|^{a})$ in $\Sth$ required in \cite{LankpetreAremuaBukweli2026} is satisfied for every $a>0$ by (i). This is the rigorous form of the $\mathcal{R}(p,q)-$Phragm\'en Lindel\"of theorem of \cite{LankpetreAremuaBukweli2026}, in which the Gaussian growth $\pf{k}\geq e^{\lambda k^{2}}$, here $\lambda<\gamma_{\mathcal{R}(p,q)}$, appears through the order zero estimate \eqref{eq:order-zero}. Without \HG\ the statement is false, see Remark \ref{rem:counterexample}.

\subsection{Fr\'echet $\mathcal{R}(p,q)-$algebras, nuclear spaces and the holomorphic case}\label{subsec:nuclear}

\begin{definition}[Fr\'echet $\mathcal{R}(p,q)-$algebra and nuclear space of functions]\label{def:frechet-nuclear}
Let $w=(w_{m}(k))_{m\geq1,k\geq0}$ be positive weights with $w_{m}(k)\leq w_{m+1}(k)$ such that, for every $m$,
\begin{eqnarray}\label{eq:weights}
	w_{m}(k)&\geq&C_{m}e^{\lambda_{m}k^{2}}\quad(k\geq0),\qquad\mbox{and}\qquad
	\sum_{k\geq0}\frac{w_{m}(k)}{w_{m'}(k)}\;<\;\infty\ \mbox{ for some } m'>m,
\end{eqnarray}
with constants $C_{m},\lambda_{m}>0$. The space $\mathcal{N}_{w}$ of $\mathcal{R}(p,q)-$analytic functions of type $w$ is the set of power series $f=\sum_{k}a_{k}z^{k}$ with
\begin{eqnarray}\label{eq:nuclear-seminorms}
	\mathfrak{p}_{m}(f)&:=&\sup_{k\geq0}|a_{k}|\,w_{m}(k)\;<\;\infty\qquad\mbox{for every } m .
\end{eqnarray}
A Fr\'echet $\mathcal{R}(p,q)-$algebra is a space $\mathcal{N}_{w}$ which is an algebra for a jointly continuous multiplication.
\end{definition}

\begin{proposition}\label{prop:frechet-nuclear}
Let $w$ be as in Definition \ref{def:frechet-nuclear}.
\begin{enumerate}
\item[\textnormal{(i)}] $\mathcal{N}_{w}$ is a nuclear Fr\'echet space, continuously embedded in $\mathcal{O}(\mathbb{C})$; every $f\in\mathcal{N}_{w}$ is entire of order zero, with $\log\max_{|z|=\varrho}|f(z)|\leq(\log\varrho)^{2}/(4\lambda_{m})+\log\big(\mathfrak{p}_{m}(f)C_{m}^{-1}(1+\sqrt{\pi/\lambda_{m}})\big)$ for $\varrho\geq1$.
\item[\textnormal{(ii)}] If, for every $m$, there are $m'$ and $C$ with $\pn{k+1}\,w_{m}(k)\leq C\,w_{m'}(k+1)$ for all $k$, then $\partial_{\mathcal{R}(p,q)}$ is a continuous endomorphism of $\mathcal{N}_{w}$.
\item[\textnormal{(iii)}] Under $\Hphi$ with $\phi_{1}>1$, the weights $w_{m}(k)=\pf{k}\,m^{k}$ satisfy \eqref{eq:weights} with every $\lambda_{m}<\frac12\log\phi_{1}$ and the condition of \textnormal{(ii)}, and $\mathcal{N}_{w}=\ER{\infty^{-}}$ is the nuclear Fr\'echet algebra of Theorem \ref{thm:frechet-algebra} for the deformed convolution.
\end{enumerate}
\end{proposition}
{\it Proof.}
(i) $\mathcal{N}_{w}$ is the K\"othe echelon space of order infinity with matrix $w$, which is Fr\'echet, and nuclear by the Grothendieck Pietsch criterion, which is the summability condition in \eqref{eq:weights} (Appendix \ref{app:kothe}). Since $|a_{k}|\leq\mathfrak{p}_{m}(f)C_{m}^{-1}e^{-\lambda_{m}k^{2}}$, the computation of the proof of Theorem \ref{thm:rigidity}(i) gives the bound on the maximum modulus, which shows that the embedding into $\mathcal{O}(\mathbb{C})$ is continuous. (ii) $\partial_{\mathcal{R}(p,q)}f=\sum_{k}\pn{k+1}a_{k+1}z^{k}$ and $\pn{k+1}|a_{k+1}|w_{m}(k)\leq C|a_{k+1}|w_{m'}(k+1)\leq C\mathfrak{p}_{m'}(f)$. (iii) $\log\pf{k}=\frac{k^{2}}{2}\log\phi_{1}+O(k)$ gives the Gaussian lower bound, $\sum_{k}(m/m')^{k}<\infty$ for $m'>m$, and $\pn{k+1}\pf{k}m^{k}=m^{-1}\pf{k+1}m^{k+1}$. Finally $\mathfrak{p}_{m}\leq\nur{m}$ and $\nur{r}\leq(1-r/m)^{-1}\mathfrak{p}_{m}$ for $r<m$, so that $\mathcal{N}_{w}=\ER{\infty^{-}}$ with equivalent seminorms (Definition \ref{def:nu-seminorm}, Theorems \ref{thm:borel} and \ref{thm:frechet-algebra}). \hfill\(\Box\)

In \cite{LankpetreAremuaBukweli2026} the seminorms are written $\sup_{k}|a_{k}|/w_{m}(k)$, which, with $w_{m}(k)\geq C_{m}e^{\lambda_{m}k^{2}}$, would allow coefficients of Gaussian growth and divergent series; the weight has to multiply the coefficients, as in \eqref{eq:nuclear-seminorms}, for the space to consist of $\mathcal{R}(p,q)-$analytic functions.

\begin{definition}[Anisotropic $\mathcal{R}(p,q)-$sector in nuclear topology]\label{def:sector-nuclear}
Let $w$ be as in Definition \ref{def:frechet-nuclear} and $(\theta_{m})_{m\geq1}$ positive angles. The anisotropic sector in nuclear topology is
\begin{eqnarray}\label{eq:sector-nuclear}
	\mathbb{S}^{*}_{\mathcal{R}(p,q)}&:=&\Big\{z\in\mathbb{C}\setminus\{0\}:\ \sup_{k}\frac{|z|^{k}}{w_{m}(k)}<\infty\ \mbox{ and }\ |\arg z|<\theta_{m}\ \mbox{ for every } m\Big\}.
\end{eqnarray}
\end{definition}

By the Gaussian lower bound in \eqref{eq:weights}, $\sup_{k}|z|^{k}/w_{m}(k)<\infty$ for every $z$, so that $\mathbb{S}^{*}_{\mathcal{R}(p,q)}=\{z\neq0:|\arg z|<\theta_{*}\}$ with $\theta_{*}=\inf_{m}\theta_{m}$, a sector if $\theta_{*}>0$; the anisotropy is carried by the angles $\theta_{m}$, for instance $\theta_{m}=\theta\rho^{*}_{\mathcal{R}(p,q)}$ for all $m$, which gives $\mathbb{S}^{*}_{\mathcal{R}(p,q)}=\Sth$.

\begin{theorem}[$\mathcal{R}(p,q)-$Phragm\'en Lindel\"of theorem in the Fr\'echet nuclear setting]\label{thm:PL-nuclear}
Let $w$ be as in Definition \ref{def:frechet-nuclear}, $0<\theta_{*}\leq\pi$, and $\mathbb{S}^{*}_{\mathcal{R}(p,q)}$ as above.
\begin{enumerate}
\item[\textnormal{(i)}] Let $F:\overline{\mathbb{S}^{*}_{\mathcal{R}(p,q)}}\to\mathcal{N}_{w}$ be continuous and weakly holomorphic on $\mathbb{S}^{*}_{\mathcal{R}(p,q)}$, with $\mathfrak{p}_{m}(F(z))\leq C_{m}\exp(A_{m}|z|^{a_{m}})$ and $a_{m}<\pi/(2\theta_{*})$ for every $m$. Then $\sup_{\mathbb{S}^{*}_{\mathcal{R}(p,q)}}\mathfrak{p}_{m}(F)=\sup_{\partial\mathbb{S}^{*}_{\mathcal{R}(p,q)}}\mathfrak{p}_{m}(F)$, and if $\mathfrak{p}_{m}(F(\zeta))\to0$ as $|\zeta|\to\infty$ on $\partial\mathbb{S}^{*}_{\mathcal{R}(p,q)}$ for every $m$, then $F(z)\to0$ in $\mathcal{N}_{w}$ as $|z|\to\infty$, uniformly in $\overline{\mathbb{S}^{*}_{\mathcal{R}(p,q)}}$.
\item[\textnormal{(ii)}] If $f\in\mathcal{N}_{w}$ is bounded on a ray, then $f$ is constant, and if $f(z)\to0$ along a ray, in particular along a boundary ray of $\mathbb{S}^{*}_{\mathcal{R}(p,q)}$, then $f\equiv0$.
\item[\textnormal{(iii)}] Under the hypothesis of Proposition \ref{prop:frechet-nuclear}(ii), the conclusions of \textnormal{(ii)} hold for every $\partial^{k}_{\mathcal{R}(p,q)}f$, $k\geq0$: if $\partial^{k}_{\mathcal{R}(p,q)}f\to0$ along a ray, then $f$ is a polynomial of degree less than $k$.
\end{enumerate}
\end{theorem}
{\it Proof.}
(i) is Theorems \ref{thm:PL-vector} and \ref{thm:lindelof} applied to each $\mathfrak{p}_{m}$, the topology of $\mathcal{N}_{w}$ being defined by the seminorms $\mathfrak{p}_{m}$. (ii) By Proposition \ref{prop:frechet-nuclear}(i), $f$ is entire with $\log\max_{|z|=\varrho}|f|=O((\log\varrho)^{2})$, so that $|f(z)|\leq\exp(|z|^{1/4})$ for $|z|$ large; the proof of Theorem \ref{thm:rigidity}(ii) and (iii) applies verbatim. (iii) $\partial^{k}_{\mathcal{R}(p,q)}f\in\mathcal{N}_{w}$ by Proposition \ref{prop:frechet-nuclear}(ii), so that $\partial^{k}_{\mathcal{R}(p,q)}f\equiv0$ by (ii); since $\partial^{k}_{\mathcal{R}(p,q)}z^{n}=\frac{\pf{n}}{\pf{n-k}}z^{n-k}$ with $\pf{n}>0$, this means $a_{n}=0$ for $n\geq k$. \hfill\(\Box\)

The formulation of \cite{LankpetreAremuaBukweli2026} concludes $F\equiv0$ from the boundary decay of all the seminorms under the order condition $a_{m}<1/\lambda_{m}$. For vector valued $F$ this conclusion is not correct, by the example following Theorem \ref{thm:lindelof}, and the order has to be compared with the aperture, as in (i); the Gaussian parameters $\lambda_{m}$ act through the order zero estimate of Proposition \ref{prop:frechet-nuclear}(i), which yields the rigidity (ii) and (iii) for the functions of $\mathcal{N}_{w}$ themselves.

\begin{proposition}[Deformed Taylor formula at the origin]\label{prop:R-taylor}
Assume \Hpos\ and let $f=\sum_{n}a_{n}z^{n}\in\mathcal{O}(D_{\rho})$. Then $\partial^{k}_{\mathcal{R}(p,q)}f\in\mathcal{O}(D_{\rho/\beta^{k}_{\mathcal{R}(p,q)}})$, $(\partial^{k}_{\mathcal{R}(p,q)}f)(0)=\pf{k}a_{k}=\frac{\pf{k}}{k!}f^{(k)}(0)$, and
\begin{eqnarray}\label{eq:R-taylor}
	f(z)&=&\sum_{k\geq0}\frac{(\partial^{k}_{\mathcal{R}(p,q)}f)(0)}{\pf{k}}\,z^{k},\qquad z\in D_{\rho}.
\end{eqnarray}
\end{proposition}
{\it Proof.}
Iterate Theorem \ref{thm:continuity} and \eqref{eq:diag}: $\partial^{k}_{\mathcal{R}(p,q)}f=\sum_{n\geq k}\frac{\pf{n}}{\pf{n-k}}a_{n}z^{n-k}$, whose value at $0$ is $\pf{k}a_{k}$. \hfill\(\Box\)

\begin{corollary}[$\mathcal{R}(p,q)-$identity theorem]\label{cor:identity}
Assume \Hpos. Let $\Omega$ be a connected open set containing $D_{\rho}$, and $f,g$ holomorphic on $\Omega$. The following are equivalent: \textnormal{(a)} $f=g$ on $\Omega$; \textnormal{(b)} $(\partial^{k}_{\mathcal{R}(p,q)}f)(0)=(\partial^{k}_{\mathcal{R}(p,q)}g)(0)$ for every $k$; \textnormal{(c)} $f$ and $g$ coincide on a set having an accumulation point in $\Omega$; \textnormal{(d)} $\partial_{\mathcal{R}(p,q)}f=\partial_{\mathcal{R}(p,q)}g$ on $D_{\rho/\beta_{\mathcal{R}(p,q)}}$ and $f(0)=g(0)$.
\end{corollary}
{\it Proof.}
(b) means that $f$ and $g$ have the same Taylor coefficients at $0$, by Proposition \ref{prop:R-taylor}, and (a), (b), (c) are equivalent by the classical identity theorem. If (d) holds, $h=f-g$ satisfies $\pn{n}h_{n}=0$ for $n\geq1$ and $h_{0}=0$, hence $h=0$ on $D_{\rho}$ since $\pn{n}>0$; conversely (a) implies (d). \hfill\(\Box\)

\begin{corollary}[$\mathcal{R}(p,q)-$maximum modulus principle and open mapping theorem]\label{cor:max-modulus}
Assume \Hpos, let $\Omega$ be a connected open set containing $D_{\rho}$ and $f$ holomorphic on $\Omega$. If $|f|$ attains its supremum over $\Omega$ at a point of $\Omega$, then $f$ is constant, $\partial_{\mathcal{R}(p,q)}f\equiv0$ on $D_{\rho/\beta_{\mathcal{R}(p,q)}}$ and $\partial_{\mathcal{R}(p,q)}f(0)=0$. If $\partial_{\mathcal{R}(p,q)}f\not\equiv0$, then $f(\Omega)$ is open.
\end{corollary}
{\it Proof.}
The first statement is the classical maximum modulus principle, and a constant has vanishing deformed derivative. If $\partial_{\mathcal{R}(p,q)}f\not\equiv0$, then $f$ is not constant by Corollary \ref{cor:identity}, and the classical open mapping theorem applies. \hfill\(\Box\)

The holomorphic case of \cite{LankpetreAremuaBukweli2026} develops $f$ in deformed powers $(z\ominus z_{0})^{k}_{\mathcal{R}(p,q)}$ around an arbitrary point $z_{0}$. Since $\partial_{\mathcal{R}(p,q)}$ involves the dilated points $pz$ and $qz$, the natural centre of the deformed Taylor formula is the fixed point $z_{0}=0$ of the dilations, as in \eqref{eq:R-taylor}; the uniqueness, maximum modulus and open mapping properties are then those of the underlying holomorphic functions, and the deformation enters through the characterization (d) of Corollary \ref{cor:identity} and the vanishing of $\partial_{\mathcal{R}(p,q)}$ exactly on the constants.

\subsection{A generalized Cauchy Hadamard theorem}\label{subsec:CH}

\begin{theorem}\label{thm:CH}
Assume \Hpos. Let $X$ be a Fr\'echet space whose topology is defined by an increasing sequence of seminorms $(\mathfrak{p}_{m})_{m\geq1}$, let $(a_{n})_{n\geq0}$ be a sequence in $X$, and put
\begin{eqnarray}\label{eq:CH-radius}
	\frac{1}{\varrho_{m}}&:=&\limsup_{n\to\infty}\Big(\frac{\mathfrak{p}_{m}(a_{n})}{\pf{n}}\Big)^{1/n},\qquad \varrho\;:=\;\inf_{m\geq1}\varrho_{m}\in[0,\infty].
\end{eqnarray}
Then the deformed exponential series
\begin{eqnarray}\label{eq:CH-series}
	F(z)&=&\sum_{n\geq0}\frac{a_{n}}{\pf{n}}\,z^{n}
\end{eqnarray}
converges absolutely in $X$, uniformly on compact subsets of $D_{\varrho}$, and defines a holomorphic $X-$valued function on $D_{\varrho}$; for $|z|>\varrho$ the terms of \eqref{eq:CH-series} are unbounded in $X$, so that the series diverges. Moreover $\partial_{\mathcal{R}(p,q)}F$ is defined on $D_{\varrho/\beta_{\mathcal{R}(p,q)}}$ and is given by the shifted series $\sum_{n}a_{n+1}z^{n}/\pf{n}$, and $a_{n}=(\partial^{n}_{\mathcal{R}(p,q)}F)(0)$ for every $n$. Finally, if $\mathfrak{p}_{m}(a_{n})\leq C_{m}\lambda_{m}^{n}$ for all $n$, then $\varrho_{m}\geq\kappa_{\mathcal{R}(p,q)}/\lambda_{m}$.
\end{theorem}
{\it Proof.}
For $|z|\leq s<\varrho$, choose $s<s'<\varrho$; for every $m$, $s'<\varrho_{m}$ and $\mathfrak{p}_{m}(a_{n})s'^{n}/\pf{n}$ is bounded, so that $\sum_{n}\mathfrak{p}_{m}(a_{n}z^{n}/\pf{n})\leq C\sum_{n}(s/s')^{n}<\infty$. Completeness of $X$ gives convergence, uniformly on $\overline{D_{s}}$, and the partial sums are polynomials with coefficients in $X$, hence the limit is holomorphic. If $|z|>\varrho$, there is $m$ with $|z|>\varrho_{m}$, and $\mathfrak{p}_{m}(a_{n}z^{n}/\pf{n})$ is unbounded by the definition of $\varrho_{m}$. The coefficients of $F$ are $b_{n}=a_{n}/\pf{n}$, and $\partial_{\mathcal{R}(p,q)}F=\sum_{n}\pn{n+1}b_{n+1}z^{n}=\sum_{n}a_{n+1}z^{n}/\pf{n}$ by \eqref{eq:diag}; its radius is at least $\varrho/\beta_{\mathcal{R}(p,q)}$ by the vector valued version of Theorem \ref{thm:continuity}, whose proof only uses the seminorms of the coefficients. Iterating, $\partial^{n}_{\mathcal{R}(p,q)}F(0)=a_{n}$. If $\mathfrak{p}_{m}(a_{n})\leq C_{m}\lambda_{m}^{n}$, then $\limsup_{n}(\mathfrak{p}_{m}(a_{n})/\pf{n})^{1/n}\leq\lambda_{m}/\liminf_{n}\pf{n}^{1/n}=\lambda_{m}/\kappa_{\mathcal{R}(p,q)}$. \hfill\(\Box\)

\begin{corollary}[Deformed exponentials]\label{cor:exponential}
Assume \Hpos\ and let $\lambda\in\mathbb{C}\setminus\{0\}$. The deformed exponential
\begin{eqnarray}\label{eq:exponential}
	\eR(\lambda z)&:=&\sum_{n\geq0}\frac{\lambda^{n}z^{n}}{\pf{n}}
\end{eqnarray}
has radius of convergence $\kappa_{\mathcal{R}(p,q)}/|\lambda|$, and satisfies $\partial_{\mathcal{R}(p,q)}\eR(\lambda z)=\lambda\,\eR(\lambda z)$ and $\eR(0)=1$. It is the unique solution of this deformed Cauchy problem among the functions holomorphic near the origin. In the Gaussian regime $\kappa_{\mathcal{R}(p,q)}=\infty$, $\eR(\lambda\,\cdot)$ is entire for every $\lambda$.
\end{corollary}
{\it Proof.}
Apply Theorem \ref{thm:CH} with $X=\mathbb{C}$ and $a_{n}=\lambda^{n}$: $\varrho=1/\limsup(|\lambda|^{n}/\pf{n})^{1/n}=\kappa_{\mathcal{R}(p,q)}/|\lambda|$. The eigenvalue equation follows from \eqref{eq:diag}. If $\partial_{\mathcal{R}(p,q)}g=\lambda g$ with $g=\sum b_{n}z^{n}$, then $\pn{n+1}b_{n+1}=\lambda b_{n}$, so that $b_{n}=\lambda^{n}b_{0}/\pf{n}$. \hfill\(\Box\)

In \cite{LankpetreAremuaBukweli2026} the radius $1/\limsup_{k}\|a_{k}/\pf{k}\|^{1/k}$ is attached to the series $\sum_{k}a_{k}z^{k}$; it is the radius of the deformed exponential series \eqref{eq:CH-series}, the radius of $\sum_{k}a_{k}z^{k}$ being given by the classical formula.

\subsection{Standing hypotheses, weighted $\mathcal{R}(p,q)-$norms and Cauchy estimates}\label{subsec:weighted}

Following \cite{LankpetreAremuaBukweli2026}, we consider the hypotheses
\begin{enumerate}
\item[\textnormal{(H1)}] $\pn{n}=\mathcal{R}(p^{n},q^{n})>0$ for every $n\geq1$;
\item[\textnormal{(H2)}] there are constants $C_{1},C_{2}>0$ and $\lambda\leq\Lambda$ with $C_{1}e^{\lambda n}\leq\pn{n}\leq C_{2}e^{\Lambda n}$ for every $n\geq1$;
\item[\textnormal{(H3)}] the coefficients of the power series under consideration are complex numbers.
\end{enumerate}
Since only finitely many indices are concerned, a bound in \textnormal{(H2)} valid for $n\geq n_{0}$ extends to all $n\geq1$ after a change of $C_{1}$ and $C_{2}$. Under \textnormal{(H2)}, $\lambda\leq\log\alpha_{\mathcal{R}(p,q)}\leq\log\beta_{\mathcal{R}(p,q)}\leq\Lambda$, so that \textnormal{(H1)} and \textnormal{(H2)} imply \Hpos, and \textnormal{(H2)} with $\lambda>0$ implies \HG\ with $\gamma_{\mathcal{R}(p,q)}\geq\lambda/2$, by Proposition \ref{prop:indices}. Under $\Hphi$, \textnormal{(H2)} holds with $\lambda=\Lambda=\log\phi_{1}$.

\begin{definition}[Weighted $\mathcal{R}(p,q)-$norms]\label{def:weighted}
Assume \textnormal{(H1)}. For $r>0$ and $f=\sum_{n\geq0}a_{n}z^{n}$ put
\begin{eqnarray}\label{eq:weighted}
	\|f\|_{\mathcal{R}(p,q),r}&:=&|a_{0}|+\sum_{n\geq1}|a_{n}|\,\pn{n}\,r^{n},
\end{eqnarray}
and let $\mathcal{A}_{\mathcal{R}(p,q),r}$ be the Banach space of power series with $\|f\|_{\mathcal{R}(p,q),r}<\infty$.
\end{definition}

The constant term is added because $\pn{0}=\mathcal{R}(1,1)=0$; without it \eqref{eq:weighted} would vanish on the constants and would only be a seminorm. The space $\mathcal{A}_{\mathcal{R}(p,q),r}$ is a weighted $\ell^{1}$ space of coefficients.

\begin{proposition}[Coefficient, supremum and Cauchy estimates]\label{prop:cauchy-R}
Assume \textnormal{(H1)} and let $f\in\mathcal{A}_{\mathcal{R}(p,q),r}$. Then:
\begin{enumerate}
\item[\textnormal{(i)}] $|a_{0}|\leq\|f\|_{\mathcal{R}(p,q),r}$ and $|a_{n}|\leq\|f\|_{\mathcal{R}(p,q),r}/(\pn{n}r^{n})$, $|f^{(n)}(0)|\leq n!\,\|f\|_{\mathcal{R}(p,q),r}/(\pn{n}r^{n})$ for $n\geq1$;
\item[\textnormal{(ii)}] for $0<\varrho<\alpha_{\mathcal{R}(p,q)}r$, $\max_{|z|\leq\varrho}|f(z)|\leq\big(1+\sum_{n\geq1}\frac{(\varrho/r)^{n}}{\pn{n}}\big)\|f\|_{\mathcal{R}(p,q),r}$, and under \textnormal{(H2)}, for $\varrho<re^{\lambda}$,
\begin{eqnarray}\label{eq:sup-disk-est}
	\max_{|z|\leq\varrho}|f(z)|&\leq&\Big(1+\frac{\varrho}{C_{1}(re^{\lambda}-\varrho)}\Big)\|f\|_{\mathcal{R}(p,q),r};
\end{eqnarray}
\item[\textnormal{(iii)}] $\partial_{\mathcal{R}(p,q)}f$ belongs to the Beurling algebra $\mathcal{A}^{+}(r)$ of Appendix \ref{app:beurling}, with $\sum_{n\geq0}|(\partial_{\mathcal{R}(p,q)}f)_{n}|r^{n}\leq\|f\|_{\mathcal{R}(p,q),r}/r$; in particular $\max_{|z|\leq r}|\partial_{\mathcal{R}(p,q)}f(z)|\leq\|f\|_{\mathcal{R}(p,q),r}/r$;
\item[\textnormal{(iv)}] under \textnormal{(H2)}, for $m\geq1$ and $0<\varrho<re^{\lambda}$,
\begin{eqnarray}\label{eq:deriv-disk-est}
	\max_{|z|\leq\varrho}|f^{(m)}(z)|&\leq&\frac{m!\,\|f\|_{\mathcal{R}(p,q),r}}{C_{1}}\,\frac{re^{\lambda}}{(re^{\lambda}-\varrho)^{m+1}} .
\end{eqnarray}
\end{enumerate}
\end{proposition}
{\it Proof.}
(i) Each term of \eqref{eq:weighted} is bounded by the sum, and $f^{(n)}(0)=n!a_{n}$. (ii) $|f(z)|\leq|a_{0}|+\sum_{n\geq1}|a_{n}|\varrho^{n}$, and (i) gives the first bound, the series converging for $\varrho<\alpha_{\mathcal{R}(p,q)}r$ by the root test; under \textnormal{(H2)}, $\sum_{n\geq1}(\varrho/r)^{n}/\pn{n}\leq C_{1}^{-1}\sum_{n\geq1}x^{n}=C_{1}^{-1}x/(1-x)$ with $x=\varrho/(re^{\lambda})$. (iii) The coefficients of $\partial_{\mathcal{R}(p,q)}f$ are $\pn{n+1}a_{n+1}$, and $\sum_{n\geq0}\pn{n+1}|a_{n+1}|r^{n}=r^{-1}\sum_{n\geq1}\pn{n}|a_{n}|r^{n}$. (iv) $|f^{(m)}(z)|\leq\sum_{n\geq m}\frac{n!}{(n-m)!}|a_{n}|\varrho^{n-m}\leq\frac{\|f\|_{\mathcal{R}(p,q),r}}{C_{1}}(re^{\lambda})^{-m}\sum_{j\geq0}\frac{(m+j)!}{j!}x^{j}$ with $x=\varrho/(re^{\lambda})$, and $\sum_{j}\frac{(m+j)!}{j!}x^{j}=m!(1-x)^{-m-1}$. \hfill\(\Box\)

The bounds of \cite{LankpetreAremuaBukweli2026} use $\min_{n\geq0}\pn{n}$, which vanishes because $\pn{0}=0$; the explicit estimates \eqref{eq:sup-disk-est} and \eqref{eq:deriv-disk-est} replace them, and they are valid on discs of radius up to $re^{\lambda}$, larger than $r$ when $\lambda>0$. Statement (iii) shows that the weighted norm is adapted to one application of $\partial_{\mathcal{R}(p,q)}$, but $\partial_{\mathcal{R}(p,q)}$ does not map $\mathcal{A}_{\mathcal{R}(p,q),r}$ into itself unless $(\pn{n})$ is bounded; the factorial weights of the next section, whose recursion $\pf{n}=\pn{n}\pf{n-1}$ absorbs the multiplier of $\partial_{\mathcal{R}(p,q)}$, are invariant.

\section{The factorial weighted scale and the spectral theory of the deformed derivative}\label{sec:scale}

\subsection{The deformed seminorms}\label{subsec:seminorms}

\begin{definition}\label{def:nu-seminorm}
Assume \Hpos. For $f(z)=\sum_{n\geq0}a_{n}z^{n}$ and $r>0$, the $\mathcal{R}(p,q)-$deformed seminorm of order $r$ is
\begin{eqnarray}\label{eq:nu-def}
	\nur{r}(f)&:=&\sum_{n\geq0}|a_{n}|\,\pf{n}\,r^{n}\;\in\;[0,\infty],
\end{eqnarray}
and we put $\ER{r}:=\{f:\nur{r}(f)<\infty\}$ and, for $0<\rho\leq\infty$, $\ER{\rho^{-}}:=\bigcap_{0<r<\rho}\ER{r}$, endowed with the family $(\nur{r})_{0<r<\rho}$.
\end{definition}

\begin{proposition}\label{prop:nu-elementary}
Assume \Hpos. Then:
\begin{enumerate}
\item[\textnormal{(i)}] for $0<r_{1}\leq r_{2}$, $\nur{r_{1}}\leq\nur{r_{2}}$; $\nur{r}$ is a norm on $\ER{r}$ and $(\ER{r},\nur{r})$ is a Banach space, isometrically isomorphic to the weighted space $\ell^{1}(\mathbb{N}\cup\{0\},\pf{n}r^{n})$ through the coefficient map;
\item[\textnormal{(ii)}] if $0<s<r\kappa_{\mathcal{R}(p,q)}$, then $\|f\|_{s}\leq C_{1}(s,r)\,\nur{r}(f)$ with $C_{1}(s,r):=\sup_{n}s^{n}/(\pf{n}r^{n})<\infty$; hence $\ER{r}$ embeds continuously into $\mathcal{O}(D_{r\kappa_{\mathcal{R}(p,q)}})$, and into $\mathcal{O}(\mathbb{C})$ if $\kappa_{\mathcal{R}(p,q)}=\infty$;
\item[\textnormal{(iii)}] if $rK_{\mathcal{R}(p,q)}<s<\rho$, then $\nur{r}(f)\leq C_{2}(r,s)\,\|f\|_{s}$ with $C_{2}(r,s):=\sum_{n}\pf{n}r^{n}s^{-n}<\infty$; hence $\mathcal{O}(D_{\rho})$ embeds continuously into $\ER{r}$ whenever $rK_{\mathcal{R}(p,q)}<\rho$;
\item[\textnormal{(iv)}] $\ER{\rho^{-}}$ is a Fr\'echet space, the topology being defined by $(\nur{r_{k}})_{k}$ for any sequence $r_{k}\uparrow\rho$; if $0<\kappa_{\mathcal{R}(p,q)}=K_{\mathcal{R}(p,q)}=:\kappa<\infty$, then $\ER{\rho^{-}}=\mathcal{O}(D_{\rho\kappa})$ as Fr\'echet spaces.
\end{enumerate}
\end{proposition}
{\it Proof.}
(i) Monotonicity follows from $r_{1}^{n}\leq r_{2}^{n}$; the norm and completeness properties are those of weighted $\ell^{1}$ spaces with positive weights, since $\pf{n}>0$ by \Hpos. (ii) For $|z|\leq s$, $|f(z)|\leq\sum_{n}|a_{n}|\pf{n}r^{n}\cdot s^{n}(\pf{n}r^{n})^{-1}\leq C_{1}(s,r)\nur{r}(f)$, and $C_{1}(s,r)<\infty$ because $(\pf{n}r^{n})^{1/n}>s$ for $n$ large when $s<r\kappa_{\mathcal{R}(p,q)}$. (iii) By the Cauchy inequalities $|a_{n}|\leq\|f\|_{s}s^{-n}$, and the series $C_{2}(r,s)$ converges by the root test since $\limsup_{n}(\pf{n}r^{n}s^{-n})^{1/n}=rK_{\mathcal{R}(p,q)}/s<1$. (iv) The family is increasing and countably generated; completeness follows from (i) by the standard projective limit argument \cite{Treves1967}. If $\kappa_{\mathcal{R}(p,q)}=K_{\mathcal{R}(p,q)}=\kappa$, (ii) and (iii) give mutually continuous inclusions $\ER{r}\subset\mathcal{O}(D_{s})$ for $s<r\kappa$ and $\mathcal{O}(D_{s'})\subset\ER{r}$ for $r\kappa<s'$, whence the identification of the projective limits. \hfill\(\Box\)

\begin{proposition}[Dilation covariance]\label{prop:dilation-covariance}
Assume \Hpos\ and let $s>0$. The dilation $P$ is an isometric isomorphism of $\ER{ps}$ onto $\ER{s}$, and $Q$ is an isometric isomorphism of $\ER{qs}$ onto $\ER{s}$, that is $\nur{s}(Pf)=\nur{ps}(f)$ and $\nur{s}(Qf)=\nur{qs}(f)$. Moreover
\begin{eqnarray}\label{eq:dilation-commutation}
	\partial_{\mathcal{R}(p,q)}P\;=\;p\,P\,\partial_{\mathcal{R}(p,q)},\qquad \partial_{\mathcal{R}(p,q)}Q\;=\;q\,Q\,\partial_{\mathcal{R}(p,q)},\qquad \Bor(Pf)(\lambda)\;=\;\Bor f(p\lambda).
\end{eqnarray}
\end{proposition}
{\it Proof.}
$Pf=\sum a_{n}p^{n}z^{n}$, so that $\nur{s}(Pf)=\sum|a_{n}|\pf{n}(ps)^{n}$, and similarly for $Q$; the inverse is the dilation by $p^{-1}$. On monomials, $\partial_{\mathcal{R}(p,q)}Pz^{n}=p^{n}\pn{n}z^{n-1}=p\,P\partial_{\mathcal{R}(p,q)}z^{n}$, and the Borel transform of $Pf$ has coefficients $a_{n}p^{n}\pf{n}$. \hfill\(\Box\)

The scale $(\ER{r})_{r>0}$ is thus covariant under the dilations which define the deformation, and the commutation relations \eqref{eq:dilation-commutation} are the analytic form of the relations between $A=\partial_{\mathcal{R}(p,q)}$ and the dilation operators $p^{N}$, $q^{N}$ of the quantum algebra.

Under $\Hphi$, Proposition \ref{prop:indices}(iv) shows three different behaviours. In the subexponential factorial regime $\phi_{1}<1$, $K_{\mathcal{R}(p,q)}=0$ and $\mathcal{O}(D_{\rho})\subset\ER{r}$ for all $\rho,r>0$: the factorial weights are weaker than any geometric weight. In the Jackson regime, $\ER{\rho^{-}}=\mathcal{O}(D_{\rho c/(1-\phi_{2})})$. In the Gaussian regime $\phi_{1}>1$, $\kappa_{\mathcal{R}(p,q)}=\infty$ and $\ER{r}$ consists of entire functions, which are of order zero by Theorem \ref{thm:rigidity} below.

\subsection{The deformed Borel transform, its integral representations and nuclearity}\label{subsec:borel}

For $r>0$ we denote by $\mathcal{A}^{+}(r)$ the weighted Wiener space of power series $g(\lambda)=\sum_{n\geq0}c_{n}\lambda^{n}$ with $\|g\|_{\mathcal{A}^{+}(r)}:=\sum_{n}|c_{n}|r^{n}<\infty$; with the ordinary product of power series it is a commutative unital Banach algebra, the Beurling algebra of the semigroup $\mathbb{N}\cup\{0\}$ with weight $r^{n}$ \cite{Kaniuth2009}.

\begin{definition}\label{def:borel}
Assume \Hpos. The deformed Borel transform of $f(z)=\sum_{n}a_{n}z^{n}$ is the formal power series
\begin{eqnarray}\label{eq:borel}
	\Bor f(\lambda)&:=&\sum_{n\geq0}a_{n}\,\pf{n}\,\lambda^{n}\;=\;\sum_{n\geq0}\big(\partial_{\mathcal{R}(p,q)}^{n}f\big)(0)\,\lambda^{n}.
\end{eqnarray}
\end{definition}

The second expression in \eqref{eq:borel} follows from $\partial^{n}_{\mathcal{R}(p,q)}z^{n}=\pf{n}$ and $\partial^{n}_{\mathcal{R}(p,q)}z^{k}(0)=0$ for $k\neq n$. For $f=\eR(\mu\,\cdot)$ one has $\Bor f(\lambda)=(1-\mu\lambda)^{-1}$; thus $\Bor$ inverts the passage from the geometric series to the deformed exponential series, in the same way as the classical Borel transform inverts the passage from $(1-\mu\lambda)^{-1}$ to $e^{\mu z}$.

\begin{theorem}\label{thm:borel}
Assume \Hpos. Then:
\begin{enumerate}
\item[\textnormal{(i)}] for every $r>0$, $\Bor$ is an isometric isomorphism of $(\ER{r},\nur{r})$ onto $\mathcal{A}^{+}(r)$;
\item[\textnormal{(ii)}] $\Bor\circ\partial_{\mathcal{R}(p,q)}=\mathsf{B}\circ\Bor$, where $\mathsf{B}g(\lambda)=(g(\lambda)-g(0))/\lambda$ is the backward shift;
\item[\textnormal{(iii)}] for $0<\rho\leq\infty$, $\Bor$ is an isomorphism of Fr\'echet spaces of $\ER{\rho^{-}}$ onto $\mathcal{O}(D_{\rho})$;
\item[\textnormal{(iv)}] $\ER{\rho^{-}}$ is a nuclear Fr\'echet space; it is the K\"othe echelon space $\lambda^{1}(A)$ of the matrix $A=(\pf{n}r_{k}^{n})_{n,k}$, $r_{k}\uparrow\rho$, which satisfies the Grothendieck Pietsch condition $\sum_{n}\pf{n}r_{k}^{n}/(\pf{n}r_{l}^{n})=\sum_{n}(r_{k}/r_{l})^{n}<\infty$ for $k<l$.
\end{enumerate}
\end{theorem}
{\it Proof.}
(i) $\|\Bor f\|_{\mathcal{A}^{+}(r)}=\sum_{n}|a_{n}|\pf{n}r^{n}=\nur{r}(f)$, and every $g=\sum c_{n}\lambda^{n}\in\mathcal{A}^{+}(r)$ is the image of $f=\sum c_{n}z^{n}/\pf{n}$. (ii) $\partial_{\mathcal{R}(p,q)}f=\sum_{n}\pn{n+1}a_{n+1}z^{n}$, so that the $n-$th coefficient of $\Bor\partial_{\mathcal{R}(p,q)}f$ is $\pn{n+1}a_{n+1}\pf{n}=a_{n+1}\pf{n+1}$, which is the $(n+1)-$th coefficient of $\Bor f$. (iii) A power series $g=\sum c_{n}\lambda^{n}$ belongs to $\bigcap_{r<\rho}\mathcal{A}^{+}(r)$ if and only if its radius of convergence is at least $\rho$, and on $\mathcal{O}(D_{\rho})$ the families $(\|\cdot\|_{\mathcal{A}^{+}(r)})_{r<\rho}$ and $(\|\cdot\|_{r})_{r<\rho}$ are equivalent, since $\|g\|_{r}\leq\|g\|_{\mathcal{A}^{+}(r)}\leq\frac{r'}{r'-r}\|g\|_{r'}$ for $r<r'<\rho$ by the Cauchy inequalities. (iv) By (iii), $\ER{\rho^{-}}$ is isomorphic to $\mathcal{O}(D_{\rho})$, which is nuclear \cite{Grothendieck1955, Pietsch1972}; alternatively, the Grothendieck Pietsch criterion for K\"othe echelon spaces \cite[Theorem 28.15]{MeiseVogt1997} applies to the stated matrix. \hfill\(\Box\)

\begin{remark}\label{rem:nuclear}
Theorem \ref{thm:borel}(iv) shows that the nuclear space naturally attached to the deformation is obtained from the fixed K\"othe matrix $(\pf{n}r_{k}^{n})$; the deformation affects the identification of $\ER{\rho^{-}}$ with a space of holomorphic functions of the variable $z$ (Proposition \ref{prop:nu-elementary}), but not its isomorphism class as a Fr\'echet space, which is always that of $\mathcal{O}(D_{\rho})$. A family of weights depending on further parameters defines a nuclear space only when it is fixed in advance and satisfies the Grothendieck Pietsch condition.
\end{remark}

The deformed Borel transform can be inverted by a contour integral and, under $\Hphi$, represented by a deformed integral.\label{subsec:borel-integral}

\begin{theorem}\label{thm:borel-integral}
Assume \Hpos.
\begin{enumerate}
\item[\textnormal{(i)}] (Inversion) Let $0<\rho\leq\infty$, $f\in\ER{\rho^{-}}$ and $0<\rho'<\rho$. For $|z|<\kappa_{\mathcal{R}(p,q)}\rho'$,
\begin{eqnarray}\label{eq:borel-inversion}
	f(z)&=&\frac{1}{2\pi i}\oint_{|\lambda|=\rho'}\Bor f(\lambda)\;\eR\Big(\frac{z}{\lambda}\Big)\,\frac{d\lambda}{\lambda}.
\end{eqnarray}
\item[\textnormal{(ii)}] (Laplace type representation) Assume $\Hphi$ with $c=1$ and $\phi_{1}\geq1$, let $a_{\mathcal{R}(p,q)}=\phi_{1}/(\phi_{1}-\phi_{2})$, $\fq=\phi_{2}/\phi_{1}$, and let $\ERw(w)=\sum_{n}\phi_{2}^{n(n-1)/2}w^{n}/\pf{n}=(-(1-\fq)w;\fq)_{\infty}$ be the deformed exponential of second kind \eqref{eq:def-exp}. For $r>0$, $f\in\ER{r}$ and $|z|<r$,
\begin{eqnarray}\label{eq:laplace}
	\int_{0}^{a_{\mathcal{R}(p,q)}}f(zt)\,\ERw(-\phi_{2}t)\,d_{\mathcal{R}(p,q)}t&=&\sum_{n\geq0}a_{n}\,\pf{n}\,\phi_{1}^{-\frac{n(n+1)}{2}}\,z^{n},
\end{eqnarray}
where the deformed integral is the Jackson type series $\int_{0}^{a}F\,d_{\mathcal{R}(p,q)}t=(\phi_{1}-\phi_{2})a\sum_{k\geq0}\phi_{2}^{k}\phi_{1}^{-k-1}F(a\phi_{2}^{k}\phi_{1}^{-k-1})$. In particular, in the Jackson regime $\phi_{1}=1$, the left hand side of \eqref{eq:laplace} equals $\Bor f(z)$.
\end{enumerate}
\end{theorem}
{\it Proof.}
(i) Write $\Bor f(\lambda)=\sum_{k}\alpha_{k}\lambda^{k}$ and $\eR(z/\lambda)=\sum_{n}z^{n}\lambda^{-n}/\pf{n}$. On $|\lambda|=\rho'$ the first series converges absolutely since $\rho'<\rho$, and the second one since $|z|/\rho'<\kappa_{\mathcal{R}(p,q)}$ (Corollary \ref{cor:exponential}); the double series converges absolutely and uniformly on the circle, and integration term by term retains the terms $k=n$, which gives $\sum_{n}\alpha_{n}z^{n}/\pf{n}=f(z)$. (ii) The nodes are $t_{k}=\fq^{k}/(\phi_{1}-\phi_{2})$ with weights $\fq^{k}$, and $\ERw(-\phi_{2}t_{k})=(\fq^{k+1};\fq)_{\infty}$. By Euler's identity $\sum_{k}w^{k}/(\fq;\fq)_{k}=1/(w;\fq)_{\infty}$ \cite{GasperRahman},
\begin{eqnarray}
	\int_{0}^{a_{\mathcal{R}(p,q)}}t^{n}\ERw(-\phi_{2}t)\,d_{\mathcal{R}(p,q)}t&=&\frac{(\fq;\fq)_{\infty}}{(\phi_{1}-\phi_{2})^{n}}\sum_{k\geq0}\frac{\fq^{k(n+1)}}{(\fq;\fq)_{k}}
	\;=\;\frac{(\fq;\fq)_{n}}{(\phi_{1}-\phi_{2})^{n}}\;=\;\phi_{1}^{-\frac{n(n+1)}{2}}\,\pf{n},\nn
\end{eqnarray}
since $\pf{n}=\phi_{1}^{n(n-1)/2}(\fq;\fq)_{n}(1-\fq)^{-n}$ and $\phi_{1}-\phi_{2}=\phi_{1}(1-\fq)$. The function $f$ is holomorphic on $D_{r\kappa_{\mathcal{R}(p,q)}}$, and the nodes satisfy $|z|t_{k}\leq|z|/(\phi_{1}-\phi_{2})<r\kappa_{\mathcal{R}(p,q)}$, because $\kappa_{\mathcal{R}(p,q)}=\infty$ if $\phi_{1}>1$ and $\kappa_{\mathcal{R}(p,q)}=1/(1-\phi_{2})$ if $\phi_{1}=1$. The termwise integration is justified by $\sum_{n}|a_{n}|\pf{n}\phi_{1}^{-n(n+1)/2}|z|^{n}\leq\nur{r}(f)$ for $\phi_{1}\geq1$ and $|z|<r$. \hfill\(\Box\)

Formula \eqref{eq:borel-inversion} is the deformed form of the Borel Polya inversion, the exponential kernel $e^{z/\lambda}$ being replaced by the deformed exponential series $\eR(z/\lambda)$, and \eqref{eq:laplace} shows that, in the Jackson regime, the deformed Borel transform is the $\fq-$Laplace transform built on the Jackson integral, while in the Gaussian regime the same integral produces the Borel transform twisted by the Gaussian factor $\phi_{1}^{-n(n+1)/2}$.

\subsection{Norm and spectrum of the deformed derivative}\label{subsec:spectrum}

\begin{theorem}\label{thm:spectrum}
Assume \Hpos\ and let $r>0$. Then $\partial_{\mathcal{R}(p,q)}$ is a bounded operator on $\ER{r}$, and:
\begin{enumerate}
\item[\textnormal{(i)}] $\nur{r}(\partial_{\mathcal{R}(p,q)}f)=r^{-1}\big(\nur{r}(f)-|f(0)|\big)$ for every $f\in\ER{r}$; in particular $\|\partial_{\mathcal{R}(p,q)}\|=1/r$ and $\|\partial^{k}_{\mathcal{R}(p,q)}\|=r^{-k}$;
\item[\textnormal{(ii)}] the spectrum of $\partial_{\mathcal{R}(p,q)}$ is the closed disc $\{|\lambda|\leq1/r\}$;
\item[\textnormal{(iii)}] for $|\lambda|<1/r$, $\lambda$ is an eigenvalue, the eigenspace is spanned by $\eR(\lambda z)$, and $\nur{r}(\eR(\lambda\,\cdot))=(1-|\lambda|r)^{-1}$;
\item[\textnormal{(iv)}] for $|\lambda|=1/r$, $\partial_{\mathcal{R}(p,q)}-\lambda$ is injective with dense range, but not surjective;
\item[\textnormal{(v)}] for $|\lambda|>1/r$, $(\partial_{\mathcal{R}(p,q)}-\lambda)^{-1}=-\sum_{k\geq0}\lambda^{-k-1}\partial^{k}_{\mathcal{R}(p,q)}$, the series converging in operator norm.
\end{enumerate}
\end{theorem}
{\it Proof.}
(i) By Theorem \ref{thm:borel}(ii), $\nur{r}(\partial_{\mathcal{R}(p,q)}f)=\sum_{n\geq0}|a_{n+1}|\pf{n+1}r^{n}=r^{-1}\sum_{n\geq1}|a_{n}|\pf{n}r^{n}$, and the value $|a_{0}|=|f(0)|$ is the missing term. The norm equals $1/r$, as shown by $f=z^{n}$, $n\geq1$, and similarly for the powers. (ii), (iii), (iv), (v): by Theorem \ref{thm:borel}, $\partial_{\mathcal{R}(p,q)}$ on $\ER{r}$ is isometrically equivalent to $\mathsf{B}$ on $\mathcal{A}^{+}(r)$, and the rescaling $c_{n}\mapsto c_{n}r^{n}$ identifies $\mathsf{B}$ on $\mathcal{A}^{+}(r)$ with $r^{-1}S^{*}$ on $\ell^{1}$, where $S^{*}(\xi_{0},\xi_{1},\dots)=(\xi_{1},\xi_{2},\dots)$. The spectral radius of $\mathsf{B}$ is $\lim_{k}\|\mathsf{B}^{k}\|^{1/k}=1/r$. For $|\lambda|<1/r$, $\mathsf{B}g=\lambda g$ if and only if $c_{n+1}=\lambda c_{n}$, that is $g=c_{0}(1-\lambda\,\cdot)^{-1}$, which belongs to $\mathcal{A}^{+}(r)$ with norm $|c_{0}|(1-|\lambda|r)^{-1}$; its preimage under $\Bor$ is $c_{0}\eR(\lambda z)$. Since the spectrum is closed and contains the open disc, it is the closed disc. For $|\lambda|=1/r$ the geometric sequence is not summable, so that $\mathsf{B}-\lambda$ is injective. The dual of $\mathcal{A}^{+}(r)$ is the space of sequences $(y_{n})$ with $\sup_{n}|y_{n}|r^{-n}<\infty$, the adjoint of $\mathsf{B}$ is $(y_{0},y_{1},\dots)\mapsto(0,y_{0},y_{1},\dots)$, and $\mathsf{B}'y=\lambda y$ with $\lambda\neq0$ forces $y=0$; hence the range of $\mathsf{B}-\lambda$ is dense. It is not surjective, otherwise $\mathsf{B}-\lambda$ would be invertible by the open mapping theorem. Finally (v) is the Neumann series, convergent since $\|\lambda^{-k-1}\partial^{k}_{\mathcal{R}(p,q)}\|=|\lambda|^{-1}(|\lambda|r)^{-k}$. \hfill\(\Box\)

\begin{theorem}[Spectral picture on entire functions]\label{thm:entire-spectrum}
Assume $\Hphi$, and consider $\partial_{\mathcal{R}(p,q)}$ as a continuous operator on the Fr\'echet space $\mathcal{O}(\mathbb{C})$ (Theorem \ref{thm:continuity}(ii)).
\begin{enumerate}
\item[\textnormal{(i)}] If $\phi_{1}>1$, then for every $\lambda\in\mathbb{C}$ the kernel of $\partial_{\mathcal{R}(p,q)}-\lambda$ is $\mathbb{C}\,\eR(\lambda z)$ and $\partial_{\mathcal{R}(p,q)}-\lambda$ is surjective on $\mathcal{O}(\mathbb{C})$.
\item[\textnormal{(ii)}] If $\phi_{1}\leq1$, then $\partial_{\mathcal{R}(p,q)}-\lambda$ is an isomorphism of $\mathcal{O}(\mathbb{C})$ onto itself for every $\lambda\neq0$, and $\partial_{\mathcal{R}(p,q)}$ is surjective with kernel the constants.
\end{enumerate}
\end{theorem}
{\it Proof.}
For $g=\sum b_{n}z^{n}$ and $f=\sum a_{n}z^{n}$, the equation $\partial_{\mathcal{R}(p,q)}f-\lambda f=g$ reads $\pn{n+1}a_{n+1}=\lambda a_{n}+b_{n}$, that is $\alpha_{n+1}=\lambda\alpha_{n}+\gamma_{n}$ with $\alpha_{n}=a_{n}\pf{n}$ and $\gamma_{n}=b_{n}\pf{n}$. The homogeneous solutions are the multiples of $\alpha_{n}=\lambda^{n}$, that is of $\eR(\lambda z)$, which is entire if and only if $\kappa_{\mathcal{R}(p,q)}=\infty$, that is $\phi_{1}>1$, or $\lambda=0$. Let $g$ be entire, so that $|b_{k}|\leq C_{R}R^{-k}$ for every $R>0$.

(i) Take $\alpha_{n}=\sum_{k<n}\lambda^{n-1-k}\gamma_{k}$, so that $|a_{n}|\leq\sum_{k<n}|\lambda|^{n-1-k}|b_{k}|\pf{k}/\pf{n}$. Under $\Hphi$, $\pn{j}=c\phi_{1}^{j-1}[j]_{\fq}\geq c\phi_{1}^{j-1}$, hence $\pf{k}/\pf{n}\leq c^{-(n-k)}\phi_{1}^{-(n-k)(n+k-1)/2}$ for $k<n$. Fix $\delta\in(0,1)$. Fix also $R\geq1$. The terms with $n-k\geq\delta n$ are bounded by $C_{R}\max\{1,|\lambda|/c\}^{n}\max\{1,R^{-1}\}^{n}\phi_{1}^{-\delta n(n-1)/2}$, and the terms with $n-k<\delta n$ by $C_{R}\max\{1,|\lambda|/c\}^{\delta n}R^{-(1-\delta)n}$. Summing at most $n$ terms, $\limsup_{n}|a_{n}|^{1/n}\leq\max\{1,|\lambda|/c\}^{\delta}R^{-(1-\delta)}$, which is arbitrarily small. Hence $f$ is entire.

(ii) For $\lambda\neq0$ take $\alpha_{n}=-\sum_{k\geq n}\lambda^{n-1-k}\gamma_{k}$, which satisfies $\alpha_{n+1}-\lambda\alpha_{n}=\gamma_{n}$ as a direct computation shows, provided the series converge. Under $\Hphi$ with $\phi_{1}\leq1$, $\pn{j}=c\,\phi_{1}^{j-1}[j]_{\fq}\leq L:=c/(1-\fq)$ for all $j$, so that $\pf{k}/\pf{n}=\prod_{j=n+1}^{k}\pn{j}\leq L^{k-n}$ for $k\geq n$, and
$|a_{n}|\leq\sum_{k\geq n}|\lambda|^{n-1-k}C_{R}R^{-k}L^{k-n}=C_{R}|\lambda|^{-1}R^{-n}\sum_{m\geq0}(L/(|\lambda|R))^{m}$, which is finite for $R>L/|\lambda|$; since $R$ is arbitrary, $f$ is entire. The kernel is trivial because $\eR(\lambda z)$ is not entire, and the open mapping theorem for Fr\'echet spaces gives the continuity of the inverse. For $\lambda=0$, $\mathcal{I}_{\mathcal{R}(p,q)}$ of Theorem \ref{thm:integral-operator} gives a right inverse on $\mathcal{O}(\mathbb{C})$, since $\alpha_{\mathcal{R}(p,q)}=\phi_{1}$ and $\mathcal{I}_{\mathcal{R}(p,q)}$ maps $\mathcal{O}(\mathbb{C})$ into itself. \hfill\(\Box\)

Theorem \ref{thm:entire-spectrum} contrasts with Theorem \ref{thm:spectrum}: on the Banach spaces $\ER{r}$ the spectrum of the deformed derivative is a disc in all regimes, whereas on the entire functions the Gaussian regime makes every complex number an eigenvalue, as for the classical derivative, while in the other regimes the deformed derivative has no nonzero eigenvalue on $\mathcal{O}(\mathbb{C})$.

\begin{proposition}[Nuclearity of the deformed derivative between members of the scale]\label{prop:nuclear-derivative}
Assume \Hpos\ and let $0<r<s$. The inclusion $\ER{s}\hookrightarrow\ER{r}$ and the operator $\partial_{\mathcal{R}(p,q)}:\ER{s}\to\ER{r}$ are nuclear, with nuclear norms at most $s/(s-r)$ and $1/(s-r)$ respectively.
\end{proposition}
{\it Proof.}
In Borel coordinates the inclusion is $\sum_{n}\alpha_{n}\lambda^{n}\mapsto\sum_{n}\alpha_{n}\lambda^{n}$ from $\mathcal{A}^{+}(s)$ to $\mathcal{A}^{+}(r)$, that is $\sum_{n}\ell_{n}(u)\,\lambda^{n}$ with $\ell_{n}(u)=\alpha_{n}$, $\|\ell_{n}\|=s^{-n}$ and $\|\lambda^{n}\|_{\mathcal{A}^{+}(r)}=r^{n}$; hence the nuclear norm is at most $\sum_{n}(r/s)^{n}=s/(s-r)$. For the derivative, $\mathsf{B}u=\sum_{n}\alpha_{n+1}\lambda^{n}$ with $\|\ell_{n+1}\|\,\|\lambda^{n}\|_{\mathcal{A}^{+}(r)}=s^{-n-1}r^{n}$, whose sum is $s^{-1}\cdot s/(s-r)=1/(s-r)$. \hfill\(\Box\)

\begin{corollary}[Deformed Cauchy problem]\label{cor:cauchy-problem}
Assume \Hpos, let $r>0$, $g\in\ER{r}$, $c\in\mathbb{C}$ and $|\lambda|<1/r$. The deformed Cauchy problem
\begin{eqnarray}\label{eq:cauchy-problem}
	\partial_{\mathcal{R}(p,q)}f&=&\lambda f+g,\qquad f(0)\;=\;c,
\end{eqnarray}
has a unique solution $f\in\ER{r}$; its deformed Borel transform is $\Bor f(\mu)=(c+\mu\,\Bor g(\mu))/(1-\lambda\mu)$, and $\nur{r}(f)\leq(|c|+r\,\nur{r}(g))/(1-|\lambda|r)$. For $|\lambda|>1/r$, the equation $\partial_{\mathcal{R}(p,q)}f=\lambda f+g$ has the unique solution $f=(\partial_{\mathcal{R}(p,q)}-\lambda)^{-1}g$ in $\ER{r}$, whose value at the origin is determined by $g$.
\end{corollary}
{\it Proof.}
Write $\alpha_{n}$ and $\gamma_{n}$ for the coefficients of $\Bor f$ and $\Bor g$. By Theorem \ref{thm:borel}(ii), \eqref{eq:cauchy-problem} is equivalent to $\alpha_{n+1}=\lambda\alpha_{n}+\gamma_{n}$, $\alpha_{0}=c$, whose unique solution is $\alpha_{n}=\lambda^{n}c+\sum_{k<n}\lambda^{n-1-k}\gamma_{k}$; summing the corresponding geometric series gives the generating function and
$\sum_{n}|\alpha_{n}|r^{n}\leq|c|(1-|\lambda|r)^{-1}+\sum_{k}|\gamma_{k}|r^{k+1}(1-|\lambda|r)^{-1}$. The second assertion is Theorem \ref{thm:spectrum}(v). \hfill\(\Box\)

\begin{corollary}[Deformed Taylor formula]\label{cor:taylor}
Assume \Hpos. Every $f\in\ER{r}$ satisfies
\begin{eqnarray}\label{eq:taylor}
	f(z)&=&\sum_{n\geq0}\frac{(\partial^{n}_{\mathcal{R}(p,q)}f)(0)}{\pf{n}}\,z^{n},
\end{eqnarray}
the series converging in $\ER{r}$, and $\nur{r}(f)=\sum_{n}|(\partial^{n}_{\mathcal{R}(p,q)}f)(0)|\,r^{n}$.
\end{corollary}
{\it Proof.}
Since $(\partial^{n}_{\mathcal{R}(p,q)}f)(0)=a_{n}\pf{n}$, \eqref{eq:taylor} is the power series of $f$, and its partial sums converge in the weighted $\ell^{1}$ norm. \hfill\(\Box\)

\subsection{The deformed integral operator and the duality of the scale}\label{subsec:integral-duality}

\begin{definition}\label{def:integral-operator}
Assume \Hpos. The deformed integral operator is the diagonal operator
\begin{eqnarray}\label{eq:integral-operator}
	\mathcal{I}_{\mathcal{R}(p,q)}\Big(\sum_{n\geq0}a_{n}z^{n}\Big)&:=&\sum_{n\geq0}\frac{a_{n}}{\pn{n+1}}\,z^{n+1}.
\end{eqnarray}
\end{definition}

Under $\Hphi$ with $c=1$, $\mathcal{I}_{\mathcal{R}(p,q)}z^{n}=z^{n+1}/\pn{n+1}$ is the value at $z$ of the Jackson type $\mathcal{R}(p,q)-$integral $\int_{0}^{z}t^{n}\,d_{\mathcal{R}(p,q)}t=z\,(\phi_{1}-\phi_{2})\sum_{k\geq0}\phi_{2}^{k}\phi_{1}^{-k-1}(z\phi_{2}^{k}\phi_{1}^{-k-1})^{n}$, so that $\mathcal{I}_{\mathcal{R}(p,q)}f(z)=\int_{0}^{z}f(t)\,d_{\mathcal{R}(p,q)}t$ for $f$ holomorphic on a disc containing $z/\phi_{1}$; this is the deformed Volterra operator of the $(p,q)-$calculus \cite{ref58, SoontharanonSitthiwirattham2020}.

\begin{theorem}\label{thm:integral-operator}
Assume \Hpos. Then:
\begin{enumerate}
\item[\textnormal{(i)}] $\partial_{\mathcal{R}(p,q)}\mathcal{I}_{\mathcal{R}(p,q)}f=f$ and $\mathcal{I}_{\mathcal{R}(p,q)}\partial_{\mathcal{R}(p,q)}f=f-f(0)$ for every power series $f$;
\item[\textnormal{(ii)}] $\mathcal{I}_{\mathcal{R}(p,q)}$ maps $\mathcal{O}(D_{\rho})$ continuously into $\mathcal{O}(D_{\alpha_{\mathcal{R}(p,q)}\rho})$, the radius being optimal, with $\alpha_{\mathcal{R}(p,q)}\rho=\infty$ if $\alpha_{\mathcal{R}(p,q)}=\infty$;
\item[\textnormal{(iii)}] on $\ER{r}$, $\nur{r}(\mathcal{I}_{\mathcal{R}(p,q)}f)=r\,\nur{r}(f)$ and $\Bor(\mathcal{I}_{\mathcal{R}(p,q)}f)(\lambda)=\lambda\,\Bor f(\lambda)$; hence $r^{-1}\mathcal{I}_{\mathcal{R}(p,q)}$ is an isometry, $\|\mathcal{I}^{k}_{\mathcal{R}(p,q)}\|=r^{k}$, the spectrum of $\mathcal{I}_{\mathcal{R}(p,q)}$ is the closed disc $\{|\mu|\leq r\}$ and its point spectrum is empty;
\item[\textnormal{(iv)}] for $|\mu|<r$, the range of $\mathcal{I}_{\mathcal{R}(p,q)}-\mu$ is the closed hyperplane $\ker\chi_{\mu}=\{f\in\ER{r}:\Bor f(\mu)=0\}$, where $\chi_{\mu}$ is the character of Theorem \ref{thm:gelfand}.
\end{enumerate}
\end{theorem}
{\it Proof.}
(i) follows from \eqref{eq:diag} on monomials. (ii) The multiplier $1/\pn{n+1}$ has $\limsup_{n}(1/\pn{n+1})^{1/n}=1/\alpha_{\mathcal{R}(p,q)}$; apply Lemma \ref{lem:multiplier} to the shifted series, the factor $z$ being harmless. (iii) The $(n+1)-$th coefficient of $\Bor\mathcal{I}_{\mathcal{R}(p,q)}f$ is $a_{n}\pf{n+1}/\pn{n+1}=a_{n}\pf{n}$, which is the $n-$th coefficient of $\Bor f$; hence $\Bor\mathcal{I}_{\mathcal{R}(p,q)}=M_{\lambda}\Bor$, where $M_{\lambda}$ is the multiplication by $\lambda$ on $\mathcal{A}^{+}(r)$, and $\|M_{\lambda}u\|_{\mathcal{A}^{+}(r)}=r\|u\|_{\mathcal{A}^{+}(r)}$. The operator $M_{\lambda}-\mu$ is the multiplication by $\lambda-\mu$, which is invertible in the Banach algebra $\mathcal{A}^{+}(r)$ if and only if $\lambda-\mu$ does not vanish on $\overline{D_{r}}$ (Theorem \ref{thm:gelfand}(iii)), that is $|\mu|>r$; it is injective since $\mathcal{A}^{+}(r)$ has no zero divisors. (iv) If $u=(\lambda-\mu)v$ then $u(\mu)=0$. Conversely, if $u=\sum c_{k}\lambda^{k}\in\mathcal{A}^{+}(r)$ and $u(\mu)=0$, then $v=u/(\lambda-\mu)$ has coefficients $v_{n}=\sum_{k>n}c_{k}\mu^{k-n-1}$, and
$\sum_{n}|v_{n}|r^{n}\leq\sum_{k}|c_{k}|\sum_{n<k}|\mu|^{k-n-1}r^{n}\leq\sum_{k}|c_{k}|r^{k}/(r-|\mu|)<\infty$. \hfill\(\Box\)

The pairing which makes the deformed derivative transparent is the deformed Fischer pairing
\begin{eqnarray}\label{eq:pairing}
	\langle f,g\rangle_{\mathcal{R}(p,q)}&:=&\sum_{n\geq0}a_{n}\,b_{n}\,\pf{n},\qquad f=\sum a_{n}z^{n},\quad g=\sum b_{n}z^{n}.
\end{eqnarray}

\begin{theorem}[Duality]\label{thm:duality}
Assume \Hpos.
\begin{enumerate}
\item[\textnormal{(i)}] For $r>0$, the pairing \eqref{eq:pairing} identifies the dual of $\ER{r}$ with the Banach space $\mathcal{H}^{\infty}(r):=\{g=\sum b_{n}z^{n}:\sup_{n}|b_{n}|r^{-n}<\infty\}$, isometrically.
\item[\textnormal{(ii)}] For $0<\rho\leq\infty$, the dual of $\ER{\rho^{-}}$ is identified with the space of power series $g$ with $\limsup_{n}|b_{n}|^{1/n}<\rho$, that is with the germs of holomorphic functions on the closed disc $\overline{D_{1/\rho}}$, with $\overline{D_{0}}=\{0\}$.
\item[\textnormal{(iii)}] For $f\in\ER{r}$ and $g\in\mathcal{H}^{\infty}(r)$ such that the pairings converge,
\begin{eqnarray}\label{eq:transpose}
	\langle\partial_{\mathcal{R}(p,q)}f,g\rangle_{\mathcal{R}(p,q)}\;=\;\langle f,zg\rangle_{\mathcal{R}(p,q)},\qquad
	\langle\mathcal{I}_{\mathcal{R}(p,q)}f,g\rangle_{\mathcal{R}(p,q)}\;=\;\Big\langle f,\frac{g-g(0)}{z}\Big\rangle_{\mathcal{R}(p,q)} .
\end{eqnarray}
\end{enumerate}
In other words, the transpose of the deformed derivative is the multiplication by $z$, and the transpose of the deformed integral operator is the classical backward shift: the deformation disappears in the dual picture.
\end{theorem}
{\it Proof.}
(i) By Theorem \ref{thm:borel}(i), $\ER{r}\cong\mathcal{A}^{+}(r)\cong\ell^{1}(r^{n})$, whose dual is $\ell^{\infty}(r^{-n})$ through $\sum_{n}\alpha_{n}b_{n}$; with $\alpha_{n}=a_{n}\pf{n}$ this is \eqref{eq:pairing}. (ii) The dual of the projective limit is the inductive limit of the duals \cite{Treves1967, MeiseVogt1997}, here $\bigcup_{r<\rho}\mathcal{H}^{\infty}(r)$, which is the set of power series with $\limsup_{n}|b_{n}|^{1/n}<\rho$, i.e. convergent on a disc of radius larger than $1/\rho$. (iii) $\langle\partial_{\mathcal{R}(p,q)}f,g\rangle_{\mathcal{R}(p,q)}=\sum_{n}\pn{n+1}a_{n+1}b_{n}\pf{n}=\sum_{n}a_{n+1}b_{n}\pf{n+1}=\langle f,zg\rangle_{\mathcal{R}(p,q)}$, and $\langle\mathcal{I}_{\mathcal{R}(p,q)}f,g\rangle_{\mathcal{R}(p,q)}=\sum_{n}a_{n}b_{n+1}\pf{n+1}/\pn{n+1}=\sum_{n}a_{n}b_{n+1}\pf{n}$. \hfill\(\Box\)

\subsection{Deformed linear equations with constant coefficients and deformed translations}\label{subsec:constant-coefficients}

\begin{theorem}\label{thm:constant-coefficients}
Assume \Hpos, let $r>0$, and let $P(\lambda)=c_{P}\prod_{i=1}^{s}(\lambda-\lambda_{i})^{m_{i}}$ be a polynomial of degree $d=\sum m_{i}\geq1$ with distinct roots $\lambda_{i}$, none of which lies on the circle $|\lambda|=1/r$. Put $m_{<}=\sum_{|\lambda_{i}|<1/r}m_{i}$. Then $P(\partial_{\mathcal{R}(p,q)})$ is a surjective Fredholm operator on $\ER{r}$ of index $m_{<}$, and its kernel is spanned by the functions
\begin{eqnarray}\label{eq:generalized-eigen}
	f_{i,j}(z)&:=&\frac{\partial^{j}}{\partial\lambda^{j}}\eR(\lambda z)\Big|_{\lambda=\lambda_{i}}\;=\;\sum_{n\geq j}\frac{n!}{(n-j)!}\,\frac{\lambda_{i}^{\,n-j}\,z^{n}}{\pf{n}},\qquad |\lambda_{i}|<\frac1r,\quad 0\leq j<m_{i}.
\end{eqnarray}
For every $g\in\ER{r}$ and every $(c_{0},\dots,c_{m_{<}-1})\in\mathbb{C}^{m_{<}}$, the equation $P(\partial_{\mathcal{R}(p,q)})f=g$ has a unique solution $f\in\ER{r}$ with $(\partial^{k}_{\mathcal{R}(p,q)}f)(0)=c_{k}$ for $0\leq k<m_{<}$.
\end{theorem}
{\it Proof.}
By Theorem \ref{thm:borel}, $P(\partial_{\mathcal{R}(p,q)})$ is equivalent to $P(\mathsf{B})=c_{P}\prod_{i}(\mathsf{B}-\lambda_{i})^{m_{i}}$ on $\mathcal{A}^{+}(r)$. By Theorem \ref{thm:spectrum} and Corollary \ref{cor:cauchy-problem}, $\mathsf{B}-\lambda$ is invertible for $|\lambda|>1/r$, and for $|\lambda|<1/r$ it is surjective with one dimensional kernel, hence Fredholm of index $1$. The product is therefore surjective and Fredholm of index $m_{<}$, so that its kernel has dimension $m_{<}$. Differentiating the identity $(\partial_{\mathcal{R}(p,q)}-\lambda)\eR(\lambda z)=0$ $j$ times with respect to $\lambda$ gives $(\partial_{\mathcal{R}(p,q)}-\lambda)\partial^{j}_{\lambda}\eR(\lambda z)=j\,\partial^{j-1}_{\lambda}\eR(\lambda z)$, whence $(\partial_{\mathcal{R}(p,q)}-\lambda_{i})^{j+1}f_{i,j}=0$; moreover $\nur{r}(f_{i,j})=\sum_{n\geq j}\frac{n!}{(n-j)!}|\lambda_{i}|^{n-j}r^{n}<\infty$. The $m_{<}$ functions $f_{i,j}$ are linearly independent, as their Borel transforms $j!\lambda^{j}(1-\lambda_{i}\lambda)^{-j-1}$ are, and therefore span the kernel. Finally, in Borel coordinates the kernel consists of the sequences $(\alpha_{n})$ satisfying the linear recurrence of order $m_{<}$ attached to $\prod_{|\lambda_{i}|<1/r}(\mathsf{B}-\lambda_{i})^{m_{i}}$, the other factors being invertible; such a sequence is determined by $\alpha_{0},\dots,\alpha_{m_{<}-1}$, and $\alpha_{k}=(\partial^{k}_{\mathcal{R}(p,q)}f)(0)$. Hence the map $f\mapsto((\partial^{k}_{\mathcal{R}(p,q)}f)(0))_{k<m_{<}}$ is a bijection from the kernel onto $\mathbb{C}^{m_{<}}$, which gives the last assertion. \hfill\(\Box\)

\begin{example}\label{ex:oscillator}
For $\omega\in\mathbb{C}$ with $0<|\omega|<1/r$, the deformed oscillator equation $\partial^{2}_{\mathcal{R}(p,q)}f+\omega^{2}f=0$ has, in $\ER{r}$, the two dimensional space of solutions spanned by the deformed trigonometric functions
\begin{eqnarray}\label{eq:trig}
	\cosR(\omega z)&=&\sum_{n\geq0}\frac{(-1)^{n}\omega^{2n}z^{2n}}{\pf{2n}},\qquad
	\sinR(\omega z)\;=\;\sum_{n\geq0}\frac{(-1)^{n}\omega^{2n+1}z^{2n+1}}{\pf{2n+1}},
\end{eqnarray}
which are $\frac12(\eR(i\omega z)+\eR(-i\omega z))$ and $\frac{1}{2i}(\eR(i\omega z)-\eR(-i\omega z))$; the solution with $f(0)=c_{0}$, $(\partial_{\mathcal{R}(p,q)}f)(0)=c_{1}$ is $c_{0}\cosR(\omega z)+c_{1}\omega^{-1}\sinR(\omega z)$. For $|\omega|>1/r$ the only solution in $\ER{r}$ is $f=0$.
\end{example}

The deformed exponential generates a family of deformed translations. For $t\in\mathbb{C}$ we define the deformed binomial powers
\begin{eqnarray}\label{eq:ward}
	(z\boxplus t)^{n}_{\mathcal{R}(p,q)}&:=&\sum_{k=0}^{n}\bigg[\begin{array}{c} n \\ k\end{array}\bigg]_{\mathcal{R}(p,q)}z^{n-k}\,t^{k},
\end{eqnarray}
which reduce to $(z+t)^{n}$ in the undeformed limit.

\begin{theorem}[Deformed translations]\label{thm:translations}
Assume \Hpos, let $r>0$ and $|t|<r\kappa_{\mathcal{R}(p,q)}$. The series
\begin{eqnarray}
	\tau_{t}&:=&\sum_{k\geq0}\frac{t^{k}}{\pf{k}}\,\partial^{k}_{\mathcal{R}(p,q)}\nn
\end{eqnarray}
converges in operator norm on $\ER{r}$, with $\|\tau_{t}\|\leq\eR(|t|/r)$, and:
\begin{enumerate}
\item[\textnormal{(i)}] $\tau_{t}z^{n}=(z\boxplus t)^{n}_{\mathcal{R}(p,q)}$, $(\tau_{t}f)(0)=f(t)$ and $\tau_{0}=\mathrm{Id}$;
\item[\textnormal{(ii)}] $\tau_{t}\eR(\lambda z)=\eR(\lambda t)\,\eR(\lambda z)$ for $|\lambda|<1/r$, and consequently
\begin{eqnarray}\label{eq:exp-law}
	\eR(\lambda z)\,\eR(\lambda t)&=&\sum_{n\geq0}\frac{\lambda^{n}}{\pf{n}}\,(z\boxplus t)^{n}_{\mathcal{R}(p,q)};
\end{eqnarray}
\item[\textnormal{(iii)}] the operators $\tau_{t}$, $\tau_{s}$ and $\partial_{\mathcal{R}(p,q)}$ commute, and $\partial_{\mathcal{R}(p,q)}\tau_{t}f=\tau_{t}\partial_{\mathcal{R}(p,q)}f$.
\end{enumerate}
\end{theorem}
{\it Proof.}
By Theorem \ref{thm:spectrum}(i), $\|t^{k}\partial^{k}_{\mathcal{R}(p,q)}/\pf{k}\|=|t|^{k}r^{-k}/\pf{k}$, whose sum is $\eR(|t|/r)$, finite since $|t|/r<\kappa_{\mathcal{R}(p,q)}$. (i) $\partial^{k}_{\mathcal{R}(p,q)}z^{n}=\pf{n}z^{n-k}/\pf{n-k}$ for $k\leq n$, which gives \eqref{eq:ward}; at $z=0$ only the term $k=n$ survives, so that $(\tau_{t}f)(0)=\sum_{n}a_{n}t^{n}=f(t)$. (ii) $\partial^{k}_{\mathcal{R}(p,q)}\eR(\lambda z)=\lambda^{k}\eR(\lambda z)$, and \eqref{eq:exp-law} follows by applying (i) termwise to $\eR(\lambda z)=\sum_{n}\lambda^{n}z^{n}/\pf{n}$. (iii) All operators are norm convergent power series in $\partial_{\mathcal{R}(p,q)}$. \hfill\(\Box\)

Formula \eqref{eq:exp-law} is the deformed exponential law: the product of two deformed exponential series is not a deformed exponential series of the sum, but the generating function of the deformed binomial powers \eqref{eq:ward}. The translation $\tau_{t}$ is the deformed Taylor shift: by (i), the Taylor expansion of $\tau_{t}f$ at the origin is the deformed expansion of $f$ around the point $t$.

\section{The deformed Banach algebra and its Gelfand theory}\label{sec:algebra}

\begin{definition}\label{def:convolution}
For $f(z)=\sum_{n}a_{n}z^{n}$ and $g(z)=\sum_{n}b_{n}z^{n}$, the $\mathcal{R}(p,q)-$binomial convolution is
\begin{eqnarray}\label{eq:convolution}
	\big(f\star_{\mathcal{R}(p,q)}g\big)(z)&:=&\sum_{n\geq0}c_{n}z^{n},\qquad
	c_{n}\;:=\;\sum_{k=0}^{n}\bigg[\begin{array}{c} n \\ k\end{array}\bigg]^{-1}_{\mathcal{R}(p,q)}a_{k}\,b_{n-k}.
\end{eqnarray}
\end{definition}

\begin{theorem}\label{thm:banach-algebra}
Assume \Hpos\ and let $r>0$. Then $\big(\ER{r},\star_{\mathcal{R}(p,q)},\nur{r}\big)$ is a commutative unital Banach algebra with unit the constant function $1$,
\begin{eqnarray}\label{eq:submult}
	\nur{r}\big(f\star_{\mathcal{R}(p,q)}g\big)&\leq&\nur{r}(f)\,\nur{r}(g),
\end{eqnarray}
and $\Bor(f\star_{\mathcal{R}(p,q)}g)=\Bor f\cdot\Bor g$, so that $\Bor$ is an isometric algebra isomorphism of $\ER{r}$ onto $\mathcal{A}^{+}(r)$. Moreover the deformed derivative satisfies the twisted Leibniz rule
\begin{eqnarray}\label{eq:twisted-leibniz}
	\partial_{\mathcal{R}(p,q)}\big(f\star_{\mathcal{R}(p,q)}g\big)&=&\big(\partial_{\mathcal{R}(p,q)}f\big)\star_{\mathcal{R}(p,q)}g+f(0)\,\partial_{\mathcal{R}(p,q)}g .
\end{eqnarray}
\end{theorem}
{\it Proof.}
The $n-$th coefficient of $\Bor f\cdot\Bor g$ is $\sum_{k}a_{k}\pf{k}b_{n-k}\pf{n-k}=\pf{n}c_{n}$, which is the $n-$th coefficient of $\Bor(f\star_{\mathcal{R}(p,q)}g)$. Commutativity, associativity, the unit and \eqref{eq:submult} are transported from $\mathcal{A}^{+}(r)$, where $\sum_{n}|\sum_{k}u_{k}v_{n-k}|r^{n}\leq\sum_{k}|u_{k}|r^{k}\sum_{j}|v_{j}|r^{j}$. For \eqref{eq:twisted-leibniz}, by Theorem \ref{thm:borel}(ii) it suffices to check $\mathsf{B}(uv)=(\mathsf{B}u)v+u(0)\mathsf{B}v$ for power series $u,v$, which follows from $u(\lambda)v(\lambda)-u(0)v(0)=(u(\lambda)-u(0))v(\lambda)+u(0)(v(\lambda)-v(0))$. \hfill\(\Box\)

\begin{theorem}[Gelfand theory]\label{thm:gelfand}
Assume \Hpos\ and let $r>0$.
\begin{enumerate}
\item[\textnormal{(i)}] The characters of $\big(\ER{r},\star_{\mathcal{R}(p,q)}\big)$ are the maps $\chi_{\lambda}(f)=\Bor f(\lambda)$, $|\lambda|\leq r$, and $\lambda\mapsto\chi_{\lambda}$ is a homeomorphism of the closed disc $\overline{D_{r}}$ onto the Gelfand spectrum.
\item[\textnormal{(ii)}] The Gelfand transform of $f$ is the restriction of $\Bor f$ to $\overline{D_{r}}$; the $\star_{\mathcal{R}(p,q)}-$spectrum of $f$ is $\Bor f(\overline{D_{r}})$, and its spectral radius is $\varrho_{\star}(f)=\max_{|\lambda|\leq r}|\Bor f(\lambda)|=\lim_{k}\nur{r}(f^{\star k})^{1/k}$.
\item[\textnormal{(iii)}] (Wiener type criterion) $f$ is invertible in $\big(\ER{r},\star_{\mathcal{R}(p,q)}\big)$ if and only if $\Bor f(\lambda)\neq0$ for every $|\lambda|\leq r$.
\item[\textnormal{(iv)}] (functional calculus) if $F$ is holomorphic on a neighbourhood of $\Bor f(\overline{D_{r}})$, there is a unique $F(f)\in\ER{r}$ with $\Bor F(f)=F\circ\Bor f$ on $\overline{D_{r}}$, and $f\mapsto F(f)$ coincides with the holomorphic functional calculus of the Banach algebra.
\end{enumerate}
\end{theorem}
{\it Proof.}
By Theorem \ref{thm:banach-algebra} it suffices to describe $\mathcal{A}^{+}(r)$. A character $\chi$ is determined by $\lambda:=\chi(\lambda\mapsto\lambda)$, since polynomials are dense, and $|\lambda|^{n}=|\chi(\lambda^{n})|\leq\|\lambda^{n}\|_{\mathcal{A}^{+}(r)}=r^{n}$ gives $|\lambda|\leq r$; conversely, evaluation at $|\lambda|\leq r$ is a character, continuous with norm $1$. The map $\lambda\mapsto\chi_{\lambda}$ is a continuous bijection from a compact space onto the Hausdorff Gelfand spectrum, hence a homeomorphism. Assertions (ii), (iii), (iv) are then the Gelfand Beurling spectral radius formula, the Gelfand invertibility criterion and the holomorphic functional calculus of commutative unital Banach algebras \cite{Kaniuth2009}, the Gelfand transform of $u\in\mathcal{A}^{+}(r)$ being $u|_{\overline{D_{r}}}$; for $F(f)$ one uses that $F\circ u$, holomorphic near $\overline{D_{r}}$, belongs to $\mathcal{A}^{+}(r)$. \hfill\(\Box\)

\begin{example}\label{ex:inverse}
For $a\in\mathbb{C}$, $\Bor(1-az)(\lambda)=1-a\pn{1}\lambda$. By Theorem \ref{thm:gelfand}(iii), $1-az$ is $\star_{\mathcal{R}(p,q)}-$invertible in $\ER{r}$ if and only if $|a|\pn{1}r<1$, and its inverse is the function whose Borel transform is $\sum_{n}(a\pn{1})^{n}\lambda^{n}$, namely $\eR(a\pn{1}z)$. The $\star_{\mathcal{R}(p,q)}-$exponential of $f$ is the function $\exp_{\star}(f)$ with $\Bor\exp_{\star}(f)=\exp(\Bor f)$; for instance $\Bor\exp_{\star}(z)(\lambda)=e^{\pn{1}\lambda}$, so that $\exp_{\star}(z)=\sum_{n\geq0}\pn{1}^{n}z^{n}/(n!\,\pf{n})$, an element of $\ER{r}$ for every $r>0$.
\end{example}

\begin{proposition}\label{prop:pointwise}
Assume $\Hphi$ with $\phi_{1}>1$ and let $r>0$. There is no constant $C$ such that $\nur{r}(fg)\leq C\,\nur{r}(f)\,\nur{r}(g)$ for all polynomials $f,g$; in particular the pointwise product is not continuous on $\ER{r}$.
\end{proposition}
{\it Proof.}
For $f=g=z^{k}$, $\nur{r}(z^{2k})/\nur{r}(z^{k})^{2}=\pf{2k}/\pf{k}^{2}$, which is the deformed binomial coefficient $\big[\begin{smallmatrix}2k\\ k\end{smallmatrix}\big]_{\mathcal{R}(p,q)}$. Under $\Hphi$, the gauge constant cancels in this quotient and $\pf{n}=c^{n}\phi_{1}^{n(n-1)/2}[n]_{\fq}!$, so that the quotient equals $\phi_{1}^{k^{2}}\big[\begin{smallmatrix}2k\\ k\end{smallmatrix}\big]_{\fq}\geq\phi_{1}^{k^{2}}\to\infty$. \hfill\(\Box\)

\begin{theorem}[The Fr\'echet algebra]\label{thm:frechet-algebra}
Assume \Hpos\ and let $0<\rho\leq\infty$. Then $\big(\ER{\rho^{-}},\star_{\mathcal{R}(p,q)}\big)$ is a commutative unital Fr\'echet algebra, and $\Bor$ is an isomorphism of topological algebras onto $\mathcal{O}(D_{\rho})$ with the pointwise product. Consequently:
\begin{enumerate}
\item[\textnormal{(i)}] the continuous characters of $\ER{\rho^{-}}$ are the maps $\chi_{\lambda}$, $|\lambda|<\rho$;
\item[\textnormal{(ii)}] $f$ is $\star_{\mathcal{R}(p,q)}-$invertible in $\ER{\rho^{-}}$ if and only if $\Bor f$ has no zero in $D_{\rho}$;
\item[\textnormal{(iii)}] if $\Bor f$ has no zero in $D_{\rho}$, there exists $g\in\ER{\rho^{-}}$ with $f=\exp_{\star}(g)$, where $\Bor\exp_{\star}(g)=\exp(\Bor g)$;
\item[\textnormal{(iv)}] for $f\in\ER{\rho^{-}}$ and $\lambda\in D_{\rho}$ there is $u\in\ER{\rho^{-}}$ with $f-\Bor f(\lambda)=\big(z/\pn{1}-\lambda\big)\star_{\mathcal{R}(p,q)}u$.
\end{enumerate}
\end{theorem}
{\it Proof.}
By Theorem \ref{thm:banach-algebra}, each $\nur{r}$ is submultiplicative for $\star_{\mathcal{R}(p,q)}$, so that $\ER{\rho^{-}}$ is a Fr\'echet algebra, and Theorem \ref{thm:borel}(iii) gives the isomorphism with $\mathcal{O}(D_{\rho})$, the product being transported by Theorem \ref{thm:banach-algebra}. (i) The continuous characters of $\mathcal{O}(D_{\rho})$ are the point evaluations at points of $D_{\rho}$: a character $\chi$ is determined by $\lambda=\chi(\mathrm{id})$, and if $\lambda\notin D_{\rho}$ then $\mathrm{id}-\lambda$ would be invertible. (ii) A function in $\mathcal{O}(D_{\rho})$ is invertible if and only if it has no zero. (iii) $D_{\rho}$ is simply connected, so that a nonvanishing holomorphic function has a holomorphic logarithm. (iv) Since $\Bor(z/\pn{1})(\mu)=\mu$, this is the division of $\Bor f-\Bor f(\lambda)$ by $\mu-\lambda$ in $\mathcal{O}(D_{\rho})$. \hfill\(\Box\)

\begin{corollary}[Deformed valuation]\label{cor:valuation}
Assume \Hpos. The map $\mathrm{v}_{r}(f):=-\log\nur{r}(f)$, $f\in\ER{r}\setminus\{0\}$, satisfies
$\mathrm{v}_{r}(f\star_{\mathcal{R}(p,q)}g)\geq\mathrm{v}_{r}(f)+\mathrm{v}_{r}(g)$ and $\mathrm{v}_{r}(f+g)\geq\min\{\mathrm{v}_{r}(f),\mathrm{v}_{r}(g)\}-\log2$.
\end{corollary}
{\it Proof.}
Take logarithms in \eqref{eq:submult} and in $\nur{r}(f+g)\leq2\max\{\nur{r}(f),\nur{r}(g)\}$. \hfill\(\Box\)

Combined with the Phragm\'en Lindel\"of principles of Section \ref{subsec:PL}, the Gelfand theory gives a principle for the Borel transforms.

\begin{corollary}\label{cor:PL-borel}
Assume \Hpos, let $r>0$, $\theta>0$ with $0<\theta\rho^{*}_{\mathcal{R}(p,q)}\leq\pi$, and let $F:\overline{\Sth}\to\ER{r}$ be continuous and holomorphic on $\Sth$ with $\nur{r}(F(z))\leq A\exp(C|z|^{\sigma})$, $0<\sigma<\omega_{\mathcal{R}(p,q)}(\theta)$. If $\max_{|\lambda|\leq r}|\Bor F(\zeta)(\lambda)|\leq M$ for $\zeta\in\partial\Sth$, then the same bound holds for $\zeta\in\Sth$.
\end{corollary}
{\it Proof.}
By Theorem \ref{thm:gelfand}(ii), $\max_{|\lambda|\leq r}|\Bor f(\lambda)|$ is the spectral radius of $f$ in $\big(\ER{r},\star_{\mathcal{R}(p,q)}\big)$; apply Theorem \ref{thm:PL-spectral} with $S=\Sth$. \hfill\(\Box\)

\section{The deformed Fock space and the coherent states}\label{sec:fock}

The deformed Fischer pairing \eqref{eq:pairing} has a Hilbert space counterpart, in which the creation and annihilation operators of the algebra \eqref{algN1} become mutually adjoint, and the deformed exponential series become reproducing kernels and coherent states. This is the analytic framework of the coherent states constructed in \cite{HounkBuk, Odzijewicz1998}.

\begin{definition}\label{def:fock}
Assume \Hpos. The deformed Fock space is
\begin{eqnarray}\label{eq:fock}
	\mathcal{F}_{\mathcal{R}(p,q)}&:=&\Big\{f=\sum_{n\geq0}a_{n}z^{n}:\ \|f\|^{2}_{\mathcal{F}}:=\sum_{n\geq0}|a_{n}|^{2}\,\pf{n}<\infty\Big\},\qquad
	\langle f,g\rangle_{\mathcal{F}}\;:=\;\sum_{n\geq0}a_{n}\overline{b_{n}}\,\pf{n}.
\end{eqnarray}
\end{definition}

\begin{theorem}\label{thm:fock}
Assume \Hpos. Then:
\begin{enumerate}
\item[\textnormal{(i)}] $\mathcal{F}_{\mathcal{R}(p,q)}$ is a Hilbert space with orthonormal basis $e_{n}=z^{n}/\sqrt{\pf{n}}$, and every $f\in\mathcal{F}_{\mathcal{R}(p,q)}$ is holomorphic on $D_{\sqrt{\kappa_{\mathcal{R}(p,q)}}}$, with $D_{\sqrt{\infty}}=\mathbb{C}$;
\item[\textnormal{(ii)}] for $|z|^{2}<\kappa_{\mathcal{R}(p,q)}$, the point evaluation at $z$ is continuous, $|f(z)|^{2}\leq\|f\|^{2}_{\mathcal{F}}\,\eR(|z|^{2})$, and $\mathcal{F}_{\mathcal{R}(p,q)}$ is a reproducing kernel Hilbert space on $D_{\sqrt{\kappa_{\mathcal{R}(p,q)}}}$ with kernel
\begin{eqnarray}\label{eq:kernel}
	K_{\mathcal{R}(p,q)}(z,w)&=&\sum_{n\geq0}\frac{(z\overline{w})^{n}}{\pf{n}}\;=\;\eR(z\overline{w}),\qquad f(w)\;=\;\langle f,K_{\mathcal{R}(p,q)}(\cdot,w)\rangle_{\mathcal{F}} .
\end{eqnarray}
\end{enumerate}
\end{theorem}
{\it Proof.}
(i) $\mathcal{F}_{\mathcal{R}(p,q)}$ is isometric to $\ell^{2}$ through $f\mapsto(a_{n}\sqrt{\pf{n}})_{n}$. Since $|a_{n}|\leq\|f\|_{\mathcal{F}}\pf{n}^{-1/2}$, the radius of convergence is at least $\liminf_{n}\pf{n}^{1/(2n)}=\sqrt{\kappa_{\mathcal{R}(p,q)}}$. (ii) By the Cauchy Schwarz inequality, $|f(z)|\leq\sum|a_{n}|\pf{n}^{1/2}|z|^{n}\pf{n}^{-1/2}\leq\|f\|_{\mathcal{F}}\big(\sum|z|^{2n}/\pf{n}\big)^{1/2}$, and the last series is $\eR(|z|^{2})$, convergent for $|z|^{2}<\kappa_{\mathcal{R}(p,q)}$ by Corollary \ref{cor:exponential}. For such $w$, $K_{\mathcal{R}(p,q)}(\cdot,w)=\sum_{n}\overline{w}^{n}z^{n}/\pf{n}$ has norm $\eR(|w|^{2})^{1/2}$, and $\langle f,K_{\mathcal{R}(p,q)}(\cdot,w)\rangle_{\mathcal{F}}=\sum_{n}a_{n}w^{n}\pf{n}/\pf{n}=f(w)$. \hfill\(\Box\)

\begin{theorem}[Fock representation]\label{thm:fock-representation}
Assume \Hpos. Let $A=\partial_{\mathcal{R}(p,q)}$ and $A^{\dagger}=$ multiplication by $z$ on their maximal domains
\begin{eqnarray}
	D(A)&=&\Big\{f\in\mathcal{F}_{\mathcal{R}(p,q)}:\sum_{n\geq1}|a_{n}|^{2}\pf{n}\pn{n}<\infty\Big\},\nn\\
	D(A^{\dagger})&=&\Big\{f\in\mathcal{F}_{\mathcal{R}(p,q)}:\sum_{n\geq0}|a_{n}|^{2}\pf{n}\pn{n+1}<\infty\Big\},\nn
\end{eqnarray}
and let $N$ be the self adjoint operator $Ne_{n}=ne_{n}$. Then:
\begin{enumerate}
\item[\textnormal{(i)}] $Ae_{n}=\sqrt{\pn{n}}\,e_{n-1}$ for $n\geq1$, $Ae_{0}=0$, and $A^{\dagger}e_{n}=\sqrt{\pn{n+1}}\,e_{n+1}$;
\item[\textnormal{(ii)}] $A$ and $A^{\dagger}$ are closed, densely defined, and $A^{*}=A^{\dagger}$, $(A^{\dagger})^{*}=A$; in particular $\langle\partial_{\mathcal{R}(p,q)}f,g\rangle_{\mathcal{F}}=\langle f,zg\rangle_{\mathcal{F}}$ for $f\in D(A)$, $g\in D(A^{\dagger})$;
\item[\textnormal{(iii)}] on the polynomials, $AA^{\dagger}=[N+1]_{\mathcal{R}(p,q)}$, $A^{\dagger}A=[N]_{\mathcal{R}(p,q)}$, $[N,A]=-A$ and $[N,A^{\dagger}]=A^{\dagger}$, where $[N+k]_{\mathcal{R}(p,q)}e_{n}=\pn{n+k}e_{n}$; this is the Fock representation of the algebra \eqref{algN1};
\item[\textnormal{(iv)}] $A$ is bounded if and only if $\sup_{n}\pn{n}<\infty$, in which case $\|A\|=\sup_{n}\sqrt{\pn{n}}$; under $\Hphi$ this happens if and only if $\phi_{1}\leq1$.
\end{enumerate}
\end{theorem}
{\it Proof.}
(i) $Az^{n}=\pn{n}z^{n-1}$ gives $Ae_{n}=\pn{n}\sqrt{\pf{n-1}/\pf{n}}\,e_{n-1}=\sqrt{\pn{n}}\,e_{n-1}$, and $A^{\dagger}e_{n}=\sqrt{\pf{n+1}/\pf{n}}\,e_{n+1}$. (ii) In the coordinates $f=\sum c_{n}e_{n}$, $A$ is the weighted backward shift $c\mapsto(w_{n+1}c_{n+1})_{n}$ with $w_{n}=\sqrt{\pn{n}}$ on its maximal domain $\{\sum|w_{n}c_{n}|^{2}<\infty\}$, and $A^{\dagger}$ is the weighted forward shift $c\mapsto(0,w_{1}c_{0},w_{2}c_{1},\dots)$ on $\{\sum|w_{n+1}c_{n}|^{2}<\infty\}$. Both domains contain the finitely supported sequences, hence are dense, and both operators are closed, being maximal diagonal type operators. If $g\in D(A^{*})$ and $A^{*}g=h$, testing with $f=e_{n}$ gives $\overline{h_{n}}=w_{n}\overline{g_{n-1}}$, so that $h=A^{\dagger}g$ and $D(A^{*})\subset D(A^{\dagger})$; conversely, for $f\in D(A)$ and $g\in D(A^{\dagger})$, $\sum_{n}w_{n+1}c_{n+1}\overline{g_{n}}=\sum_{n}c_{n+1}\overline{w_{n+1}g_{n}}$, the series converging absolutely by the Cauchy Schwarz inequality, which is $\langle Af,g\rangle=\langle f,A^{\dagger}g\rangle$. The equality $(A^{\dagger})^{*}=A$ follows in the same way. (iii) $AA^{\dagger}e_{n}=w_{n+1}^{2}e_{n}=\pn{n+1}e_{n}$ and $A^{\dagger}Ae_{n}=w_{n}^{2}e_{n}$, while $NAe_{n}=(n-1)Ae_{n}$ and $NA^{\dagger}e_{n}=(n+1)A^{\dagger}e_{n}$. (iv) A weighted shift is bounded if and only if its weights are bounded, with norm their supremum. Under $\Hphi$, $\pn{n}\to0$ if $\phi_{1}<1$, $\pn{n}\uparrow c/(1-\phi_{2})$ if $\phi_{1}=1$, and $\pn{n}\to\infty$ if $\phi_{1}>1$. \hfill\(\Box\)

\begin{theorem}[Deformed coherent states]\label{thm:coherent}
Assume \Hpos. For $|w|^{2}<\kappa_{\mathcal{R}(p,q)}$, let $k_{w}:=K_{\mathcal{R}(p,q)}(\cdot,w)=\eR(\overline{w}\,\cdot)$ and $|w\rangle:=\eR(|w|^{2})^{-1/2}k_{w}$. Then $|w\rangle$ is a unit vector of $D(A)$, $A|w\rangle=\overline{w}\,|w\rangle$, and
\begin{eqnarray}\label{eq:coherent-overlap}
	\langle k_{w},k_{v}\rangle_{\mathcal{F}}\;=\;\eR(v\overline{w}),\qquad
	\big|\langle w|v\rangle\big|^{2}\;=\;\frac{|\eR(v\overline{w})|^{2}}{\eR(|w|^{2})\,\eR(|v|^{2})}\;\leq\;1 .
\end{eqnarray}
Moreover the map $w\mapsto|w\rangle$ is continuous, and its image is total in $\mathcal{F}_{\mathcal{R}(p,q)}$. If $\lim_{n}\pn{n}=L\in(0,\infty)$, the spectrum of the bounded operator $A$ is the closed disc of radius $\sqrt{L}$, and every point of the open disc is an eigenvalue.
\end{theorem}
{\it Proof.}
By Corollary \ref{cor:exponential}, $\partial_{\mathcal{R}(p,q)}\eR(\overline{w}z)=\overline{w}\eR(\overline{w}z)$, and $k_{w}\in D(A)$ since $\sum_{n}|w|^{2n}\pn{n}/\pf{n}=|w|^{2}\eR(|w|^{2})<\infty$. The overlap formula is \eqref{eq:kernel} with $f=k_{w}$, and the inequality is the Cauchy Schwarz inequality. Continuity follows from the continuity of the kernel, and totality from the reproducing property: if $f\perp k_{w}$ for all $w$ then $f\equiv0$ on $D_{\sqrt{\kappa_{\mathcal{R}(p,q)}}}$. If $\pn{n}\to L$, then $\kappa_{\mathcal{R}(p,q)}=L$ by Proposition \ref{prop:indices}(iii), so that every $\overline{w}$ with $|w|<\sqrt{L}$ is an eigenvalue. The spectral radius is $\lim_{k}\|A^{k}\|^{1/k}$ with $\|A^{k}\|=\sup_{n}\big(\pn{n}\pn{n-1}\cdots\pn{n-k+1}\big)^{1/2}$; since $\pn{n}\to L$, for every $\varepsilon>0$ this supremum lies between $(L-\varepsilon)^{k/2}$ and $C_{\varepsilon}(L+\varepsilon)^{k/2}$, so that the spectral radius is $\sqrt{L}$, and the spectrum, being closed, is the closed disc. \hfill\(\Box\)

In the Jackson regime, the coherent states are those of the $q-$oscillator, with normalisation $\eR(|w|^{2})=1/((1-\phi_{2})|w|^{2}/c;\phi_{2})_{\infty}$ for $|w|^{2}<c/(1-\phi_{2})$. In the Gaussian regime they exist for every $w\in\mathbb{C}$, $A$ is unbounded and $\eR(|w|^{2})$ is entire of order zero. In the subexponential factorial regime, $\kappa_{\mathcal{R}(p,q)}=0$, the coherent states do not exist in $\mathcal{F}_{\mathcal{R}(p,q)}$, and $A$ is a compact weighted shift with spectrum $\{0\}$. The resolution of the identity requires a moment problem attached to $\pf{n}$, treated in \cite{HounkBuk}.

\begin{remark}\label{rem:hamiltonian}
The deformed oscillator Hamiltonian $H=A^{\dagger}A=[N]_{\mathcal{R}(p,q)}$ is diagonal in the basis $(e_{n})$, with spectrum the closure of $\{\pn{n}:n\geq0\}$. Under $\Hphi$ the energy levels accumulate at $0$ in the subexponential factorial regime, increase to $c/(1-\phi_{2})$ in the Jackson regime, and grow geometrically like $c\,\phi_{1}^{n}/(\phi_{1}-\phi_{2})$ in the Gaussian regime. The commutator $[A,A^{\dagger}]=[N+1]_{\mathcal{R}(p,q)}-[N]_{\mathcal{R}(p,q)}$ is a nonnegative operator if and only if $n\mapsto\pn{n}$ is nondecreasing, which by Proposition \ref{prop:log-convex-factorial} happens under $\Hphi$ exactly when $\phi_{1}\geq1$. In this case, for the self adjoint quadratures $X=(A+A^{\dagger})/\sqrt2$ and $Y=(A-A^{\dagger})/(i\sqrt2)$, defined on the polynomials, one has $[X,Y]=i\,[A,A^{\dagger}]$ and the Robertson inequality $\Delta X\,\Delta Y\geq\frac12\langle[A,A^{\dagger}]\rangle$ holds in every normalized state of the polynomial domain; in the Jackson regime the right hand side tends to $0$ on high excitations, while in the Gaussian regime it grows.
\end{remark}

\section{Convexity and subharmonicity in the deformed setting}\label{sec:convexity}

The growth principles of Sections \ref{subsec:PL} and \ref{sec:PL} rest, as in the classical theory \cite{Hadamard1896, PhragmenLindelof1908, Titchmarsh1939, Ransford1995}, on two mechanisms: the convexity of a modulus function with respect to the logarithm of the radius, and the subharmonicity of the logarithm of the modulus of a holomorphic function. We isolate these mechanisms inside the deformation.

\subsection{Convexity of the deformed seminorms}\label{subsec:log-convex}

\begin{theorem}\label{thm:log-convex-nu}
Assume \Hpos, let $f$ be a power series, $0<r_{1}<r_{2}$ and $\theta\in[0,1]$, and put $r_{\theta}:=r_{1}^{1-\theta}r_{2}^{\theta}$. Then
\begin{eqnarray}\label{eq:log-convex-nu}
	\nur{r_{\theta}}(f)&\leq&\nur{r_{1}}(f)^{1-\theta}\,\nur{r_{2}}(f)^{\theta}.
\end{eqnarray}
Equivalently, $t\mapsto\log\nur{e^{t}}(f)$ is convex and nondecreasing on the interval where it is finite. In particular, for $r\in[r_{1},r_{2}]$, the deformed three circles inequality
\begin{eqnarray}\label{eq:three-circles}
	\nur{r}(f)^{\log(r_{2}/r_{1})}&\leq&\nur{r_{1}}(f)^{\log(r_{2}/r)}\;\nur{r_{2}}(f)^{\log(r/r_{1})}
\end{eqnarray}
holds.
\end{theorem}
{\it Proof.}
With $c_{n}:=|a_{n}|\pf{n}\geq0$, $\nur{r_{\theta}}(f)=\sum_{n}(c_{n}r_{1}^{n})^{1-\theta}(c_{n}r_{2}^{n})^{\theta}$, and the H\"older inequality with exponents $1/(1-\theta)$ and $1/\theta$ gives \eqref{eq:log-convex-nu}. Writing $r=e^{t}$, \eqref{eq:log-convex-nu} is the convexity inequality, and monotonicity follows from Proposition \ref{prop:nu-elementary}(i). For \eqref{eq:three-circles}, choose $\theta=\log(r/r_{1})/\log(r_{2}/r_{1})$ and raise \eqref{eq:log-convex-nu} to the power $\log(r_{2}/r_{1})$. \hfill\(\Box\)

\begin{proposition}\label{prop:log-convex-factorial}
Assume \Hpos\ and that $n\mapsto\pn{n}$ is nondecreasing on $\mathbb{N}$. Then $\pf{n}^{2}\leq\pf{n-1}\,\pf{n+1}$ for $n\geq1$, and the deformed binomial coefficients satisfy $\big[\begin{smallmatrix} n \\ k\end{smallmatrix}\big]_{\mathcal{R}(p,q)}\geq1$ for $0\leq k\leq n$. Under $\Hphi$, the sequence $n\mapsto\pn{n}$ is nondecreasing if and only if $\phi_{1}\geq1$.
\end{proposition}
{\it Proof.}
$\pf{n+1}/\pf{n}=\pn{n+1}\geq\pn{n}=\pf{n}/\pf{n-1}$. For the binomial coefficients, $\big[\begin{smallmatrix} n \\ k\end{smallmatrix}\big]_{\mathcal{R}(p,q)}=\prod_{j=1}^{k}\pn{n-k+j}/\pn{j}$ and each factor is at least $1$. Under $\Hphi$,
\begin{eqnarray}
	\frac{\phi_{1}-\phi_{2}}{c}\,\big(\pn{n+1}-\pn{n}\big)&=&\phi_{1}^{n}(\phi_{1}-1)-\phi_{2}^{n}(\phi_{2}-1).\nn
\end{eqnarray}
If $\phi_{1}<1$, then $\pn{n}\to0$ and the sequence cannot be nondecreasing. If $\phi_{1}\geq1\geq\phi_{2}$, both terms are nonnegative. If $\phi_{1}>\phi_{2}>1$, the right hand side is positive because $x\mapsto x^{n}(x-1)$ is increasing on $(1,\infty)$. \hfill\(\Box\)

\begin{remark}\label{rem:threshold}
The convexity of Theorem \ref{thm:log-convex-nu} concerns the factorial weighted seminorms and holds for all radii. For the maximum of $|f|$ over deformed balls, $M_{\mathcal{R}(p,q)}(f,r):=\max_{\|z\|_{\mathcal{R}(p,q)}\leq r}|f(z)|$, Theorem \ref{thm:gauge-Hphi} gives, under $\Hphi$ with $c=1$, $M_{\mathcal{R}(p,q)}(f,r)=|f(0)|$ for $r<1/\phi_{1}$ and $M_{\mathcal{R}(p,q)}(f,r)=\max_{|z|\leq r}|f(z)|$ for $r\geq1/\phi_{1}$. The function $t\mapsto\log M_{\mathcal{R}(p,q)}(f,e^{t})$ is therefore convex on $[-\log\phi_{1},\infty)$ by the classical Hadamard theorem \cite{Hadamard1896}, but it has a jump at the threshold $t=-\log\phi_{1}$ whenever $\max_{|z|\leq1/\phi_{1}}|f|>|f(0)|$; the three circles inequality for deformed balls holds only above the minimal resolution radius.
\end{remark}

\subsection{Subharmonicity of the deformed weight function and of the dilation orbit maximum}\label{subsec:subharmonic-weight}

\begin{theorem}\label{thm:subharm-weight}
Assume \Hpos, let $0<\rho\leq\infty$ and $f\in\ER{\rho^{-}}$, and define the deformed weight function
\begin{eqnarray}\label{eq:Phi-def}
	\Phi_{f}(z)&:=&\log\Big(\sum_{n\geq0}|a_{n}|\,\pf{n}\,|z|^{n}\Big)\;=\;\log\nur{|z|}(f),\qquad |z|<\rho .
\end{eqnarray}
Then $\Phi_{f}$ is subharmonic on $D_{\rho}$, radial, and $t\mapsto\Phi_{f}(e^{t})$ is convex and nondecreasing. If moreover $\pf{n}\geq1$ for all $n$, then $\log|f(z)|\leq\Phi_{f}(z)$ on $D_{\rho}$. Finally $\Phi_{f}(z)=\log\|\Bor f\|_{\mathcal{A}^{+}(|z|)}$.
\end{theorem}
{\it Proof.}
For $N\in\mathbb{N}$, put $g_{n}(z)=|a_{n}|\pf{n}z^{n}$. Then $\sum_{n\leq N}|g_{n}(z)|=\sup\{|\sum_{n\leq N}\lambda_{n}g_{n}(z)|:|\lambda_{n}|=1\}$, the supremum being attained, and each $\log|\sum_{n\leq N}\lambda_{n}g_{n}|$ is subharmonic; the supremum is continuous, hence $\Phi^{(N)}_{f}=\log\sum_{n\leq N}|g_{n}|$ is subharmonic \cite{Ransford1995, HaymanKennedy1976}. The sequence $\Phi^{(N)}_{f}$ increases to $\Phi_{f}$, which is finite and continuous on $D_{\rho}$ because $f\in\ER{\rho^{-}}$, and an increasing limit of subharmonic functions which is locally bounded above is subharmonic. Radiality is clear, convexity in $\log|z|$ is Theorem \ref{thm:log-convex-nu}, and the comparison with $\log|f|$ follows from $|f(z)|\leq\sum|a_{n}||z|^{n}\leq\sum|a_{n}|\pf{n}|z|^{n}$. The last identity is Theorem \ref{thm:borel}(i). \hfill\(\Box\)

\label{subsec:orbit}

\begin{definition}\label{def:orbite}
For $z\in\mathbb{C}$ the deformed dilation orbit of $z$ is $\Lambda(z):=\{p^{m}q^{n}z:m,n\in\mathbb{N}\cup\{0\}\}$, the orbit of $z$ under the semigroup generated by $P$ and $Q$. A domain $\Omega\subset\mathbb{C}$ is deformation stable if $p\,\Omega\subset\Omega$ and $q\,\Omega\subset\Omega$.
\end{definition}

\begin{theorem}\label{thm:subharm-orbit}
Let $\Omega$ be a deformation stable domain and let $f$ be holomorphic on $\Omega$. Put
\begin{eqnarray}\label{eq:orbite-max}
	u_{f}(z)&:=&\sup_{m,n\geq0}\log\big|f(p^{m}q^{n}z)\big|,\qquad z\in\Omega .
\end{eqnarray}
If $u_{f}$ is locally bounded above on $\Omega$, its upper semicontinuous regularization $u_{f}^{*}$ is subharmonic on $\Omega$, $\log|f|\leq u_{f}\leq u^{*}_{f}$, and $u_{f}(pz)\leq u_{f}(z)$, $u_{f}(qz)\leq u_{f}(z)$ for $z\in\Omega$.
\end{theorem}
{\it Proof.}
Each $z\mapsto p^{m}q^{n}z$ is a holomorphic self map of $\Omega$, so that $\log|f(p^{m}q^{n}z)|$ is subharmonic, and the upper semicontinuous regularization of the supremum of a locally bounded above family of subharmonic functions is subharmonic \cite{Ransford1995, HaymanKennedy1976}. The inequalities follow from $(m,n)=(0,0)$ and from $\Lambda(pz)\cup\Lambda(qz)\subset\Lambda(z)$. \hfill\(\Box\)

\begin{corollary}[Maximum principle on the orbits]\label{cor:max-orbite}
Let $0<q<p\leq1$, let $\Omega$ be a bounded deformation stable domain and let $f$ be holomorphic on $\Omega$ and continuous on $\overline{\Omega}$. Then
\begin{eqnarray}
	\sup_{z\in\Omega}\sup_{m,n\geq0}\big|f(p^{m}q^{n}z)\big|&=&\sup_{z\in\partial\Omega}\sup_{m,n\geq0}\big|f(p^{m}q^{n}z)\big| .\nn
\end{eqnarray}
\end{corollary}
{\it Proof.}
Since $\overline{\Omega}$ is compact and stable, $u_{f}$ is bounded above, and $u_{f}^{*}$ is subharmonic on $\Omega$ and upper semicontinuous up to the boundary, where it coincides with the boundary supremum of the orbit; the maximum principle for subharmonic functions \cite{Ransford1995} applies. \hfill\(\Box\)

\section{Gaussian growth: zeros, logarithmic type and the deformed indicator}\label{sec:PL}

In this section we draw the consequences of the $\mathcal{R}(p,q)-$Phragm\'en Lindel\"of theorem \ref{thm:rigidity} for the zeros and the growth of the functions of finite deformed exponential type, and we prove the trigonometric convexity of the growth indicator in the anisotropic sectors.

\begin{corollary}\label{cor:exp-gaussian}
Assume \Hpos\ and \HG. For every $\lambda\neq0$, the deformed exponential $\eR(\lambda z)$ is entire of order zero and unbounded on every ray; in particular $\eR(-x)$ does not tend to $0$ as $x\to+\infty$.
\end{corollary}
{\it Proof.}
Under \HG, $\frac{1}{n}\log\pf{n}\geq\gamma'n$ for $n$ large and some $\gamma'>0$, so that $\kappa_{\mathcal{R}(p,q)}=\infty$; moreover $|\lambda|^{n}\pf{n}^{-1}\pf{n}r^{n}$ is bounded for $r=1/|\lambda|$. Apply Theorem \ref{thm:rigidity}, $\eR(\lambda z)$ being nonconstant. \hfill\(\Box\)

\begin{theorem}[Zeros under Gaussian growth]\label{thm:zeros}
Assume \Hpos\ and \HG. Let $f=\sum_{n}a_{n}z^{n}$ be nonconstant, with $\sup_{n}|a_{n}|\pf{n}r^{n}<\infty$ for some $r>0$, and let $n_{f}(\varrho)$ denote the number of zeros of $f$ in the closed disc $\overline{D_{\varrho}}$, counted with multiplicity. Then, for every $0<\gamma'<\gamma_{\mathcal{R}(p,q)}$,
\begin{eqnarray}\label{eq:zeros}
	n_{f}(\varrho)&\leq&\frac{\log\varrho}{\gamma'}+O(1),\qquad\varrho\to\infty,
\end{eqnarray}
the series $\sum_{k}|z_{k}|^{-\varepsilon}$ over the nonzero zeros converges for every $\varepsilon>0$, and $f(z)=C z^{m}\prod_{k}(1-z/z_{k})$, the product converging absolutely and uniformly on compact sets.
\end{theorem}
{\it Proof.}
Let $m$ be the order of the zero of $f$ at the origin and $g(z)=f(z)/z^{m}$, which is entire with $g(0)\neq0$, $n_{g}=n_{f}-m$ and $\max_{|z|=\varrho}|g|=\varrho^{-m}\max_{|z|=\varrho}|f|\leq\max_{|z|=\varrho}|f|$ for $\varrho\geq1$. Jensen's formula gives $\int_{0}^{\varrho'}n_{g}(t)t^{-1}dt\leq\log\max_{|z|=\varrho'}|g|-\log|g(0)|$, and since $n_{g}$ is nondecreasing, $n_{g}(\varrho)\log(\varrho'/\varrho)\leq\int_{\varrho}^{\varrho'}n_{g}(t)t^{-1}dt$. With $\varrho'=\varrho^{2}$ and \eqref{eq:order-zero},
\begin{eqnarray}
	n_{g}(\varrho)\,\log\varrho&\leq&\frac{\big(2\log\varrho-\log r\big)^{2}}{4\gamma'}+O(1),\nn
\end{eqnarray}
which gives \eqref{eq:zeros}, since $n_{f}=n_{g}+m$. Then $\sum_{k}|z_{k}|^{-\varepsilon}=\varepsilon\int_{0}^{\infty}n_{f}(t)t^{-\varepsilon-1}dt$ is finite, and the Hadamard factorization theorem for functions of order zero \cite{Boas1954, Levin1980} gives the product representation with genus zero. \hfill\(\Box\)

\begin{corollary}[Value distribution]\label{cor:picard}
Assume \Hpos\ and \HG, and let $f$ be of finite deformed exponential type. Then either $f$ is a polynomial, or $f$ takes every complex value $a$ infinitely often, and the number of solutions of $f(z)=a$ in $\overline{D_{\varrho}}$ is at most $\log\varrho/\gamma'+O(1)$ for every $0<\gamma'<\gamma_{\mathcal{R}(p,q)}$.
\end{corollary}
{\it Proof.}
The function $f-a$ has finite deformed exponential type, since only its constant coefficient changes. If $f-a$ has finitely many zeros, the Hadamard factorization of Theorem \ref{thm:zeros} shows that $f-a$ is a polynomial. The counting estimate is \eqref{eq:zeros} applied to $f-a$. \hfill\(\Box\)

Under $\Hphi$ with $\phi_{1}>1$, \eqref{eq:zeros} reads $n_{f}(\varrho)\leq(2+o(1))\log\varrho/\log\phi_{1}$; Section \ref{sec:real-zeros} shows that the deformed exponential attains this bound up to the factor $2$.

\begin{theorem}[Logarithmic type characterization]\label{thm:log-type}
Assume $\Hphi$ with $\phi_{1}>1$, and for an entire function $f$ put $\tau(f):=\limsup_{\varrho\to\infty}\log\max_{|z|=\varrho}|f(z)|/(\log\varrho)^{2}$. Then:
\begin{enumerate}
\item[\textnormal{(i)}] if $f\in\ER{r}$ for some $r>0$, then $\tau(f)\leq1/(2\log\phi_{1})$;
\item[\textnormal{(ii)}] if $\tau(f)<1/(2\log\phi_{1})$, then $f\in\ER{r}$ for every $r>0$, that is $f\in\ER{\infty^{-}}$;
\item[\textnormal{(iii)}] for $\lambda\neq0$, $\tau(\eR(\lambda\,\cdot))=1/(2\log\phi_{1})$, while $\eR(\lambda\,\cdot)\in\ER{r}$ if and only if $|\lambda|r<1$; hence the critical value $1/(2\log\phi_{1})$ in \textnormal{(i)} and \textnormal{(ii)} is sharp.
\end{enumerate}
\end{theorem}
{\it Proof.}
(i) is \eqref{eq:order-zero} with $\gamma'\uparrow\gamma_{\mathcal{R}(p,q)}=\frac12\log\phi_{1}$. (ii) Let $\tau(f)<\tau'<1/(2\log\phi_{1})$; then $\log\max_{|z|=\varrho}|f|\leq\tau'(\log\varrho)^{2}$ for $\varrho\geq\varrho_{0}$. The Cauchy inequalities with $\log\varrho=n/(2\tau')$, admissible for $n$ large, give $\log|a_{n}|\leq\tau'(\log\varrho)^{2}-n\log\varrho=-n^{2}/(4\tau')$. Since $\pf{n}=c^{n}\phi_{1}^{n(n-1)/2}[n]_{\fq}!\leq(c/(1-\fq))^{n}\phi_{1}^{n(n-1)/2}$, we get $|a_{n}|\pf{n}r^{n}\leq C^{n}\exp\big(-n^{2}(\tfrac{1}{4\tau'}-\tfrac12\log\phi_{1})\big)$, which is summable for every $r$. (iii) The membership statement is Theorem \ref{thm:spectrum}(iii). For the type, since the coefficients are positive for $\lambda>0$, $\max_{|z|=\varrho}|\eR(\lambda z)|=\eR(|\lambda|\varrho)\geq(|\lambda|\varrho)^{n}/\pf{n}$ for every $n$; choosing $n$ with $n\approx\log(|\lambda|\varrho)/\log\phi_{1}$ and using $\log\pf{n}=\frac{n^{2}}{2}\log\phi_{1}+O(n)$ gives $\log\eR(|\lambda|\varrho)\geq\frac{(\log\varrho)^{2}}{2\log\phi_{1}}+O(\log\varrho)$, and (i) gives the reverse inequality. \hfill\(\Box\)

Theorem \ref{thm:log-type} identifies, in the Gaussian regime, the Fr\'echet space $\ER{\infty^{-}}$ with the entire functions of logarithmic type smaller than or equal to $1/(2\log\phi_{1})$, up to the boundary case: the factorial weights translate the Gaussian growth of the deformed factorials into the logarithmic growth $\exp\big((\log\varrho)^{2}/(2\log\phi_{1})\big)$ of the maximum modulus, the critical functions being the deformed exponentials.

\begin{remark}\label{rem:counterexample}
Hypothesis \HG\ cannot be removed from Theorem \ref{thm:rigidity}. In the Jackson regime, $\pf{n}=c^{n}[n]_{\phi_{2}}!\leq(c/(1-\phi_{2}))^{n}$, so that $f(z)=e^{-z}$ satisfies $\sup_{n}\pf{n}r^{n}/n!<\infty$ for every $r>0$, while $e^{-x}\to0$ on the positive axis and $f\not\equiv0$. In the same regime, the deformed exponential series $\eR(z)=1/((1-\phi_{2})z/c;\phi_{2})_{\infty}$ is meromorphic and $\eR(-x)\to0$ as $x\to+\infty$, in contrast with Corollary \ref{cor:exp-gaussian}. Under $\Hphi$, \HG\ holds exactly in the Gaussian regime $\phi_{1}>1$, with $\gamma_{\mathcal{R}(p,q)}=\frac12\log\phi_{1}$.
\end{remark}

\medskip\noindent\textbf{Trigonometric convexity of the deformed indicator.}\label{subsec:indicator}

\begin{proposition}\label{prop:indicateur}
Assume \Hpos, let $\theta>0$ with $0<\theta\rho^{*}_{\mathcal{R}(p,q)}\leq\pi$, and let $f$ be holomorphic on $\Sth$, bounded on $\Sth\cap\{|z|<1\}$, of order $\sigma>0$ and finite type, in the sense that $\log|f(z)|\leq C|z|^{\sigma}$ for $z\in\Sth$, $|z|\geq1$. Define
\begin{eqnarray}\label{eq:indicator}
	h_{f}(\varphi)&:=&\limsup_{\varrho\to\infty}\frac{\log|f(\varrho e^{i\varphi})|}{\varrho^{\sigma}},\qquad|\varphi|<\theta_{\mathcal{R}(p,q)} .
\end{eqnarray}
Then, for $-\theta_{\mathcal{R}(p,q)}<\varphi_{1}<\varphi<\varphi_{2}<\theta_{\mathcal{R}(p,q)}$ with $\varphi_{2}-\varphi_{1}<\pi/\sigma$ and $h_{f}(\varphi_{1}),h_{f}(\varphi_{2})>-\infty$,
\begin{eqnarray}\label{eq:trig-convex}
	h_{f}(\varphi)\sin\big(\sigma(\varphi_{2}-\varphi_{1})\big)&\leq&h_{f}(\varphi_{1})\sin\big(\sigma(\varphi_{2}-\varphi)\big)+h_{f}(\varphi_{2})\sin\big(\sigma(\varphi-\varphi_{1})\big).
\end{eqnarray}
\end{proposition}
{\it Proof.}
Let $H(\psi)=\alpha\cos(\sigma\psi)+\beta\sin(\sigma\psi)$ be the unique such function with $H(\varphi_{j})=h_{f}(\varphi_{j})$, $j=1,2$; it exists since $\sin(\sigma(\varphi_{2}-\varphi_{1}))\neq0$, and $H(\arg z)|z|^{\sigma}=\Re((\alpha-i\beta)z^{\sigma})$ is harmonic. Let $\psi_{0}=(\varphi_{1}+\varphi_{2})/2$, $\delta=(\varphi_{2}-\varphi_{1})/2$ and $w(z)=\Re\big((ze^{-i\psi_{0}})^{\sigma}\big)/\cos(\sigma\delta)$, which is harmonic on $S'=\{\varphi_{1}<\arg z<\varphi_{2}\}$, satisfies $w(z)\geq|z|^{\sigma}$ there because $\sigma\delta<\pi/2$, and equals $|z|^{\sigma}$ on the two boundary rays. For $\varepsilon>0$ the function
$u_{\varepsilon}(z)=\log|f(z)|-H(\arg z)|z|^{\sigma}-\varepsilon w(z)$ is subharmonic on $S'$ and bounded above near the vertex. On the ray $\arg z=\varphi_{j}$, which lies inside $\Sth$, $u_{\varepsilon}(z)=\log|f(z)|-(h_{f}(\varphi_{j})+\varepsilon)|z|^{\sigma}$, which tends to $-\infty$ by the definition of $h_{f}$ and is continuous; hence $u_{\varepsilon}\leq m_{\varepsilon}<\infty$ on $\partial S'\setminus\{0\}$. Since $u_{\varepsilon}(z)\leq C'|z|^{\sigma}$ for $|z|\geq1$ and $\sigma<\pi/(2\delta)$, Theorem \ref{thm:PL-subharm}, applied after a rotation to $u_{\varepsilon}-m_{\varepsilon}$ on $S'$, gives $u_{\varepsilon}\leq m_{\varepsilon}$ on $S'$. Along the ray $\arg z=\varphi$ this reads $\log|f(\varrho e^{i\varphi})|\leq(H(\varphi)+\varepsilon k_{\varphi})\varrho^{\sigma}+m_{\varepsilon}$ with $k_{\varphi}=\cos(\sigma(\varphi-\psi_{0}))/\cos(\sigma\delta)$, whence $h_{f}(\varphi)\leq H(\varphi)+\varepsilon k_{\varphi}$. Letting $\varepsilon\to0$ and expanding $H(\varphi)$ by means of its values at $\varphi_{1}$ and $\varphi_{2}$ gives \eqref{eq:trig-convex} \cite{Levin1980, Boas1954}. \hfill\(\Box\)

\section{Real zeros of the deformed exponential and trigonometric functions}\label{sec:real-zeros}

Theorem \ref{thm:zeros} bounds the number of zeros of a function of finite deformed exponential type, but not their position. For the deformed exponential series $\eR$ and its trigonometric functions \eqref{eq:trig}, and, through the rescaling of Lemma \ref{lem:rescaling}, for the deformed exponentials and trigonometric functions \eqref{eq:def-exp}, the zeros can be located in terms of the deformed numbers by a dominant term principle in the spirit of Hutchinson \cite{Hutchinson1923} and of the problem of Hardy, Petrovitch and Hutchinson \cite{KostovShapiro2013}.

For a power series $f(z)=\sum_{n\geq0}a_{n}z^{n}$ with positive coefficients we write
\begin{eqnarray}\label{eq:quotients}
	\rho_{n}(f)&:=&\frac{a_{n-1}}{a_{n}},\qquad \theta_{n}(f)\;:=\;\frac{a_{n}^{2}}{a_{n-1}a_{n+1}}\;=\;\frac{\rho_{n+1}(f)}{\rho_{n}(f)},\qquad n\geq1,
\end{eqnarray}
and we use the theta value
\begin{eqnarray}\label{eq:theta-value}
	\vartheta(Q)&:=&\sum_{j\in\mathbb{Z}}Q^{-j^{2}/2}\;=\;\vartheta_{3}\big(0,Q^{-1/2}\big),\qquad Q>1,
\end{eqnarray}
where $\vartheta_{3}(0,x)=1+2\sum_{j\geq1}x^{j^{2}}$ is the Jacobi theta function \cite{GasperRahman}. The function $\vartheta$ decreases strictly from $+\infty$ to $1$ on $(1,\infty)$, and we denote by $Q_{*}$ the unique solution of $\vartheta(Q_{*})=2$; numerically $Q_{*}=4.8105828$ and $Q_{*}^{1/4}=1.4809808$.

\subsection{A dominant term localization principle}\label{subsec:dominant}

\begin{theorem}[Dominant term localization]\label{thm:localization}
Let $f(z)=\sum_{n\geq0}a_{n}z^{n}$ with $a_{n}>0$ for all $n\geq0$, and assume that there is $Q>1$ with $\vartheta(Q)\leq2$ such that $\theta_{n}(f)\geq Q$ for every $n\geq1$. Put $\rho_{n}=\rho_{n}(f)$ for $n\geq1$, $\rho_{0}:=\rho_{1}/Q$ and
\begin{eqnarray}\label{eq:circles}
	r_{k}&:=&\big(\rho_{k}\rho_{k+1}\big)^{1/2},\qquad k\geq0 .
\end{eqnarray}
Then $f$ is entire of order zero, $r_{k+1}\geq Q\,r_{k}$ for every $k\geq0$, and:
\begin{enumerate}
\item[\textnormal{(i)}] for every $k\geq0$ and every $z$ with $|z|=r_{k}$,
\begin{eqnarray}\label{eq:dominant}
	\big|f(z)-a_{k}z^{k}\big|&<&\big(\vartheta(Q)-1\big)\,a_{k}r_{k}^{k},\qquad
	\big(2-\vartheta(Q)\big)\,a_{k}r_{k}^{k}\;<\;|f(z)|\;<\;\vartheta(Q)\,a_{k}r_{k}^{k};
\end{eqnarray}
in particular $f$ has no zero on the circle $|z|=r_{k}$ and exactly $k$ zeros in $D_{r_{k}}$, counted with multiplicity;
\item[\textnormal{(ii)}] the zeros of $f$ are real, negative and simple, and they form a sequence $-\xi_{1}>-\xi_{2}>\cdots$ with
\begin{eqnarray}\label{eq:localization}
	r_{k-1}&<&\xi_{k}\;<\;r_{k},\qquad k\geq1;
\end{eqnarray}
\item[\textnormal{(iii)}] $f(z)=a_{0}\prod_{k\geq1}(1+z/\xi_{k})$, the product converging absolutely and uniformly on compact subsets of $\mathbb{C}$, and
\begin{eqnarray}\label{eq:rayleigh}
	\sum_{k\geq1}\frac{1}{\xi_{k}}&=&\frac{a_{1}}{a_{0}},\qquad \sum_{k\geq1}\frac{1}{\xi_{k}^{2}}\;=\;\frac{a_{1}^{2}-2a_{0}a_{2}}{a_{0}^{2}} .
\end{eqnarray}
\end{enumerate}
\end{theorem}
{\it Proof.}
Since $\theta_{m}=\rho_{m+1}/\rho_{m}\geq Q$ for $m\geq1$, one has $\rho_{m+i}\geq Q^{i}\rho_{m}$ for $m\geq1$ and $i\geq0$, and, by the choice of $\rho_{0}$, also for $m=0$. Hence $a_{n}=a_{0}/(\rho_{1}\cdots\rho_{n})\leq a_{0}\rho_{1}^{-n}Q^{-n(n-1)/2}$, so that $f$ is entire, and the estimate of the proof of Theorem \ref{thm:rigidity}(i), applied with $\gamma'=\frac12\log Q$, shows that $f$ has order zero. Moreover $r_{k+1}/r_{k}=(\rho_{k+2}/\rho_{k})^{1/2}\geq Q$.

(i) Fix $k\geq0$ and $|z|=r:=r_{k}$. By definition of $r_{k}$,
\begin{eqnarray}
	\frac{r}{\rho_{k+1}}\;=\;\frac{\rho_{k}}{r}\;=\;\Big(\frac{\rho_{k}}{\rho_{k+1}}\Big)^{1/2}\;\leq\;Q^{-1/2},\nn
\end{eqnarray}
because $\rho_{k+1}/\rho_{k}=\theta_{k}\geq Q$ for $k\geq1$ and $\rho_{1}/\rho_{0}=Q$. For $j\geq1$, using $\rho_{k+i}\geq Q^{i-1}\rho_{k+1}$,
\begin{eqnarray}
	\frac{a_{k+j}r^{k+j}}{a_{k}r^{k}}&=&\prod_{i=1}^{j}\frac{r}{\rho_{k+i}}\;\leq\;\prod_{i=1}^{j}Q^{-1/2-(i-1)}\;=\;Q^{-j^{2}/2},\nn
\end{eqnarray}
and for $1\leq j\leq k$, using $\rho_{k-i}\leq Q^{-i}\rho_{k}$ for $0\leq i\leq j-1$,
\begin{eqnarray}
	\frac{a_{k-j}r^{k-j}}{a_{k}r^{k}}&=&\prod_{i=0}^{j-1}\frac{\rho_{k-i}}{r}\;\leq\;\prod_{i=0}^{j-1}Q^{-1/2-i}\;=\;Q^{-j^{2}/2}.\nn
\end{eqnarray}
Therefore
\begin{eqnarray}
	\big|f(z)-a_{k}z^{k}\big|&\leq&a_{k}r^{k}\Big(\sum_{j\geq1}Q^{-j^{2}/2}+\sum_{j=1}^{k}Q^{-j^{2}/2}\Big)\;<\;a_{k}r^{k}\big(\vartheta(Q)-1\big),\nn
\end{eqnarray}
since the first sum equals $(\vartheta(Q)-1)/2$ and the second one, which is finite, is strictly smaller than $(\vartheta(Q)-1)/2$. The two sided bound follows from the triangle inequality. As $\vartheta(Q)-1\leq1$, one has $|f(z)-a_{k}z^{k}|<|a_{k}z^{k}|$ on the circle, and Rouch\'e's theorem shows that $f$ and $a_{k}z^{k}$ have the same number of zeros, namely $k$, in $D_{r_{k}}$.

(ii) By (i), $f$ has no zero in $\overline{D_{r_{0}}}$, no zero on the circles $|z|=r_{k}$, and exactly one zero, counted with multiplicity, in each annulus $r_{k-1}<|z|<r_{k}$, $k\geq1$; since $r_{k}\to\infty$, every zero lies in one of these annuli. The coefficients being real, the zero set is invariant under complex conjugation, and a nonreal zero $z_{0}$ in the annulus would produce a second zero $\overline{z_{0}}\neq z_{0}$ in the same annulus. Hence the zero of the $k-$th annulus is real and simple, and it is negative because $f(x)\geq a_{0}>0$ for $x\geq0$.

(iii) Since $f$ has order zero and $f(0)=a_{0}\neq0$, the Hadamard factorization theorem \cite{Boas1954, Levin1980} gives $f(z)=a_{0}\prod_{k}(1+z/\xi_{k})$ with $\sum_{k}\xi_{k}^{-1}<\infty$. For $|z|<\xi_{1}$, $\log(f(z)/a_{0})=\sum_{k}\log(1+z/\xi_{k})=s_{1}z-\frac12s_{2}z^{2}+O(z^{3})$ with $s_{m}=\sum_{k}\xi_{k}^{-m}$, while $\log(1+\frac{a_{1}}{a_{0}}z+\frac{a_{2}}{a_{0}}z^{2}+\cdots)=\frac{a_{1}}{a_{0}}z+\big(\frac{a_{2}}{a_{0}}-\frac{a_{1}^{2}}{2a_{0}^{2}}\big)z^{2}+O(z^{3})$; comparing the coefficients of $z$ and $z^{2}$ gives \eqref{eq:rayleigh}. \hfill\(\Box\)

\begin{remark}\label{rem:hutchinson}
Hutchinson's condition $\theta_{n}(f)\geq4$ \cite{Hutchinson1923} implies the reality of the zeros, and $4$ cannot be lowered, as the quadratic polynomials show \cite{KostovShapiro2013}. The stronger condition $Q\geq Q_{*}$ gives in addition the localization \eqref{eq:localization}, the simplicity of the zeros and, for $Q>Q_{*}$, the estimate $\log\min_{|z|=r_{k}}|f|=\log\max_{|z|=r_{k}}|f|+O(1)$, uniformly in $k$.
\end{remark}

Conversely, real zeros force inequalities on the coefficients.

\begin{proposition}[Newton type necessary conditions]\label{prop:newton}
Let $f(z)=\sum_{n\geq0}a_{n}z^{n}$ be a transcendental entire function of order less than $1$ with $a_{n}>0$ for all $n\geq0$. If all the zeros of $f$ are real, then $f$ has infinitely many zeros, all of them negative, and
\begin{eqnarray}\label{eq:newton}
	\theta_{n}(f)&\geq&1+\frac{1}{n},\qquad n\geq1,\qquad\mbox{and}\qquad a_{1}^{2}\;>\;2a_{0}a_{2} .
\end{eqnarray}
\end{proposition}
{\it Proof.}
Real zeros are negative since $f(x)\geq a_{0}>0$ for $x\geq0$. Since the order is less than $1$, the Hadamard factorization reads $f(z)=a_{0}\prod_{k}(1+z/\xi_{k})$ with $\xi_{k}>0$ and $\sum_{k}\xi_{k}^{-1}<\infty$ \cite{Boas1954, Levin1980}; the product is infinite because $f$ is not a polynomial. Comparing coefficients as in the proof of Theorem \ref{thm:localization}(iii) gives $\sum_{k}\xi_{k}^{-2}=(a_{1}^{2}-2a_{0}a_{2})/a_{0}^{2}$, which is positive. Let $P_{N}(z)=\prod_{k\leq N}(1+z/\xi_{k})=\sum_{n=0}^{N}e^{(N)}_{n}z^{n}$, where $e^{(N)}_{n}$ is the $n-$th elementary symmetric function of $\xi_{1}^{-1},\dots,\xi_{N}^{-1}$. Newton's inequalities \cite{HardyLittlewoodPolya1952} state that $e^{(N)}_{n}/\binom{N}{n}$ is a logarithmically concave sequence of $n$, that is
\begin{eqnarray}
	\big(e^{(N)}_{n}\big)^{2}&\geq&\frac{(n+1)(N-n+1)}{n(N-n)}\,e^{(N)}_{n-1}e^{(N)}_{n+1}\;\geq\;\Big(1+\frac1n\Big)e^{(N)}_{n-1}e^{(N)}_{n+1},\qquad 1\leq n\leq N-1 .\nn
\end{eqnarray}
Since $P_{N}\to f/a_{0}$ uniformly on compact sets, $e^{(N)}_{n}\to a_{n}/a_{0}$ for every $n$ by the Cauchy formula, and letting $N\to\infty$ gives the first inequality in \eqref{eq:newton}. \hfill\(\Box\)

\subsection{The deformed exponential series and its trigonometric functions}\label{subsec:exp-zeros}

\begin{lemma}\label{lem:theta-phi}
Assume $\Hphi$. Then, for $n\geq1$,
\begin{eqnarray}\label{eq:theta-phi}
	\frac{\pn{n+1}}{\pn{n}}&=&\phi_{1}+\frac{\phi_{1}-\phi_{2}}{\fq^{-n}-1},
\end{eqnarray}
so that $\pn{n+1}/\pn{n}$ decreases strictly from $\pn{2}/\pn{1}=\phi_{1}+\phi_{2}$ to its limit $\phi_{1}$. Consequently, for every $m\geq1$, $\pn{n+m}/\pn{n}$ decreases strictly in $n$ to the limit $\phi_{1}^{m}$.
\end{lemma}
{\it Proof.}
The identity $\phi_{1}^{n+1}-\phi_{2}^{n+1}=\phi_{1}(\phi_{1}^{n}-\phi_{2}^{n})+\phi_{2}^{n}(\phi_{1}-\phi_{2})$ gives \eqref{eq:theta-phi}, and $\fq^{-n}$ increases strictly to $\infty$. The last assertion follows from $\pn{n+m}/\pn{n}=\prod_{i=0}^{m-1}\pn{n+i+1}/\pn{n+i}$. \hfill\(\Box\)

\begin{theorem}[Zeros of the deformed exponential series]\label{thm:exp-zeros}
Assume $\Hphi$ with $\phi_{1}>1$. Then the deformed exponential series $\eR(z)=\sum_{n\geq0}z^{n}/\pf{n}$ is entire of order zero and has infinitely many zeros $z_{k}$, repeated according to multiplicity, which satisfy
\begin{eqnarray}\label{eq:exp-rayleigh}
	\sum_{k}\frac{1}{z_{k}}&=&-\frac{1}{c},\qquad \sum_{k}\frac{1}{z_{k}^{2}}\;=\;\frac{\phi_{1}+\phi_{2}-2}{c^{2}(\phi_{1}+\phi_{2})},
\end{eqnarray}
both series being absolutely convergent. Moreover:
\begin{enumerate}
\item[\textnormal{(i)}] if $\phi_{1}+\phi_{2}\leq2$, then $\eR$ has nonreal zeros;
\item[\textnormal{(ii)}] if $\phi_{1}\geq4$, then all the zeros of $\eR$ are real and negative;
\item[\textnormal{(iii)}] if $\phi_{2}=1$, then
\begin{eqnarray}\label{eq:euler-case}
	\eR(z)&=&\prod_{k\geq1}\Big(1+\frac{(\phi_{1}-1)\,z}{c\,\phi_{1}^{k}}\Big),
\end{eqnarray}
so that the zeros are the simple negative numbers $-\xi_{k}$ with $\xi_{k}=c\,\phi_{1}^{k}/(\phi_{1}-1)=\pn{k}+c/(\phi_{1}-1)$, $k\geq1$;
\item[\textnormal{(iv)}] if $\phi_{2}=\phi_{1}^{1/2}$, then, with $\fq=\phi_{1}^{-1/2}$,
\begin{eqnarray}\label{eq:ramanujan-case}
	\eR(z)&=&A_{\fq}\Big(-\frac{(1-\fq)\,z}{c\,\fq}\Big),\qquad A_{q}(x)\;:=\;\sum_{n\geq0}\frac{q^{n^{2}}(-x)^{n}}{(q;q)_{n}},
\end{eqnarray}
where $A_{q}$ is Ramanujan's entire function, and all the zeros of $\eR$ are real and negative;
\item[\textnormal{(v)}] if $\phi_{1}\geq Q_{*}$, then the zeros of $\eR$ are simple and negative, $-\xi_{1}>-\xi_{2}>\cdots$, with
\begin{eqnarray}\label{eq:exp-localization}
	c\,\phi_{1}^{-1/2}&<&\xi_{1}\;<\;c\,(\phi_{1}+\phi_{2})^{1/2},\nn\\
	\big(\pn{k-1}\pn{k}\big)^{1/2}&<&\xi_{k}\;<\;\big(\pn{k}\pn{k+1}\big)^{1/2},\qquad k\geq2,
\end{eqnarray}
$\eR(z)=\prod_{k\geq1}(1+z/\xi_{k})$, and the number $n(\varrho)$ of zeros in $\overline{D_{\varrho}}$ satisfies $n(\varrho)=\log\varrho/\log\phi_{1}+O(1)$ as $\varrho\to\infty$.
\end{enumerate}
\end{theorem}
{\it Proof.}
Under $\Hphi$ with $\phi_{1}>1$, \HG\ holds with $\gamma_{\mathcal{R}(p,q)}=\frac12\log\phi_{1}$ (Proposition \ref{prop:indices}(iv)), so that $\eR$ is entire of order zero by Corollary \ref{cor:exp-gaussian}; it is not a polynomial, hence it has infinitely many zeros by Corollary \ref{cor:picard} applied with $a=0$, and $\sum_{k}|z_{k}|^{-\varepsilon}<\infty$ for every $\varepsilon>0$ by Theorem \ref{thm:zeros}. Writing $\eR(z)=\prod_{k}(1-z/z_{k})$ and comparing the coefficients of $z$ and $z^{2}$ in $\log\eR(z)$ as in the proof of Theorem \ref{thm:localization}(iii), with $a_{0}=1$, $a_{1}=1/\pn{1}=1/c$ and $a_{2}=1/(\pn{1}\pn{2})=1/(c^{2}(\phi_{1}+\phi_{2}))$, gives \eqref{eq:exp-rayleigh}.

(i) If all the zeros were real, Proposition \ref{prop:newton} would give $a_{1}^{2}>2a_{0}a_{2}$, that is $\phi_{1}+\phi_{2}>2$.

(ii) Here $\theta_{n}(\eR)=\pn{n+1}/\pn{n}>\phi_{1}\geq4$ by Lemma \ref{lem:theta-phi}, and Hutchinson's theorem \cite{Hutchinson1923} applies.

(iii) and (iv). By the proof of Proposition \ref{prop:indices}(iv), $\pf{n}=c^{n}\phi_{1}^{n(n-1)/2}(\fq;\fq)_{n}(1-\fq)^{-n}$, so that, with $w=(1-\fq)z/c$,
\begin{eqnarray}\label{eq:exp-qseries}
	\eR(z)&=&\sum_{n\geq0}\frac{\phi_{1}^{-n(n-1)/2}\,w^{n}}{(\fq;\fq)_{n}} .
\end{eqnarray}
If $\phi_{2}=1$, then $\fq=\phi_{1}^{-1}$ and Euler's identity $\sum_{n}q^{n(n-1)/2}w^{n}/(q;q)_{n}=(-w;q)_{\infty}$ \cite{GasperRahman} gives $\eR(z)=\prod_{j\geq0}(1+\fq^{j}w)$, which is \eqref{eq:euler-case} since $(1-\fq)\fq^{k-1}=(\phi_{1}-1)\phi_{1}^{-k}$; moreover $c\,\phi_{1}^{k}/(\phi_{1}-1)-\pn{k}=c/(\phi_{1}-1)$. If $\phi_{2}=\phi_{1}^{1/2}$, then $\phi_{1}^{-n(n-1)/2}=\fq^{n(n-1)}$ and \eqref{eq:exp-qseries} equals $\sum_{n}\fq^{n^{2}}(w/\fq)^{n}/(\fq;\fq)_{n}=A_{\fq}(-w/\fq)$. It remains to recall why the zeros of $A_{q}$, $0<q<1$, are positive \cite{Ismail2005, Zhang2017}. The polynomials $S_{n}(x)=\sum_{k=0}^{n}\big[\begin{smallmatrix} n\\ k\end{smallmatrix}\big]_{q}q^{k^{2}}(-x)^{k}$, with the Gaussian binomial coefficients $\big[\begin{smallmatrix} n\\ k\end{smallmatrix}\big]_{q}=(q;q)_{n}/((q;q)_{k}(q;q)_{n-k})$, are, up to the factor $(q;q)_{n}^{-1}$, the Stieltjes Wigert polynomials, which are orthogonal with respect to a positive measure on $(0,\infty)$ \cite{Ismail2005}; their zeros are therefore positive. Since $0<\big[\begin{smallmatrix} n\\ k\end{smallmatrix}\big]_{q}\leq(q;q)_{\infty}^{-2}$ and $\big[\begin{smallmatrix} n\\ k\end{smallmatrix}\big]_{q}\to(q;q)_{k}^{-1}$ as $n\to\infty$, dominated convergence gives $S_{n}\to A_{q}$ uniformly on compact sets, and Hurwitz's theorem shows that every zero of $A_{q}$ is a limit of zeros of $S_{n}$, hence real and nonnegative; finally $A_{q}(0)=1$.

(v) By Lemma \ref{lem:theta-phi}, $\theta_{n}(\eR)>\phi_{1}\geq Q_{*}$ for $n\geq1$, and Theorem \ref{thm:localization} applies with $Q=\phi_{1}$, $\rho_{n}=\pn{n}$, $\rho_{0}=c/\phi_{1}$, $r_{0}=c\,\phi_{1}^{-1/2}$ and $r_{k}=(\pn{k}\pn{k+1})^{1/2}$; this gives \eqref{eq:exp-localization} and the product. Since $\pn{k}=c\,\phi_{1}^{k}(1-\fq^{k})/(\phi_{1}-\phi_{2})$, one has $\log r_{k}=(k+\frac12)\log\phi_{1}+O(1)$, and $n(\varrho)\in\{k,k+1\}$ for $r_{k}\leq\varrho<r_{k+1}$, which gives the asymptotics of $n(\varrho)$. \hfill\(\Box\)

\begin{remark}\label{rem:exp-zeros}
The condition $\phi_{1}+\phi_{2}>2$ is necessary for the reality of the zeros of $\eR$, $\phi_{1}\geq4$ is sufficient, and the exactly solvable cases $\phi_{2}=1$ and $\phi_{2}=\phi_{1}^{1/2}$ show that the zeros may be real for every $\phi_{1}>1$. In cases (iii) and (v), $\log\xi_{k}=k\log\phi_{1}+O(1)$ and $n(\varrho)=\log\varrho/\log\phi_{1}+O(1)$, so that the bound $(2+o(1))\log\varrho/\log\phi_{1}$ of Theorem \ref{thm:zeros} is attained up to the factor $2$.
\end{remark}

The functions \eqref{eq:trig} with $\omega=1$ are written $\cosR(z)=C_{\mathcal{R}(p,q)}(-z^{2})$ and $\sinR(z)=z\,S_{\mathcal{R}(p,q)}(-z^{2})$, where
\begin{eqnarray}\label{eq:CS}
	C_{\mathcal{R}(p,q)}(w)&=&\sum_{k\geq0}\frac{w^{k}}{\pf{2k}},\qquad S_{\mathcal{R}(p,q)}(w)\;=\;\sum_{k\geq0}\frac{w^{k}}{\pf{2k+1}} .
\end{eqnarray}
By \eqref{eq:diag}, $\partial_{\mathcal{R}(p,q)}\cosR=-\sinR$ and $\partial_{\mathcal{R}(p,q)}\sinR=\cosR$. All the zeros of $\cosR$ are real if and only if all the zeros of $C_{\mathcal{R}(p,q)}$ are real, since the zeros of $C_{\mathcal{R}(p,q)}$, if real, are negative; the same holds for $\sinR$ and $S_{\mathcal{R}(p,q)}$.

\begin{theorem}[Zeros of the trigonometric series]\label{thm:trig-zeros}
Assume $\Hphi$ with $\phi_{1}>1$. Then $\cosR$ and $\sinR$ are entire functions of order zero, $C_{\mathcal{R}(p,q)}$ and $S_{\mathcal{R}(p,q)}$ have infinitely many zeros, and the zeros $w_{k}$ of $C_{\mathcal{R}(p,q)}$ and $w'_{k}$ of $S_{\mathcal{R}(p,q)}$ satisfy
\begin{eqnarray}\label{eq:trig-rayleigh}
	\sum_{k}\frac{1}{w_{k}}&=&-\frac{1}{\pn{1}\pn{2}},\qquad \sum_{k}\frac{1}{w'_{k}}\;=\;-\frac{1}{\pn{2}\pn{3}} .
\end{eqnarray}
Moreover:
\begin{enumerate}
\item[\textnormal{(i)}] if $\pn{3}\pn{4}\leq2\pn{1}\pn{2}$, that is $(\phi_{1}^{2}+\phi_{1}\phi_{2}+\phi_{2}^{2})(\phi_{1}^{2}+\phi_{2}^{2})\leq2$, then $\cosR$ has nonreal zeros, and if $\pn{4}\pn{5}\leq2\pn{2}\pn{3}$, then $\sinR$ has nonreal zeros;
\item[\textnormal{(ii)}] if $\phi_{1}\geq\sqrt2$, then all the zeros of $\cosR$ and of $\sinR$ are real;
\item[\textnormal{(iii)}] if $\phi_{1}\geq Q_{*}^{1/4}$, then the zeros of $\cosR$ are simple and real, $\pm x_{k}$ with $0<x_{1}<x_{2}<\cdots$, and those of $\sinR$ are simple and real, $0$ and $\pm y_{k}$ with $0<y_{1}<y_{2}<\cdots$; with $\sigma_{k}=\pn{2k-1}\pn{2k}$, $\sigma'_{k}=\pn{2k}\pn{2k+1}$ for $k\geq1$ and $\sigma_{0}=\sigma_{1}/\phi_{1}^{4}$, $\sigma'_{0}=\sigma'_{1}/\phi_{1}^{4}$,
\begin{eqnarray}\label{eq:trig-localization}
	\big(\sigma_{k-1}\sigma_{k}\big)^{1/4}&<&x_{k}\;<\;\big(\sigma_{k}\sigma_{k+1}\big)^{1/4},\qquad
	\big(\sigma'_{k-1}\sigma'_{k}\big)^{1/4}\;<\;y_{k}\;<\;\big(\sigma'_{k}\sigma'_{k+1}\big)^{1/4},\qquad k\geq1;
\end{eqnarray}
moreover
\begin{eqnarray}\label{eq:trig-products}
	\cosR(z)&=&\prod_{k\geq1}\Big(1-\frac{z^{2}}{x_{k}^{2}}\Big),\qquad
	\sinR(z)\;=\;\frac{z}{c}\prod_{k\geq1}\Big(1-\frac{z^{2}}{y_{k}^{2}}\Big),\nn\\
	\sum_{k\geq1}\frac{1}{x_{k}^{2}}&=&\frac{1}{\pn{1}\pn{2}},\qquad \sum_{k\geq1}\frac{1}{y_{k}^{2}}\;=\;\frac{1}{\pn{2}\pn{3}},
\end{eqnarray}
and the numbers of zeros of $\cosR$ and of $\sinR$ in $\overline{D_{\varrho}}$ are both equal to $\log\varrho/\log\phi_{1}+O(1)$.
\end{enumerate}
\end{theorem}
{\it Proof.}
The coefficients of $C_{\mathcal{R}(p,q)}$ are $b_{k}=1/\pf{2k}$, with $\rho_{k}(C_{\mathcal{R}(p,q)})=\pn{2k-1}\pn{2k}$ and
\begin{eqnarray}
	\theta_{k}(C_{\mathcal{R}(p,q)})&=&\frac{\pn{2k+1}\pn{2k+2}}{\pn{2k-1}\pn{2k}}\;=\;\frac{\pn{2k+1}}{\pn{2k-1}}\,\frac{\pn{2k+2}}{\pn{2k}}\;>\;\phi_{1}^{4},\qquad k\geq1,\nn
\end{eqnarray}
by Lemma \ref{lem:theta-phi}; similarly the coefficients of $S_{\mathcal{R}(p,q)}$ are $b'_{k}=1/\pf{2k+1}$, with $\rho_{k}(S_{\mathcal{R}(p,q)})=\pn{2k}\pn{2k+1}$ and $\theta_{k}(S_{\mathcal{R}(p,q)})=\pn{2k+2}\pn{2k+3}/(\pn{2k}\pn{2k+1})>\phi_{1}^{4}$. The functions $\cosR$ and $\sinR$ are linear combinations of $\eR(iz)$ and $\eR(-iz)$, hence entire of order zero by Theorem \ref{thm:exp-zeros}; since $\max_{|w|=\varrho}|C_{\mathcal{R}(p,q)}(w)|=\max_{|z|=\varrho^{1/2}}|\cosR(z)|$ and similarly for $S_{\mathcal{R}(p,q)}$, the functions $C_{\mathcal{R}(p,q)}$ and $S_{\mathcal{R}(p,q)}$ also have order zero, and, being transcendental, they have infinitely many zeros by the Hadamard factorization. The identities \eqref{eq:trig-rayleigh} follow from the comparison of the coefficients of $w$ in the logarithms, since $b_{1}/b_{0}=1/(\pn{1}\pn{2})$ and $b'_{1}/b'_{0}=\pn{1}/\pf{3}=1/(\pn{2}\pn{3})$.

(i) By Proposition \ref{prop:newton}, if all the zeros of $C_{\mathcal{R}(p,q)}$ are real then $b_{1}^{2}>2b_{0}b_{2}$, that is $\pf{4}>2\pf{2}^{2}$, or $\pn{3}\pn{4}>2\pn{1}\pn{2}$; under $\Hphi$, $\pn{3}\pn{4}/(\pn{1}\pn{2})=(\phi_{1}^{2}+\phi_{1}\phi_{2}+\phi_{2}^{2})(\phi_{1}^{2}+\phi_{2}^{2})$. Similarly $b_{1}'^{2}>2b'_{0}b'_{2}$ reads $\pn{1}\pf{5}>2\pf{3}^{2}$, that is $\pn{4}\pn{5}>2\pn{2}\pn{3}$.

(ii) If $\phi_{1}\geq\sqrt2$, then $\theta_{k}(C_{\mathcal{R}(p,q)})$ and $\theta_{k}(S_{\mathcal{R}(p,q)})$ exceed $\phi_{1}^{4}\geq4$, and Hutchinson's theorem \cite{Hutchinson1923} shows that all the zeros of $C_{\mathcal{R}(p,q)}$ and $S_{\mathcal{R}(p,q)}$ are real, hence negative.

(iii) If $\phi_{1}^{4}\geq Q_{*}$, Theorem \ref{thm:localization} applies to $C_{\mathcal{R}(p,q)}$ and to $S_{\mathcal{R}(p,q)}$ with $Q=\phi_{1}^{4}$. Since $\rho_{k}(C_{\mathcal{R}(p,q)})=\sigma_{k}$, the zeros of $C_{\mathcal{R}(p,q)}$ are $-\eta_{k}$ with $(\sigma_{k-1}\sigma_{k})^{1/2}<\eta_{k}<(\sigma_{k}\sigma_{k+1})^{1/2}$, and the zeros of $\cosR(z)=C_{\mathcal{R}(p,q)}(-z^{2})$ are $\pm x_{k}$, $x_{k}=\eta_{k}^{1/2}$, which gives the first part of \eqref{eq:trig-localization}. They are simple because $\frac{d}{dz}C_{\mathcal{R}(p,q)}(-z^{2})=-2zC'_{\mathcal{R}(p,q)}(-z^{2})$ does not vanish at $\pm x_{k}$. The same argument applied to $S_{\mathcal{R}(p,q)}$, with $\rho_{k}(S_{\mathcal{R}(p,q)})=\sigma'_{k}$ and $S_{\mathcal{R}(p,q)}(0)=1/c$, gives the second part, the zero of $\sinR$ at the origin being simple. The products and the sums \eqref{eq:trig-products} are Theorem \ref{thm:localization}(iii) for $C_{\mathcal{R}(p,q)}$ and $S_{\mathcal{R}(p,q)}$, written in the variable $z$. Finally $x_{k}$ and $y_{k}$ lie between consecutive terms of sequences whose logarithms are $2k\log\phi_{1}+O(1)$, and each $x_{k}$ or $y_{k}$ accounts for two zeros, which gives the counting functions. \hfill\(\Box\)

\subsection{The deformed exponentials and trigonometric functions of first and second kind}\label{subsec:trig-zeros}

The weights $\phi_{1}^{n(n-1)/2}$ and $\phi_{2}^{n(n-1)/2}$ of \eqref{eq:def-exp} are absorbed by a rescaling of the deformed numbers, which reduces the functions of first and second kind to deformed exponential series of two other two root sequences.

\begin{lemma}[Rescaling]\label{lem:rescaling}
Assume $\Hphi$ and put $\pn{n}^{(j)}:=\phi_{j}^{1-n}\pn{n}$, $j=1,2$, $n\geq0$. Then $\pf{n}^{(j)}:=\prod_{k=1}^{n}\pn{k}^{(j)}=\phi_{j}^{-n(n-1)/2}\pf{n}$, the sequence $(\pn{n}^{(1)})$ satisfies $\Hphi$ with $(\phi_{1},\phi_{2},c)$ replaced by $(1,\fq,c)$, and $(\pn{n}^{(2)})$ satisfies $\Hphi$ with $(\phi_{1},\phi_{2},c)$ replaced by $(\phi_{1}/\phi_{2},1,c)$. Consequently $\eRw$, $\cosRw$ and $\sinRw$ are the series $\eR$, $\cosR$ and $\sinR$ built on $(\pn{n}^{(1)})$, and $\ERw$, $\CosRw$ and $\SinRw$ are those built on $(\pn{n}^{(2)})$.
\end{lemma}
{\it Proof.}
$\sum_{k=1}^{n}(k-1)=n(n-1)/2$ gives the factorials. Moreover $\phi_{1}^{1-n}\pn{n}=c(1-\fq^{n})/(1-\fq)$ and $\phi_{2}^{1-n}\pn{n}=c((\phi_{1}/\phi_{2})^{n}-1)/((\phi_{1}/\phi_{2})-1)$, and the coefficients of \eqref{eq:def-exp} are $1/\pf{n}^{(j)}$, respectively $(-1)^{n}/\pf{2n}^{(j)}$ and $(-1)^{n}/\pf{2n+1}^{(j)}$. \hfill\(\Box\)

Since the results of the previous subsection use the structure function only through the sequence $(\pn{n})$, they apply to the rescaled sequences. The sequence $(\pn{n}^{(1)})$ is in the Jackson regime and the sequence $(\pn{n}^{(2)})$ is in the Gaussian regime with second root equal to $1$, which is the Euler case of Theorem \ref{thm:exp-zeros}(iii).

\begin{corollary}[Zeros and poles of the deformed exponentials]\label{cor:E-zeros}
Assume $\Hphi$. Then
\begin{eqnarray}\label{eq:E-product}
	\ERw(z)&=&\prod_{k\geq0}\Big(1+\frac{(1-\fq)\,\fq^{k}}{c}\,z\Big),\qquad
	\eRw(z)\;=\;\prod_{k\geq0}\Big(1-\frac{(1-\fq)\,\fq^{k}}{c}\,z\Big)^{-1},
\end{eqnarray}
so that $\eRw(z)\ERw(-z)=1$; $\ERw$ is entire of order zero with the simple negative zeros $-c\,\fq^{-k}/(1-\fq)$, $k\geq0$, and $\eRw$ is holomorphic in $|z|<c/(1-\fq)$, without zeros, and extends to a meromorphic function on $\mathbb{C}$ with the simple poles $c\,\fq^{-k}/(1-\fq)$, $k\geq0$.
\end{corollary}
{\it Proof.}
By Lemma \ref{lem:rescaling}, $\ERw$ is the series $\eR$ of $(\pn{n}^{(2)})$, whose first root $\phi_{1}/\phi_{2}=\fq^{-1}$ exceeds $1$ and whose second root equals $1$; formula \eqref{eq:euler-case} with $\phi_{1}$ replaced by $\fq^{-1}$ gives the first product, since $(\fq^{-1}-1)\fq^{k}=(1-\fq)\fq^{k-1}$. For $\eRw$, $\pf{n}^{(1)}=c^{n}(\fq;\fq)_{n}(1-\fq)^{-n}$, and Euler's identity $\sum_{n}w^{n}/(\fq;\fq)_{n}=1/(w;\fq)_{\infty}$, $|w|<1$ \cite{GasperRahman}, with $w=(1-\fq)z/c$ gives the second product. \hfill\(\Box\)

\begin{corollary}[Zeros of the deformed trigonometric functions]\label{cor:Cos-zeros}
Assume $\Hphi$.
\begin{enumerate}
\item[\textnormal{(i)}] The functions $\CosRw$ and $\SinRw$ of second kind are entire of order zero. If $\phi_{1}\geq\sqrt{2}\,\phi_{2}$, all their zeros are real, and their positive zeros $x_{k}$ and $y_{k}$ satisfy $\sum_{k}x_{k}^{-2}=\phi_{2}/(\pn{1}\pn{2})$ and $\sum_{k}y_{k}^{-2}=\phi_{2}^{3}/(\pn{2}\pn{3})$; if $\phi_{1}\geq Q_{*}^{1/4}\phi_{2}$, the zeros are simple and real and satisfy \eqref{eq:trig-localization}, with $\sigma_{k}=\pn{2k-1}^{(2)}\pn{2k}^{(2)}=\phi_{2}^{3-4k}\pn{2k-1}\pn{2k}$, $\sigma'_{k}=\pn{2k}^{(2)}\pn{2k+1}^{(2)}=\phi_{2}^{1-4k}\pn{2k}\pn{2k+1}$ and $\phi_{1}$ replaced by $\phi_{1}/\phi_{2}$.
\item[\textnormal{(ii)}] The functions $\cosRw$ and $\sinRw$ of first kind are holomorphic in $|z|<c/(1-\fq)$ and extend to meromorphic functions on $\mathbb{C}$, with poles on the imaginary axis, at $\pm i\,c\,\fq^{-k}/(1-\fq)$, $k\geq0$.
\end{enumerate}
\end{corollary}
{\it Proof.}
(i) By Lemma \ref{lem:rescaling}, Theorem \ref{thm:trig-zeros} applies to the sequence $(\pn{n}^{(2)})$, whose first root is $\phi_{1}/\phi_{2}>1$; the sums are \eqref{eq:trig-products} with $\pn{n}^{(2)}=\phi_{2}^{1-n}\pn{n}$. (ii) $\cosRw(z)$ and $\sinRw(z)$ are linear combinations of $\eRw(iz)$ and $\eRw(-iz)$, and Corollary \ref{cor:E-zeros} gives the poles of these two functions; none of them is cancelled, since at a pole of $\eRw(iz)$ the function $\eRw(-iz)$ is finite. \hfill\(\Box\)

For the Jagannathan Srinivasa Rao deformation $(p,q)=(1.5,0.5)$, $\phi_{1}/\phi_{2}=3\geq Q_{*}^{1/4}$, and the first positive zeros of $\CosRw$ are $2.0077971$, $22.851861$, $209.91871$ and $1893.4727$, each lying in the predicted interval, the sum of their inverse squares being $0.2499999965$, in accordance with $\phi_{2}/(\pn{1}\pn{2})=1/4$.

\subsection{Numerical illustrations}\label{subsec:numerics-zeros}

The value $Q_{*}=4.810582806$ was obtained by solving $\vartheta_{3}(0,Q^{-1/2})=2$ with fifty significant digits. All the zeros below were computed with fifty digits, by bisection on the negative axis and Newton's method in the complex plane, and their number in each disc was certified by the argument principle on $3000$ points of the circle.

\begin{table}[!htb]
\centering
\caption{Localization \eqref{eq:exp-localization} of the zeros $-\xi_{k}$ of $\eR$ for the Jagannathan Srinivasa Rao deformation $(p,q)=(5,0.5)$, $\phi_{1}=5>Q_{*}$.}\label{tab:exp-loc}
\begin{tabular}{lllll}
\hline
$k$ & $r_{k-1}$ & $\xi_{k}$ & $r_{k}$ & $\xi_{k}/\pn{k}$\\
\hline
$1$ & $0.4472136$ & $1.2875673$ & $2.3452079$ & $1.2875673$\\
$2$ & $2.3452079$ & $5.6061648$ & $12.354149$ & $1.0193027$\\
$3$ & $12.354149$ & $27.803142$ & $62.078831$ & $1.0019150$\\
$4$ & $62.078831$ & $138.90156$ & $310.54792$ & $1.0001913$\\
\hline
\end{tabular}
\end{table}

For $(p,q)=(5,0.5)$, Table \ref{tab:exp-loc} confirms \eqref{eq:exp-localization}: the $18$ zeros in $[-10^{12},0]$ are real and lie in the predicted annuli, the partial sums of \eqref{eq:exp-rayleigh} equal $0.9999999999997$ and $0.636363636363636=7/11$, and $\xi_{k}/\pn{k}-1$ decreases like $\fq^{k}$, $\fq=0.1$. For the Chakrabarti Jagannathan deformation $(p,q)=(0.8,0.5)$, $\phi_{1}+\phi_{2}=1.75<2$, and the argument principle gives $15$ zeros in the disc of radius $40$: the real zeros $-15.946337,\dots,-37.895606$ of Section \ref{sec:examples} and five pairs of nonreal zeros, of moduli $3.7020846$, $5.6405743$, $7.8377432$, $10.332530$, $13.176359$; with the $60$ real zeros in $[-10^{7},-14]$, $\sum_{k}z_{k}^{-2}=-0.142857142857=-1/7$, as given by \eqref{eq:exp-rayleigh}. In the Euler case $(p,q)=(1.5,1)$ the computed zeros are $-2\cdot1.5^{k}$, and in the Ramanujan case $(p,q)=(1.44,1.2)$ the $46$ zeros in $[-10^{8},0]$ are real, the argument principle on the first fifteen circles $|z|=r_{k}$ finding no nonreal zero; in both cases $\sum_{k}z_{k}^{-2}$ agrees with \eqref{eq:exp-rayleigh} ($0.2$ and $8/33$) to twelve digits.

\begin{figure}[!htb]
\centering
\includegraphics[width=0.88\linewidth]{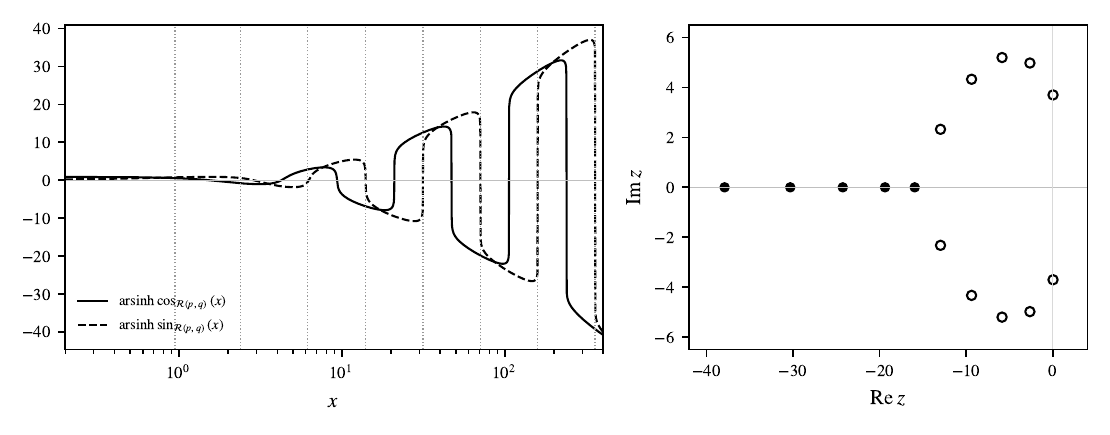}
\caption{Left: $\mathrm{arsinh}\,\cosR(x)$ and $\mathrm{arsinh}\,\sinR(x)$ for the Jagannathan Srinivasa Rao deformation $(p,q)=(1.5,0.5)$, on a logarithmic scale in $x$; the dotted lines are the bounds of \eqref{eq:trig-localization} for the zeros of $\cosR$, each interval containing exactly one zero. Right: the zeros of $\eR$ in the disc of radius $40$ for the Chakrabarti Jagannathan deformation $(p,q)=(0.8,0.5)$, where $\phi_{1}+\phi_{2}<2$: real zeros (filled) and nonreal zeros (open).}\label{fig:realzeros}
\end{figure}

For the trigonometric functions of the deformation $(p,q)=(1.5,0.5)$, Figure \ref{fig:realzeros} illustrates Theorem \ref{thm:trig-zeros}(iii): the $22$ positive zeros of $\cosR$ and of $\sinR$ in $(0,10^{8}]$ are simple and lie in the predicted intervals, the first ones being $x_{1}=1.5242122$, $y_{1}=2.8646855$, $x_{2}=4.2559365$, $y_{2}=6.2506322$, the sums \eqref{eq:trig-products} equal $0.5=1/(\pn{1}\pn{2})$ and $0.153846153846154=1/(\pn{2}\pn{3})$, and $x_{k+1}/x_{k}$ equals $2.2455$, $2.2495$, $2.24995$ for $k=3,4,5$, the limit being $\phi_{1}^{2}=2.25$. The zeros interlace, $x_{1}<y_{1}<x_{2}<y_{2}<\cdots$, for the $21$ pairs computed; we did not attempt a proof of this interlacing, which does not follow from \eqref{eq:trig-localization}, the intervals of $x_{k}$ and $y_{k}$ overlapping. The conditions of Theorem \ref{thm:trig-zeros} are not necessary: for $(p,q)=(0.8,0.5)$, $\phi_{1}=1.25<\sqrt2$, the first $11$ zeros of $C_{\mathcal{R}(p,q)}$ and of $S_{\mathcal{R}(p,q)}$ are real. On the other hand, for the Jagannathan Srinivasa Rao deformation $(p,q)=(1.1,0.1)$ one has $(\phi_{1}^{2}+\phi_{1}\phi_{2}+\phi_{2}^{2})(\phi_{1}^{2}+\phi_{2}^{2})=1.6226<2$, and among the $11$ zeros of $C_{\mathcal{R}(p,q)}$ in the disc of radius $r_{11}$ only $5$ are real, as predicted by Theorem \ref{thm:trig-zeros}(i); the same proportion is observed for $S_{\mathcal{R}(p,q)}$, in accordance with $\pn{4}\pn{5}=2.3577<2\pn{2}\pn{3}=3.192$.

\section{Growth regimes of the standard deformations and numerical illustrations}\label{sec:examples}

Table \ref{tab:regimes} collects the indices of Section \ref{subsec:indices} for five deformations satisfying $\Hphi$, computed from Proposition \ref{prop:indices}(iv), together with the gauge threshold $1/\alpha_{\mathcal{R}(p,q)}$ of Proposition \ref{prop:gauge}, the largest logarithmic growth rate $\rho^{*}_{\mathcal{R}(p,q)}=\sup_{k}k^{-1}\log\pn{k}$, which is the half aperture of the anisotropic sector for the scaling angle $\theta=1$, its asymptotic value $\rho^{\infty}_{\mathcal{R}(p,q)}=\log\phi_{1}$, and the critical orders $\omega_{\mathcal{R}(p,q)}(1)=\pi/(2\rho^{*}_{\mathcal{R}(p,q)})$ of $\Sth$ and $\omega^{\infty}_{\mathcal{R}(p,q)}(1)=\pi/(2\rho^{\infty}_{\mathcal{R}(p,q)})$ of the curved sector $\Om$ (Definition \ref{def:sector}). The supremum $\rho^{*}_{\mathcal{R}(p,q)}$ is attained at $k=2$ for the first and third lines and at $k=3$ for the second and fourth ones; for the $q-$Quesne line, $\rho_{\mathcal{R}(p,q)}(k)$ increases to $\log\phi_{1}$, because $c<1$.

\begin{table}[!htb]
\centering
\caption{Growth indices, gauge threshold and anisotropic sectors (scaling angle $\theta=1$); the index $\mathcal{R}(p,q)$ of $\beta_{\mathcal{R}(p,q)},\dots,\omega^{\infty}_{\mathcal{R}(p,q)}$ is omitted in the header.}\label{tab:regimes}
\small\setlength{\tabcolsep}{3pt}
\begin{tabular}{lllllllll}
\hline
Deformation & $\beta=\alpha$ & $\kappa=K$ & $\gamma$ & $1/\alpha$ & $\rho^{*}$ & $\rho^{\infty}$ & $\omega(1)$ & $\omega^{\infty}(1)$\\
\hline
JS, $(0.9,0.5)$ & $0.9$ & $0$ & $-0.05268$ & $1.1111$ & $0.16824$ & $-0.10536$ & $9.3369$ & $\infty$\\
JS, $(1.2,0.5)$ & $1.2$ & $\infty$ & $0.09116$ & $0.8333$ & $0.27618$ & $0.18232$ & $5.6875$ & $8.6155$\\
Jackson, $q=0.6$ & $1$ & $2.5$ & $0$ & $1$ & $0.23500$ & $0$ & $6.6842$ & $\infty$\\
CJ, $(0.8,0.5)$ & $1.25$ & $\infty$ & $0.11157$ & $0.8$ & $0.29699$ & $0.22314$ & $5.2890$ & $7.0394$\\
$q-$Quesne, $(0.5,0.8)$ & $1.25$ & $\infty$ & $0.11157$ & $0.8$ & $0.22314$ & $0.22314$ & $7.0394$ & $7.0394$\\
\hline
\end{tabular}
\end{table}

The first line is in the subexponential factorial regime: the factorial weighted spaces contain all the spaces $\mathcal{O}(D_{\rho})$, the deformed exponential series have radius $0$, the gauge saturates at $1/\phi_{1}>1$, and the curved sector $\Om$ is bounded, while $\Sth$ is not empty because $\pn{2}=1.4>1$. The third line is the Jackson regime, in which $\ER{\rho^{-}}=\mathcal{O}(D_{2.5\rho})$. The other lines are in the Gaussian regime: $\ER{r}$ consists of entire functions of order zero, the gauge saturates at $1/\phi_{1}<1$, and the curved sector has the positive asymptotic half aperture $\theta\log\phi_{1}$, with a critical order larger than that of $\Sth$; the gauge constant $c=p/q$ of the $q-$Quesne deformation only affects the gauge, $\|z\|_{\mathcal{R}(p,q)}=\max\{|z|/c,1/\phi_{1}\}$.

The convergence to the limits of Proposition \ref{prop:indices} was checked with sixty digits: for the Chakrabarti Jagannathan deformation, $\pn{n}^{1/n}=1.25090$ and $\log\pf{n}/n^{2}=0.11257$ at $n=400$, the limits being $1.25$ and $0.11157$; for the Jackson case $\pf{n}^{1/n}=2.48788$, the limit being $2.5$; for $(p,q)=(0.9,0.5)$, $\pf{n}^{1/n}=1.7\times10^{-9}$, which illustrates $K_{\mathcal{R}(p,q)}=0$.

\begin{figure}[!htb]
\centering
\includegraphics[width=0.78\linewidth]{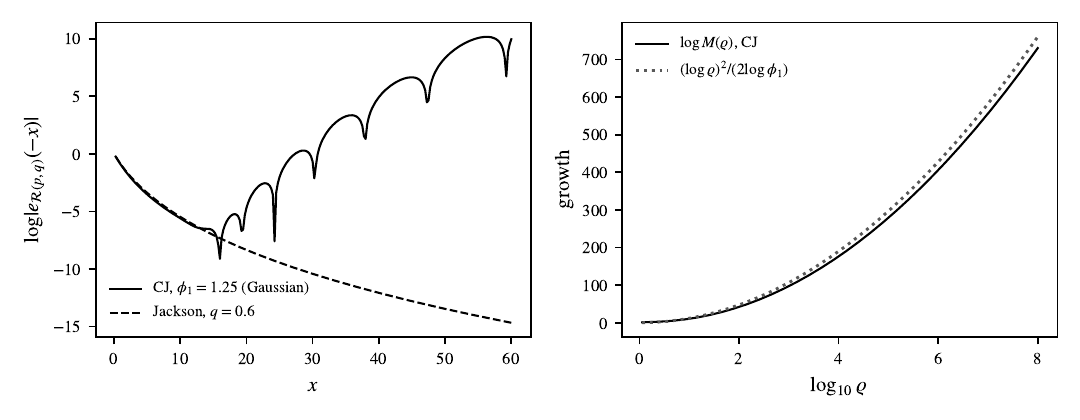}
\caption{Left: $x\mapsto\log|\eR(-x)|$ in the Gaussian regime (Chakrabarti Jagannathan deformation, $(p,q)=(0.8,0.5)$) and in the Jackson regime ($q=0.6$). Right: $\log\max_{|z|=\varrho}|\eR(z)|$ in the Gaussian case, compared with the leading term $(\log\varrho)^{2}/(4\gamma_{\mathcal{R}(p,q)})=(\log\varrho)^{2}/(2\log\phi_{1})$ of \eqref{eq:order-zero}.}\label{fig:gaussian}
\end{figure}

Figure \ref{fig:gaussian} illustrates Theorem \ref{thm:rigidity}, Corollary \ref{cor:exp-gaussian} and Remark \ref{rem:counterexample}: in the Gaussian regime $\eR(-x)$ oscillates with growing amplitude ($\eR(-60)=-20921.1$), while in the Jackson regime $\eR(-x)=1/(-(1-q)x;q)_{\infty}$ decreases to $0$. The quotient of $\log\max_{|z|=\varrho}|\eR(z)|$ by $(\log\varrho)^{2}/(2\log\phi_{1})$ equals $0.928$ at $\varrho=10^{4}$ and $0.960$ at $\varrho=10^{8}$, in agreement with \eqref{eq:order-zero}.

The real zeros of $\eR(-x)$ in the Chakrabarti Jagannathan case were computed by bisection with eighty digits on $[1,2.9\times10^{4}]$. There are $34$ of them, the first ones being $15.9463366$, $19.3766132$, $24.2539558$, $30.3164586$, $37.8956063$, and the quotients of consecutive zeros, $1.2151138$, $1.2517129$, $1.2499593$, $1.2500011$, converge very rapidly to $\phi_{1}=1.25$. Hence the number of real zeros in $[-\varrho,0]$ is asymptotic to $\log\varrho/\log\phi_{1}$, half of the upper bound of Theorem \ref{thm:zeros} for $\gamma'\to\gamma_{\mathcal{R}(p,q)}$. Besides these real zeros, $\eR$ has five pairs of nonreal zeros of modulus less than $14$, in accordance with Theorem \ref{thm:exp-zeros}(i), since $\phi_{1}+\phi_{2}=1.75<2$ (Section \ref{subsec:numerics-zeros}).

\begin{figure}[!htb]
\centering
\includegraphics[width=0.78\linewidth]{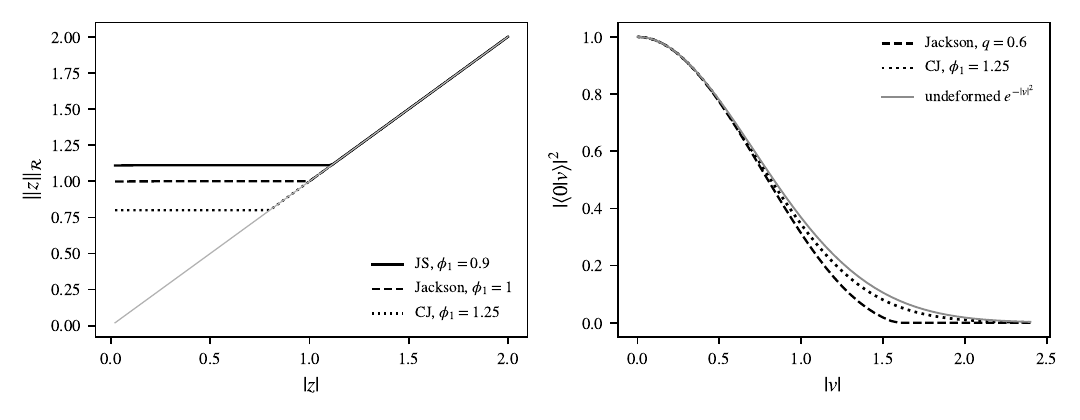}
\caption{Left: the deformed gauge $\|z\|_{\mathcal{R}(p,q)}$ as a function of $|z|$, computed from \eqref{eq:gauge} with $n\leq3000$, for the Jagannathan Srinivasa Rao deformation $(0.9,0.5)$, the Jackson case $q=0.6$ and the Chakrabarti Jagannathan deformation $(0.8,0.5)$; the plateaus are at $1/\phi_{1}$, in accordance with Theorem \ref{thm:gauge-Hphi}. Right: the vacuum overlap $|\langle0|v\rangle|^{2}=1/\eR(|v|^{2})$ of the deformed coherent states of Theorem \ref{thm:coherent}, compared with the undeformed Gaussian.}\label{fig:gauge-coherent}
\end{figure}

Figure \ref{fig:gauge-coherent} illustrates Theorems \ref{thm:gauge-Hphi} and \ref{thm:coherent}. In the Jackson case the vacuum overlap vanishes at $|v|=\sqrt{\kappa_{\mathcal{R}(p,q)}}=\sqrt{2.5}$, where $\eR(|v|^{2})$ has its first pole and beyond which the coherent states do not exist, while in the Gaussian regime the coherent states exist for every $v$ and the overlap decays more slowly than the undeformed Gaussian, since $\pf{n}\geq n!$ eventually.

The other results were also checked numerically. The inversion formula \eqref{eq:borel-inversion} and the Laplace type identity \eqref{eq:laplace} hold to $10^{-30}$ for test functions with $60$ coefficients, and the identities of Theorems \ref{thm:constant-coefficients} and \ref{thm:translations}(ii) to the working precision. On $200$ random polynomials of degree $6$, the largest quotient $\max_{|z|=1}|\partial_{\mathcal{R}(p,q)}P|/\max_{|z|=1}|P|$ is $1.547$ for the deformation $(0.9,0.5)$ and $4.168$ for $(0.8,0.5)$, against the bounds $6$ and $7.753$ of Section \ref{subsec:mean-value}; for $(\phi_{1},\phi_{2})=(3,1.5)$ it is $257.5$, the sharp bound $\pn{6}=478.4$ being attained by $z^{6}$. For $f(z)=(z+0.3)/(1+0.3z)$ and $(0.9,0.5)$, $|\partial_{\mathcal{R}(p,q)}f(z)|=0.8382$, $0.7458$, $0.6682$ at $z=0.2,0.5,0.8$, below the bounds $1.0206$, $1.1464$, $1.5125$ of \eqref{eq:schwarz-pick}.

\subsection{A dictionary for the standard deformations}\label{subsec:dictionary}

We summarize the objects of the paper for each deformation of Section \ref{subsec:indices}.

\paragraph{Jagannathan Srinivasa Rao deformation, $0<q<p$.}
Here $\partial_{\mathcal{R}(p,q)}f(z)=(f(pz)-f(qz))/((p-q)z)$ maps $\mathcal{O}(D_{\rho})$ into $\mathcal{O}(D_{\rho/p})$ with the bound \eqref{eq:JS-bound}, the deformed gauge is $\max\{|z|,1/p\}$, the deformed integral operator is the $(p,q)-$Volterra operator $\int_{0}^{z}f\,d_{p,q}t$, the deformed exponentials $\eRw$ and $\ERw$ are the two $(p,q)-$exponentials of \cite{JagannathanSrinivasa, JaganSridh}, and the deformed exponential series is $\eR(z)=\sum_{n}z^{n}/[n]_{p,q}!$. For $p<1$ the series $\eR$ has radius zero, the annihilation operator of the Fock space is compact and quasinilpotent and the curved sectors are bounded; $p=1$ is the Jackson case; for $p>1$, $\gamma_{\mathcal{R}(p,q)}=\frac12\log p$, the asymptotic critical order of the curved sector is $\pi/(2\theta\log p)$, the zeros of $\eR$, $\cosR$ and $\sinR$ are described by Theorems \ref{thm:exp-zeros} and \ref{thm:trig-zeros}, and those of $\ERw$, $\CosRw$ and $\SinRw$ by Corollaries \ref{cor:E-zeros} and \ref{cor:Cos-zeros}.

\paragraph{Jackson case, $p=1$, $0<q<1$.}
$\partial_{\mathcal{R}(p,q)}$ is a continuous endomorphism of every $\mathcal{O}(D_{\rho})$, $\ER{\rho^{-}}=\mathcal{O}(D_{\rho/(1-q)})$, the deformed Borel transform is the $q-$Laplace transform \eqref{eq:laplace}, the annihilation operator of the $q-$oscillator Fock space is bounded with spectrum the closed disc of radius $(1-q)^{-1/2}$, and rigidity fails, as shown by $e^{-z}$.

\paragraph{Chakrabarti Jagannathan deformation, $pq<1$.}
Here $\phi_{1}=1/p$, $\phi_{2}=q$. For $p<1$ the deformation is in the Gaussian regime: $\partial_{\mathcal{R}(p,q)}$ maps $\mathcal{O}(D_{\rho})$ into $\mathcal{O}(D_{p\rho})$, the gauge saturates at $p$, every element of $\ER{r}$ is entire of logarithmic type at most $1/(2\log(1/p))$ (Theorem \ref{thm:log-type}), the deformed exponential series has real zeros asymptotically in geometric progression of ratio $1/p$ and, when $p^{-1}+q\leq2$, also nonreal zeros, and the coherent states exist for every $w\in\mathbb{C}$.

\paragraph{Generalized $q-$Quesne and Hounkonnou Ngompe deformations, $pq<1$.}
For the first one, $\Hphi$ holds with $(c,\phi_{1},\phi_{2})=(p/q,1/q,p)$, and the indices are those of the Chakrabarti Jagannathan deformation with $p$ replaced by $q$, the gauge becoming $\max\{q|z|/p,q\}$. For the second one, $(c,\phi_{1},\phi_{2})=(h\,p^{1-\mu}q^{\nu-1},p^{-\mu}q^{\nu-1},p^{1-\mu}q^{\nu})$: the parameters $\mu$, $\nu$ move the deformation between the subexponential factorial and the Gaussian regimes, through the Jackson regime on the curve $p^{-\mu}q^{\nu-1}=1$, and the asymptotic aperture of the curved sectors depends on them only through $\rho^{\infty}_{\mathcal{R}(p,q)}=\log\phi_{1}=(\nu-1)\log q-\mu\log p$.

\section{Concluding remarks}\label{sec:conclusion}

We have studied the deformed derivative of the $\mathcal{R}(p,q)-$deformed quantum algebras on spaces of holomorphic functions, with every statement formulated through explicit growth indices of the deformed numbers and factorials. The index $\beta_{\mathcal{R}(p,q)}$ governs the optimal loss of radius of $\partial_{\mathcal{R}(p,q)}$ and, through $\log\beta_{\mathcal{R}(p,q)}=\limsup_{k}k^{-1}\log\pn{k}$, the asymptotic aperture of the curved sectors, while the anisotropic sectors are governed by the largest logarithmic growth rate $\sup_{k}k^{-1}\log\pn{k}$, the dominant root $\phi_{1}$ fixes the minimal resolution radius $1/\phi_{1}$ of the deformed gauge, the factorial weighted scale carries an explicit spectral theory and a Gelfand theory realized by the deformed Borel transform, and the Gaussian growth of the deformed factorials produces order zero, rigidity and, for large $\phi_{1}$, real and explicitly localized zeros of the deformed exponential and trigonometric functions. The numerical illustrations confirm the three regimes of the standard deformations.

Several questions remain open: the exact region of the parameters $(\phi_{1},\phi_{2})$ for which all the zeros of $\eR$ are real, which lies between the necessary condition $\phi_{1}+\phi_{2}>2$ and the sufficient conditions of Theorem \ref{thm:exp-zeros}; the interlacing of the zeros of the deformed cosine and sine observed in Section \ref{subsec:numerics-zeros}; a deformed potential theory in which the dilation orbits replace the Euclidean discs below the resolution radius $1/\phi_{1}$; and applications to deformed evolution equations $\partial_{\mathcal{R}(p,q)}f=\lambda f+g$ (Corollary \ref{cor:cauchy-problem}) and to the fractional $\mathcal{R}(p,q)-$operators of \cite{SoontharanonSitthiwirattham2020}.

\appendix

\section{Auxiliary facts}\label{app:auxiliary}

\medskip\noindent\textbf{The Beurling algebra $\mathcal{A}^{+}(r)$.}\label{app:beurling}

For $r>0$, $\mathcal{A}^{+}(r)$ is the space of power series $u=\sum_{n\geq0}c_{n}\lambda^{n}$ with $\|u\|=\sum|c_{n}|r^{n}<\infty$ and the Cauchy product; it is isometric to $\ell^{1}$ through $c_{n}\mapsto c_{n}r^{n}$, and $\|uv\|\leq\|u\|\,\|v\|$ by the triangle inequality. Its elements are continuous on $\overline{D_{r}}$ and holomorphic on $D_{r}$, and the polynomials are dense. A character $\chi$ satisfies $|\chi(u)|\leq\|u\|$ \cite{Kaniuth2009}, so that $\lambda_{0}=\chi(\lambda)$ satisfies $|\lambda_{0}|^{n}=|\chi(\lambda^{n})|\leq r^{n}$, and $\chi(u)=u(\lambda_{0})$ by density; conversely, every evaluation at $|\lambda_{0}|\leq r$ is a character. The Gelfand transform of $u$ is therefore its restriction to $\overline{D_{r}}$, and $\lim_{k}\|u^{k}\|^{1/k}=\max_{\overline{D_{r}}}|u|$.

\medskip\noindent\textbf{Weighted shifts.}\label{app:shifts}

Let $(w_{n})_{n\geq1}$ be positive numbers. The maximal backward weighted shift $T(c_{n})_{n}=(w_{n+1}c_{n+1})_{n}$ on $\ell^{2}$ is closed, since convergence in $\ell^{2}$ implies coordinatewise convergence, and so is the forward shift $S(c_{n})=(0,w_{1}c_{0},w_{2}c_{1},\dots)$; $T^{*}=S$ by Theorem \ref{thm:fock-representation}. If $\sup w_{n}<\infty$, then $\|T^{k}\|=\sup_{n}w_{n+1}\cdots w_{n+k}$; if $w_{n}\to0$, $T$ is compact and quasinilpotent.

\medskip\noindent\textbf{K\"othe echelon spaces.}\label{app:kothe}

For a matrix $A=(a_{n,k})$ of positive numbers, increasing in $k$, the K\"othe echelon space $\lambda^{1}(A)$ is the space of sequences $(\xi_{n})$ with $\sum_{n}|\xi_{n}|a_{n,k}<\infty$ for every $k$, with the seminorms given by these sums; it is a Fr\'echet space, and it is nuclear if and only if for every $k$ there is $l$ with $\sum_{n}a_{n,k}/a_{n,l}<\infty$ (Grothendieck Pietsch criterion \cite[Theorem 28.15]{MeiseVogt1997}, \cite{Pietsch1972}). For $a_{n,k}=\pf{n}r_{k}^{n}$ with $r_{k}\uparrow\rho$, the map $f\mapsto(a_{n})_{n}$ identifies $\ER{\rho^{-}}$ with $\lambda^{1}(A)$, and $a_{n,k}/a_{n,l}=(r_{k}/r_{l})^{n}$ is summable for $l>k$, independently of the deformation.

\end{document}